\documentclass[11.5pt,twoside,openany]{article}
\usepackage{amssymb,amsthm}
\usepackage[tbtags]{amsmath}
\usepackage{cite}
\usepackage{indentfirst}
\usepackage{cases}
\usepackage{graphicx}
\usepackage{subfigure}
\usepackage{lineno}
\usepackage{enumerate}
\usepackage{threeparttable}
\usepackage{multirow}
\usepackage{booktabs}
\usepackage{upgreek}  
\usepackage{bm}
\usepackage{fancyhdr}
\usepackage{ifthen}
\usepackage{lipsum}
\allowdisplaybreaks[4]
\usepackage[misc]{ifsym}
\usepackage{caption}
\usepackage{float}
\usepackage{appendix}
\usepackage[format=hang]{caption}
\usepackage{algorithm,algorithmic}

\usepackage{charter}  
\usepackage{indentfirst}  
\usepackage{soul}
\usepackage{authblk}

\newcommand{\Rmnum}[1]{\uppercase\expandafter{\romannumeral #1}} 

\usepackage[colorlinks,
linkcolor=blue,
anchorcolor=blue,
citecolor=red]{hyperref}

\numberwithin{equation}{section}

\newtheorem{Lemma}{Lemma}[section]
\newtheorem{Theorem}{Theorem}[section]
\newtheorem{Remark}{Remark}[section]

\newcounter{saveeqn}

\makeatletter
\def\@maketitle{%
	\newpage
	\null
	\vskip 2em%
	\begin{center}%
		\let \footnote \thanks
		{\LARGE \@title \par}%
		\vskip 1.5em%
		{\large
			\lineskip .5em%
			\begin{tabular}[t]{c}%
				\@author
			\end{tabular}\par}%
	\end{center}%
	\par
	\vskip 1.5em}
\makeatother

\makeatletter
\evensidemargin
\oddsidemargin
\makeatother

\title{\bf  Convergence analysis of decoupled mixed FEM for the Cahn-Hilliard-Navier-Stokes equations\thanks{This research is supported by the National Natural Science Foundation of China (No.11971337) and Natural Science Foundation of Sichuan Province (No. 2025ZNSFSC0070).}}
\author{Haijun Gao$^{1}$, Xi Li$^{2}$, and Minfu Feng\footnote{Corresponding author. e-mail:~gaohijun@163.com,~lixi@cdut.edu.cn,~fmf@scu.edu.cn.}}
\affil[1]{\small School of Mathematics, Sichuan University, Chengdu, Sichuan 610064, China}
\affil[2]{School of Mathematics Sciences, Chengdu University of Technology, Chengdu, Sichuan 610059, China}

\begin{document}
	\maketitle
	\newcommand\blfootnote[1]{%
		\begingroup
		\renewcommand\thefootnote{}\footnote{#1}%
		\par\setlength\parindent{2em}
		\endgroup
	}
	\captionsetup[figure]{labelfont={bf},labelformat={default},labelsep=period,name={Figure}}
	\captionsetup[table]{labelfont={bf},labelformat={default},labelsep=period,name={Table}}
	\begin{abstract}
		We develop a decoupled, first-order, fully discrete, energy-stable scheme for the Cahn-Hilliard-Navier-Stokes equations. This scheme calculates the Cahn-Hilliard and Navier-Stokes equations separately, thus effectively decoupling the entire system. To further separate the velocity and pressure components in the Navier-Stokes equations, we use the pressure-correction projection method. We demonstrate that the scheme is primitively energy stable and prove the optimal $L^2$ error estimate of the fully discrete scheme in the $P_r\times P_r\times P_r\times P_{r-1}$ finite element spaces, where the phase field, chemical potential, velocity and pressure satisfy the first-order accuracy in time and the $\left(r+1,r+1,r+1,r\right)th$-order accuracy in space, respectively.  Furthermore, numerical experiments are conducted to support these theoretical findings. Notably, compared to other numerical schemes, our algorithm is more time-efficient and numerically shown to be unconditionally stable.
		\\	
		
	\end{abstract}
	\baselineskip 15pt
	\parskip 10pt
	\setcounter{page}{1}
	\vspace{-0.5cm}
	\section{Introduction}
	The Cahn-Hilliard-Navier-Stokes (CHNS) model features a nonlinear interaction between the incompressible Navier-Stokes (NS) equations \cite{Temam_Roger_Navier_Stokes_equations_Theory_and_numerical_analysis} and the Cahn-Hilliard (CH) equations \cite{1958_Cahn_Hilliard_Free_Energy_of_a_Nonuniform_System_I_Interfacial_Free_Energy}. This model captures the interfacial dynamics of two-phase, incompressible, and macroscopically immiscible Newtonian fluids that exhibit matched densities \cite{2007_Kay_David_and_Welford_Richard_Efficient_numerical_solution_of_Cahn_Hilliard_Navier_Stokes_fluids_in_2D,
		2008_Kay_David_and_Welford_Richard_Finite_element_approximation_of_a_Cahn_Hilliard_Navier_Stokes_system} as follows:
	\begin{subequations}\label{eq_chns_equations}
		\begin{align}
			\label{eq_phi_con}
			\frac{\partial \phi }{\partial t}+(\mathbf{u}\cdot\nabla)\phi- M\Delta \mu=0,\quad \text{in} \quad\Omega\times (0,T],&\\
			\label{eq_mu_con}
			\mu+\lambda\Delta \phi-\lambda F'(\phi)=0,\quad \text{in} \quad\Omega\times (0,T],&\\
			\label{eq_ns_con}
			\frac{\partial \mathbf{u}}{\partial t}+(\mathbf{u}\cdot\nabla)\mathbf{u}-\nu\Delta\mathbf{u}+\nabla p-\mu\nabla\phi=0,\quad\text{in} \quad\Omega\times (0,T],&\\
			\label{eq_incompressi}
			\nabla\cdot\mathbf{u}=0,\quad\text{in} \quad\Omega\times (0,T],&
		\end{align}
	\end{subequations}
	with the following boundary and initial conditions of \eqref{eq_chns_equations}:
	\begin{subequations}
		\label{eq_boundary_initial_conditions_equations}
		\begin{align}
			\label{eq_boundary_conditions_equation}
			\frac{\partial \phi}{\partial \mathbf{n}}=\frac{\partial\mu}{\partial \mathbf{n}}=0,\quad\mathbf{u}=0,&\quad\text{on}\quad \partial\Omega\times (0,T],\\
			\label{eq_initial_conditions_equation}
			\phi(\mathbf{x},0)=\phi^0,\quad\mathbf{u}(\mathbf{x},0)=\mathbf{u}^0,&\quad\text{in}\quad \Omega.
		\end{align}
	\end{subequations}
	Here $\Omega\subset\mathbb{R}^d,~(d=2,3)$ is a bounded convex polygonal or polyhedral domain, and for $t\in (0,T]$,
	and $F(\phi)=\frac{1}{4\epsilon^2}\left(\phi^2-1\right)^2$.
	In the model, $\phi$ represents the phase field variable, and $\mu$ denotes the chemical potential, and the $\mathbf{u}$ and $p$ are the velocity and pressure of the fluid, respectively.
	$\epsilon$ is the interfacial width between the two phases field, and the parameters $M$, $\lambda$ and $\nu$ represent the mobility constant, 
	the mixing coefficient and the fluid viscosity, respectively. 
	It is well know that the energy of CHNS model is defined by 
	\begin{equation}
		E(\phi,\mathbf{u})=\int_{\Omega}\left(\frac{1}{2}|\mathbf{u}|^2+\frac{\lambda}{2}|\nabla\phi|^2+\lambda F(\phi)\right) dx.
	\end{equation}
	The system \eqref{eq_chns_equations}-\eqref{eq_boundary_initial_conditions_equations} satisfies the following energy dissipation law at $\forall t\in (0,T]$,
	\begin{equation}
		\frac{\partial E(\phi(t),\mathbf{u}(t))}{\partial t}=-M\|\nabla \mu\|^2-\nu\|\nabla\mathbf{u}\|^2.
	\end{equation}

	Over the past decade or so, one of the primary challenges for the CH  equations have been the efficient treatment of nonlinear terms to ensure that the resulting discretized system is solved efficiently while maintaining energy stability. Some of the techniques for dealing with the nonlinear terms in the numerical format of the CH equation  include the convex splitting \cite{1998_Eyre_David_J_Unconditionally_gradient_stable_time_marching_the_Cahn_Hilliard_equation}, stabilized semi-implicit methods \cite{2011_ShenJie_Energy_stable_schemes_for_Cahn_Hilliard_phase_field_model_of_two_phase_incompressible_flows}, invariant energy quadratization (IEQ) \cite{2017_YangXiaofeng_Numerical_approximations_for_the_molecular_beam_epitaxial_growth_model_based_on_the_invariant_energy_quadratization_method}, and the scalar auxiliary variable (SAV) approach \cite{2018_Shenjie_Xujie_SAV,2018_Shenjie_Xujie_CAEAFTSAVSYGF}. 
	On the other hand, the main challenges in solving the NS equations lie in the treatment of nonlinear terms and the coupling between velocity and pressure under incompressible conditions. For the coupling between velocity and pressure, a widely adopted strategy is the projection method  \cite{2006_Guermond_ShenJie_An_overview_of_projection_methods_for_incompressible_flows}, originally proposed by Chorin and Temam \cite{1968_Chorin_Alexandre_Numerical_solution_of_the_Navier_Stokes_equations}. The nonlinear terms in the NS equations can be handled either in an fully implicit, semi-implicit, or through SAV full explicatization. Wang et al. \cite{2017_Diegel_Convergence_analysis_and_error_estimates_for_a_second_order_accurate_finite_element_method_for_the_Cahn_Hilliard_Navier_Stokes_system} presented a novel second order in time mixed finite element scheme for the Cahn–Hilliard–Navier–Stokes equations with matched densities, and they \cite{2024_Wangcheng_Convergence_analysis_of_a_temporally_second_order_accurate_finite_element_scheme_for_the_Cahn_Hilliard_magnetohydrodynamics_system_of_equations} proposed and analyze a time second-order accurate numerical
	scheme for the Cahn–Hilliard-Magnetohydrodynamics equations.
	As a result, based on the several challenges mentioned earlier, a large number of scholars have worked on proposing energy-stable and efficient numerical schemes for CHNS
	 \cite{2015_HanDaozhi_A_second_order_in_time_uniquely_solvable_unconditionally_stable_numerical_scheme_for_Cahn_Hilliard_Navier_Stokes_equation,2015_ShenJie_Decoupled_energy_stable_schemes_for_phase_field_models_of_two_phase_incompressible_flows,2018_CaiYongyong_Error_estimates_for_a_fully_discretized_scheme_to_a_Cahn_Hilliard_phase_field_model_for_two_phase_incompressible_flows,2018_GaoYali_Decoupled_linear_and_energy_stable_finite_element_method_for_the_Cahn_Hilliard_Navier_Stokes_Darcy_phase_field_model,2019_LinLianlei_Numerical_approximation_of_incompressible_Navier_Stokes_equations_based_on_an_auxiliary_energy_variable,2020_JiaHongen_A_novel_linear_unconditional_energy_stable_scheme_for_the_incompressible_Cahn_Hilliard_Navier_Stokes_phase_field_model,2020_YangXiaofeng_Error_Analysis_of_a_Decoupled_Linear_Stabilization_Scheme_for_the_Cahn_Hilliard_Model_of_Two_Phase_Incompressible_Flows,2021_Yang_Xiaofeng_Decoupled_linear_and_unconditionally_energy_stable_fully_discrete_finite_element_numerical_scheme_for_a_two_phase_ferrohydrodynamics_model,2022JieShen_LiXiaoliMSAVCHNStwo_phase_incompressible_flows,2017_Diegel_Convergence_analysis_and_error_estimates_for_a_second_order_accurate_finite_element_method_for_the_Cahn_Hilliard_Navier_Stokes_system,2024_ChenWenbin_Convergence_analysis_of_a_second_order_numerical_scheme_for_the_Flory_Huggins_Cahn_Hilliard_Navier_Stokes_system}.
	 
	Recently, Xu et al. \cite{2020_YangXiaofeng_Error_Analysis_of_a_Decoupled_Linear_Stabilization_Scheme_for_the_Cahn_Hilliard_Model_of_Two_Phase_Incompressible_Flows} proposed a first-order in time, linear, fully decoupled, and energy stable scheme for the CHNS model. They introduced a velocity intermediate quantity in the discretization of the CH equation, thus explicitly dealing with the coupling term between the velocity and phase functions.
	The introduction of this intermediate variable makes error estimation involved, especially if one also wants to obtain optimal $L^2$ error estimates for the spatial variables.
	Then, Wen et al. \cite{2022_Wenjuan_Semi_implicit_unconditionally_energy_stable_stabilized_finite_element_method_based_on_multiscale_enrichment_for_the_Cahn_Hilliard_Navier_Stokes_phase_field_model} proposed a novel fully discrete stabilized finite element method using the lowest equal-order finite elements, and they analyzed the optimal $H^1$ error estimates for the phase field function, chemical potential and velocity.
	Cai et al.  \cite{2023_CaiWentao_Optimal_L2_error_estimates_of_unconditionally_stable_finite_element_schemes_for_the_Cahn_Hilliard_Navier_Stokes_system}
	first presented optimal $L^2$ error estimates of  the convex splitting finite element method for the CHNS system. However, these schemes are a coupled system, whose numerical results in relatively long computation time.
	Soon after, Yang et al. \cite{2024_YiNianyu_Convergence_analysis_of_a_decoupled_pressure_correction_SAV_FEM_for_the_Cahn_Hilliard_Navier_Stokes_model} proposed a step-by-step decoupling method that utilizes the SAV approach to split the CHNS model into CH and NS components, enabling the solution of these two equations separately, and also analyzed the optimal $L^2$ error estimates for the numerical solution, but their scheme did not achieve original energy stabilization.
	
	To the best of our knowledge, there is no literature on the optimal $L^2$ error estimates for the fully decoupled, original energy stable numerical scheme of CHNS model.  To achieve this purpose, we firstly decouple the CH and NS equations by handling the convective term implicitly/explicitly in the CH equation, thus resulting in a fully decoupled and original energy stable semi-discrete scheme. Secondly, we discretize the spatial variables using the mixed finite element method to obtain a fully decoupled, fully discretized scheme with original energy stable, which is analyzed to obtain the optimal $L^2$ error estimate. We notice that Kay et al.  \cite{2008_Kay_David_and_Welford_Richard_Finite_element_approximation_of_a_Cahn_Hilliard_Navier_Stokes_system} also constructed a fully decoupled numerical scheme with original energy stabilization, they, however, did not provide the optimal $L^2$ error estimates. Additionally, Cai et al. \cite{2023_CaiWentao_Optimal_L2_error_estimates_of_unconditionally_stable_finite_element_schemes_for_the_Cahn_Hilliard_Navier_Stokes_system} analytically derived the optimal $L^2$ error estimates for the scheme with original energy stabilization, but the coupled nature of their scheme made computation expensive. We emphasize that despite a mild time-step restriction is imposed in our energy stability, our numerical experiment demonstrates that our fully decoupled scheme is unconditionally energy stable.
	
	This manuscript is organized as follows: The Section \ref{section_preliminaries} introduces fundamental notations and lemmas. In Section \ref{section_the_numerical_schemes}, we propose a fully decoupled, energy-stable scheme utilizing fully discrete mixed finite element methods based on pressure-correction and convex splitting techniques for the CHNS model; we then prove the energy stability of our scheme. In Section \ref{section_error_analysis}, we provide the corresponding error estimates and establish optimal $L^2$ error estimates for the fully discrete scheme. Numerical experiments are conducted in Section \ref{section_numerical_experiments} to validate the theoretical results of our scheme.
	
	\section{The mathematical prelimilaries}\label{section_preliminaries}
		In the section, we introduce some mathematical notations. 
	For any integer $k\geq0$ and $1\leq p\leq \infty$, let $W^{k,p}(\Omega)$ 
	be the Sobolev space,
	we denote $H^k(\Omega)=W^{k,2}(\Omega)$ and $L^p(\Omega)=W^{0,p}(\Omega)$.
	The norm of $L^2$ denote by $\|\cdot\|$ and the closure of $C_0^{\infty}(\Omega)$
	in $W^{k,p}$ space denotes by $W_0^{k,p}$ and $H_0^k(\Omega)=W_0^{k,p}$.
	Let $\mathbf{L}^p(\Omega)=[L^p(\Omega)]^d$, $\mathbf{H}_0^1(\Omega)=[H_0^1(\Omega)]^d$,  $\mathbf{W}^{k,p}(\Omega)=[W^{k,p}(\Omega)]^d$ the corresponding vector valued Sobolev spaces
	with the norms $\|\cdot\|_{H^k}$ and $\|\cdot\|_{L^p}$. 
	For $\boldsymbol{v}\in \mathbf{H}^1(\Omega)$ and $\phi, ~w \in H^1(\Omega)$, we define 
	the trilinear form $b(\cdot,\cdot,\cdot)$ by
	\begin{equation}
		\label{eq_trilinear_forms}
		b(\phi,\boldsymbol{v},w)=\left(\nabla\phi \cdot\boldsymbol{v},w\right),
	\end{equation}
	and for $\boldsymbol{u},\boldsymbol{v},\boldsymbol{w}\in \mathbf{H}^1(\Omega)$, we define the trilinear form
 	$\boldsymbol{B}(\cdot,\cdot,\cdot)$ by
 	\begin{equation}
 		\label{eq_trilinear_form_b_puls}
 		\boldsymbol{B}(\boldsymbol{u},\boldsymbol{v},\boldsymbol{w})
 		=\frac{1}{2}\left(\left(\boldsymbol{u}\cdot\nabla\right)\boldsymbol{v},\boldsymbol{w}\right)-
 		\frac{1}{2}\left((\boldsymbol{u}\cdot\nabla)\boldsymbol{w},\boldsymbol{v}\right).
 	\end{equation}
	It is well-know result that 
	\begin{equation}
		\label{eq_trilinear_forms_bound}
		\bigg|b(\phi,\boldsymbol{v},w)\bigg|=\bigg|\int_{\Omega}w\nabla\phi\cdot\boldsymbol{v}dx\bigg|
		\leq \|\phi\boldsymbol{v}\|\|\nabla w\|
		\leq \|\phi\|_{L^4}\|\boldsymbol{v}\|_{L^4}\|\nabla w\|
		\leq \|\phi\|_{H^1}\|\boldsymbol{v}\|_{H^1}\|\nabla w\|.
	\end{equation}
	We next introduce the Ritz and Stokes projections. 
	Let $\mathcal{T}_h$ denotes a quasi-uniform partition of $\Omega$ into triangles $\mathcal{K}_j,~j=1,\cdots,M$, in $\mathbb{R}^2$ with mesh size $h=\max_{1\leq j\leq M}\{\dim\mathcal{K}_j\}$. For a given positive integer $r \geq 2$, we set the phase field-velocity-pressure finite element spaces 
	\begin{equation}
		\begin{aligned}
			& S_h^r=\left\{v_h \in C(\Omega):\left.v_h\right|_{\mathcal{K}_j} \in P_r\left(\mathcal{K}_j\right), \forall \mathcal{K}_j \in \mathcal{T}_h\right\}, \\
			& \mathring{S}_h^{r-1}=S_h^{r-1} \cap L_0^2(\Omega), \\
			& \mathbf{X}_h^{r}=\left\{\mathbf{v}_h \in \mathbf{H}_0^1(\Omega)^d:\left.\mathbf{v}_h\right|_{\mathcal{K}_j} \in \mathbf{P}_{r}\left(\mathcal{K}_j\right)^d, \forall \mathcal{K}_j \in \mathcal{T}_h\right\},
		\end{aligned}
	\end{equation}
	where $P_r\left(\mathcal{K}_j\right)$ is the space of polynomials of degree $r$ on $\mathcal{K}_j$. 
	For some constant $C>0$,
	it is well-known that both the Taylor-Hood elements satisfy the discrete inf-sup condition \cite{1984_Brezzi_FortinA_stable_finite_element_for_the_Stokes_equations}
	\begin{equation}\label{eq_inf_sup_condition}
		\left\|q_h\right\|_{L^2} \leq C \sup _{\mathbf{0} \neq \mathbf{v}_h \in \mathbf{X}_h^{r} } 
		\frac{\left(q_h, \nabla \cdot \mathbf{v}_h\right)}{\left\|\nabla\mathbf{v}_h\right\|}, \quad \forall q_h \in \mathring{S}_h^{r-1}.
	\end{equation}
	For the simplicity of notations and $r\geq 2$, we define
	\begin{equation}
		\mathbf{\mathcal{X}}_h^r= S_h^r \times S_h^r \times \mathbf{X}_h^{r}  \times \mathring{S}_h^{r-1},
	\end{equation}
	and denote $\boldsymbol{g}^{n+1}=\boldsymbol{g}\left(x, t^{n+1}\right)$, 
	and
	\begin{equation}\label{eq_d_t}
		\delta_\tau\boldsymbol{g}^{n+1}=\frac{\boldsymbol{g}^{n+1}-\boldsymbol{g}^n}{\tau}.
	\end{equation}
	\begin{Remark}
		In an abuse of notation, we use  $\boldsymbol{g}^{n+1}$  hereafter to denote the value of the exact solution $\boldsymbol{g}$ at $t^{n+1}$. 
	\end{Remark}
	Next, 
	we denote $R_h:H^1(\Omega)\rightarrow S_h^{r}$ 
	as the classic Ritz projection \cite{1973_Wheeler_Mary_Fanett_A_priori_L2_error_estimates_for_Galerkin_approximations_to_parabolic_partial_differential_equations},
	\begin{equation}\label{eq_Ritz_projection}
		\left(\nabla\left(\psi- R_h\psi\right),\nabla\varphi_h\right)=0,\quad\forall~\psi\in H^1(\Omega),~\varphi_h\in S_h^{r}.
	\end{equation}
	Now, for $s\in[2,\infty]$ and $n=1,2,\cdots,N$,  from the finite element approximation theory \cite{2008_Brenner_Susanne_C_The_mathematical_theory_of_finite_element_methods}, 
	it holds that
	\begin{flalign}
		\label{eq_psi_Rhpsi_Ls_norm_inequation}
		&\|\psi-R_h\psi\|_{L^s}+h\|\psi-R_h\psi\|_{W^{1,s}}\leq Ch^{r+1}\|\psi\|_{W^{r+1,s}},\\
		\label{eq_psi_Rhpsi_Hneg1_norm_inequation}
		&\|\psi-R_h\psi\|_{H^{-1}}\leq Ch^{r+2}\|\psi\|_{H^{r+1}},\\
		\label{eq_Dtau_psi_Rhpsi_L2_norm_inequation}
		&\|\delta_\tau\left(\psi^n-R_h\psi^n\right)\|+h\|\delta_\tau\left(\psi^n-R_h\psi^n\right)\|_{H^1}\leq Ch^{r+1}\|\delta_\tau\psi^{n}\|_{H^{r+1}},\\
		\label{eq_Dtau_psi_Rhpsi_Hneg1_norm_inequation}
		&\|\delta_\tau\left(\psi^n-R_h\psi^n\right)\|_{H^{-1}}\leq Ch^{r+2}\|\delta_\tau\psi^n\|_{H^{r+1}},
	\end{flalign}
	where $\delta_\tau$ is defined in \eqref{eq_d_t}.
	Let $I_h:L^2(\Omega)\rightarrow S_h^r $  and $\boldsymbol{I}_h:\mathbf{L}^2(\Omega)\rightarrow\mathbf{X}_h^{r}$
	are defined the $L^2$ projection operators \cite{2008_Hou_An_efficient_semi_implicit_immersed_boundary_method_for_the_Navier_Stokes_equations,2015_HeYinnian_Unconditional_convergence_of_the_Euler_semi_implicit_scheme_for_the_three_dimensional_incompressible_MHD_equations} by
	\begin{flalign}
		\label{eq_L2_projection_operators}
		&\left(v-I_hv,w_h\right)=0,~\forall w_h\in S_h^r,\\
		&\left(\mathbf{v}-\boldsymbol{I}_h\mathbf{v},\mathbf{w}_h\right)=0,~\forall \mathbf{w}_h\in \mathbf{X}_h^{r}.
	\end{flalign}
	It is known to all that the $L^2$ projection 
	satisfies the following estimates,
	\begin{flalign}
		& \left\|v-I_h v\right\|+h\left\|\nabla\left(v-I_h v\right)\right\| \leq C h^{r+1}\|v\|_{H^{r+1}}, \\
		& \left\|\mathbf{v}-\boldsymbol{I}_h \mathbf{v}\right\|+h\left\|\nabla\left(\mathbf{v}-\boldsymbol{I}_h \mathbf{v}\right)\right\|\leq C h^{r+1}\|\mathbf{v}\|_{H^{r+1}}.
	\end{flalign}
	For $\left(\mathbf{v}_h,q_h\right)\in\mathbf{X}_h^{r}\times\mathring{S}_h^{r-1}$, the Stokes projection \cite{1986_Girault_Vivette_Finite_element_methods_for_Navier_Stokes_equations} $P_h$ 
	or $\boldsymbol{P}_h:\mathbf{H}_0^1(\Omega)\times L_0^2(\Omega)\rightarrow \mathbf{X}_h^{r}\times\mathring{S}_h^{r-1}$ 
	are defined by
	\begin{flalign}
		\label{eq_stokes_quasi_proojection_equation0001}
		&\left(\nabla\left(\mathbf{u}-\boldsymbol{P}_h(\mathbf{u},p)\right),\nabla \mathbf{v}_h\right)-\left(p-P_h(\mathbf{u},p),\nabla\cdot\mathbf{v}_h\right)=0,\\
		\label{eq_stokes_quasi_proojection_equation0002}
		&\left(\nabla\cdot\left(\mathbf{u}-\boldsymbol{P}_h(\mathbf{u},p)\right),q_h\right)=0.
	\end{flalign}
	To be the simplicity of notations, we denote $P_hp=P_h(\mathbf{u},p)$ and $\boldsymbol{P}_h\mathbf{u}=\boldsymbol{P}_h(\mathbf{u},p)$.
	\begin{Lemma}[\cite{2006_Thomee_Galerkin_finite_element_methods_for_parabolic_problems,2023_CaiWentao_Optimal_L2_error_estimates_of_unconditionally_stable_finite_element_schemes_for_the_Cahn_Hilliard_Navier_Stokes_system}]
		\label{lemma_stokes_projection001}
		Under the assumption of regularity \eqref{eq_varibles_satisfied_regularities}, it holds that
		\begin{equation}
			\|\mathbf{u}-\boldsymbol{P}_h\mathbf{u}\|+h\left(\|\nabla\left(\mathbf{u}-\boldsymbol{P}_h\mathbf{u}\right)\|+\|p-P_hp\|\right)
			\leq Ch^{r+1}\left(\|\mathbf{u}\|_{H^{r+1}}+\|p\|_{H^r}\right),
		\end{equation}
		and
		\begin{equation}
			\begin{aligned}
				\|\delta_{\tau}\left(\mathbf{u}^n-\boldsymbol{P}_h\mathbf{u}^n\right)\|\leq&~ Ch^{r+1}\left(\|\delta_{\tau}\mathbf{u}^n\|_{H^{r+1}}+\|\delta_{\tau}p^n\|_{H^r}\right).
			\end{aligned}
		\end{equation}
		where $C$ is the constant the only depends on $\Omega$.
	\end{Lemma}
	The proof of this lemma is similar to Lemma 3.2 of \cite{2023_CaiWentao_Optimal_L2_error_estimates_of_unconditionally_stable_finite_element_schemes_for_the_Cahn_Hilliard_Navier_Stokes_system}. Here we omit the detailed procedure and refer the interested reader to \cite{2023_CaiWentao_Optimal_L2_error_estimates_of_unconditionally_stable_finite_element_schemes_for_the_Cahn_Hilliard_Navier_Stokes_system}.
	\begin{Lemma}[\cite{2019_YangXiaofeng_Convergence_analysis_of_an_unconditionally_energy_stable_projection_scheme_for_magneto_hydrodynamic_equations}]
		\label{lemma_stokes_projection002}
		The Stokes projection $\boldsymbol{P}_h(\mathbf{u},p)$ is $\mathbf{H}^1$ stable in the case such that 
		\begin{equation}
			\|\boldsymbol{P}_h(\mathbf{u},p)\|_{H^1}\leq C\left(\|\mathbf{u}\|_{H^2}+\|p\|_{H^1}\right).
		\end{equation}
	\end{Lemma}
	Next we introduce the discrete Laplacian operator
	\cite{2020_Chen_Hongtao_Optimal_error_estimates_for_the_scalar_auxiliary_variable_finite_element_schemes_for_gradient_flows} $\Delta_h:\mathring{S}_h^{r-1}\rightarrow \mathring{S}_h^{r-1}$ such that
	\begin{flalign}
		\label{eq_discrete_Laplacian_operator}
		&\left(-\Delta_hv_h,\xi_h \right)=\left(\nabla v_h,\nabla\xi_h\right),\\
		\label{eq_discrete_Laplacian_operator11}
		&\left(\nabla(-\Delta_h^{-1})v_h,\nabla\xi_h\right)=\left(v_h,\xi_h\right).
	\end{flalign}
	By defining the discrete Stokes operator $A_h=-\boldsymbol{I}_h\Delta_h$ \cite{2024_YiNianyu_Convergence_analysis_of_a_decoupled_pressure_correction_SAV_FEM_for_the_Cahn_Hilliard_Navier_Stokes_model}, it holds
	\begin{flalign}
		\label{eq_discrete_stokes_operators001}
		\|v_h\|_{H^2}=\|A_hv_h\|,~\|v_h\|_{H^{-1}}=\|A_h^{-1/2}v_h\|,~v_h\in \mathbf{X}_h^{r},\\
		\label{eq_discrete_stokes_operators002}
		\|A_h^{1/2}v_h\|=\|\nabla v_h\|,~\|\nabla A_h^{-1/2}v_h\|=\|v_h\|,~v_h\in\mathbf{X}_h^{r}.
	\end{flalign}
	If $v_h$ is a constant, 
	we denote $(-\Delta_h)^{\frac{1}{2}}v_h=0$ and $(-\Delta_h)^{-\frac{1}{2}}v_h=0$.
	\begin{Lemma}[\cite{2023_CaiWentao_Optimal_L2_error_estimates_of_unconditionally_stable_finite_element_schemes_for_the_Cahn_Hilliard_Navier_Stokes_system}]\label{Lemma_operators_H10203}
		For $v_h\in\mathring{S}_h^{r-1}$, and the operators $\left(-\Delta_h\right)^{1/2}:\mathring{S}_h^{r-1}\rightarrow\mathring{S}_h^{r-1}$ and
		$\left(-\Delta_h\right)^{-1/2}:\mathring{S}_h^{r-1}\rightarrow\mathring{S}_h^{r-1}$, we have the estimates
		\begin{flalign}
			\label{eq_operators_estimates_positive_one}
			&\|v_h\|_{H^{1}}\leq C\|\left(-\Delta_h\right)^{1/2}v_h\|\leq c\|v_h\|_{H^1},\\
			\label{eq_operators_estimates_neg_one}
			&\|v_h\|_{H^{-1}}\leq C\|\left(-\Delta_h\right)^{-1/2}v_h\|\leq c\|v_h\|_{H^{-1}},
		\end{flalign}
		and
		\begin{equation}\label{eq_operators_estimates_neg_3}
			\|v_h\|^2\leq\varepsilon\|\nabla v_h\|^2+C\varepsilon^{-1}\|\left(-\Delta_h\right)^{-1/2}v_h\|^2,
		\end{equation}
		where $\varepsilon$ is an arbitrary positive constant independent of $h$.
	\end{Lemma}
	Finally, we note the well-known inverse inequalities \cite{2023_CaiWentao_Optimal_L2_error_estimates_of_unconditionally_stable_finite_element_schemes_for_the_Cahn_Hilliard_Navier_Stokes_system}
	\begin{flalign}
		\label{eq_inverse_inequalities001}
		&\|v_h\|_{L^4}\leq Ch^{-1/2}\|v_h\|, \quad \forall v_h\in S_h^r,\\
		\label{eq_inverse_inequalities002}
		&\|v_h\|_{L^\infty}\leq C\sqrt{\ln (1/h)}\|v_h\|_{H^1}, \quad \forall v_h\in S_h^r.
	\end{flalign}
	
	\begin{Lemma}[\cite{2006_FengXiaobing_FullydiscretefiniteelementapproximationsoftheNavierStokesCahnHilliarddiffuseinterfacemodelfortwophasefluidflows,2015_Diegel_Analysis_of_a_mixed_finite_element_method_for_a_Cahn_Hilliard_Darcy_Stokes_system}]
		\label{Lemma_0402}
		Suppose $g\in H^1(\Omega)$, and $v\in \mathring{S}_h^{r-1}$,
		\begin{equation}
			|(g,v)|\leq C\|\nabla g\|\|v\|_{H^{-1}},
		\end{equation}
		and for all $g\in L_0^2(\Omega)$, we have 
		\begin{equation}
			\|g\|_{H^{-1}}\leq C\|g\|.
		\end{equation}
		where $C$ is positive constant independent of $\tau$ and $h$.
	\end{Lemma}
	The proof of the lemma \ref{Lemma_0402} is similar to Lemma 3.2 of  
	\cite{2006_FengXiaobing_FullydiscretefiniteelementapproximationsoftheNavierStokesCahnHilliarddiffuseinterfacemodelfortwophasefluidflows}, here we omit it. 
	The following inequalities hold \cite{2007_HeYinnian_SunWeiwei_Stability_and_convergence_of_the_Crank_Nicolson_Adams_Bashforth_scheme_for_the_time_dependent_Navier_Stokes_equations}
	\begin{flalign}
		\label{eq_basic_inequalities_001}
		&\|\phi\|_{L^r}\leq C\|\phi\|_{H^1},~\forall \phi\in H^1(\Omega),~ r\in [2,\infty],\\
		\label{eq_basic_inequalities_002}
		&\|\phi\|_{L^{\infty}}\leq C\|\phi\|^{\frac{1}{2}}\left(\|\phi\|^2+\|\Delta \phi\|^2\right)^{\frac{1}{2}}, \forall \phi \in H^2(\Omega),\\
		\label{eq_basic_inequalities_003}
		&\|\boldsymbol{v}\|_{L^r}\leq C\|\nabla\boldsymbol{v}\|^2, ~(2\leq r\leq 6),
		\quad\|\boldsymbol{v}\|_{L^4}\leq C\|\boldsymbol{v}\|^{\frac{1}{2}}\|\nabla\boldsymbol{v}\|^{\frac{1}{2}}, ~\forall \boldsymbol{v}\in \mathbf{H}_0^1(\Omega),\\
		\label{eq_basic_inequalities_004}
		&\|\boldsymbol{v}\|_{H^2}\leq C\|A\boldsymbol{v}\|^2, \quad
		\|\boldsymbol{v}\|_{L^\infty}\leq C\|\boldsymbol{v}\|^{\frac{1}{2}}\|A\boldsymbol{v}\|^{\frac{1}{2}},
		\forall \boldsymbol{v}\in \mathbf{H}^2(\Omega)\cap\mathbf{H}_0^1(\Omega),
	\end{flalign}
	where $A$ is the continuous Stokes operator \cite{1986_Girault_Vivette_Finite_element_methods_for_Navier_Stokes_equations}.
	
	We will frequently use the following discrete version of the Gr\"{o}nwall lemma
	\cite{1990_Heywood_Finite_element_approximation_of_the_nonstationary_Navier_Stokes_problem_IV_Error_analysis_for_second_order_time_discretization}:
	\begin{Lemma}[discrete Gr\"{o}nwall inequality \cite{1990_Heywood_Finite_element_approximation_of_the_nonstationary_Navier_Stokes_problem_IV_Error_analysis_for_second_order_time_discretization}]
		\label{lemma_discrete_Gronwall_inequation}
		For all $0 \leq n \leq m$, let $a_n, b_n, c_n, d_n, \tau, C \geq 0$ such that
		\begin{equation}
			a_m+\tau\sum_{n=0}^{m}b_n\leq \tau\sum_{n=0}^{m}d_na_n+\tau\sum_{n=0}^{m}c_n+C,
		\end{equation}
		suppose that $\tau d_n<1$, for all $n\geq 0$, and set $\sigma_n=(1-\tau d_n)^{-1}$. Then
		\begin{equation}
			a_{m}+\tau\sum_{n=0}^{m} b_n \leq \exp \left(\tau\sum_{n=0}^m \sigma_nd_n\right)\left\{a_0+\left(b_0+c_0\right) \tau+\tau\sum_{n=1}^{m} c_n \right\} .
		\end{equation}
	\end{Lemma}
	Throughout the manuscript we use $C$ or $c$, with or without subscript, to denote a positive constant independent of discretization parameters that could have uncertain values in different places. 
	
	\section{The numerical scheme}\label{section_the_numerical_schemes}
	In the section, we  construct first-order fully decoupled scheme,
	and prove that the numerical scheme satisfies energy stabilization.
	For $r\geq 2$, we assume that the solutions to the CHNS model \eqref{eq_chns_equations} in this manuscript
	exist and satisfy the following regularity assumption:
	\begin{equation}
		\label{eq_varibles_satisfied_regularities}
		\begin{aligned}
			&\phi\in L^\infty\left(0,T;H^{r+1}(\Omega)\right),~
				\phi_t\in L^\infty\left(0,T;H^1(\Omega)\right)\cap L^2\left(0,T;H^1(\Omega)\right),
				\phi_{tt}\in L^2\left(0,T;L^2(\Omega)\right),\\
			&\mu\in L^\infty\left(0,T;H^{r+1}(\Omega)\right),~\mu_t\in L^{\infty}\left(0,T;H^1(\Omega)\right),
			~p\in L^\infty\left(0,T;H^{r}(\Omega)\right),~p_t\in L^2\left(0,T;H^r(\Omega)\right),\\
			&\mathbf{u}\in L^\infty\left(0,T;\mathbf{H}^{r+1}(\Omega)\right),~
				\mathbf{u}_t\in  L^\infty\left(0,T;\mathbf{H}^1(\Omega)\right)\cap L^2\left(0,T;\mathbf{H}^{r+1}(\Omega)\right), ~
				\mathbf{u}_{tt}\in L^2\left(0,T;\mathbf{H}^{-1}(\Omega)\right).
		\end{aligned}
	\end{equation}
	\subsection{Fully discrete scheme}
	In the subsection, we construct a decoupled, first-order and fully discrete finite elements method based on Euler implicit-explict scheme for solving the CHNS system. For any test functions $w\in H^1(\Omega)$, $\varphi\in H^1(\Omega)$, $\mathbf{v}\in \mathbf{H}_0^1(\Omega)$, $q \in L_0^2(\Omega)$,
	the weak solutions $\left(\phi,\mu,\mathbf{u},p\right)$ of \eqref{eq_chns_equations} satisfy the following variational 
	forms:
	\begin{flalign}
		\label{eq_variational_forms_phi}
		\left(\partial_t\phi,w\right)+M\left(\nabla\mu,\nabla w\right)+b\left(\phi,\mathbf{u},w\right)&=0,\\
		\label{eq_variational_forms_mu}
		\left(\mu,\varphi\right)-\lambda\left(\nabla\phi,\nabla\varphi\right)-\frac{\lambda}{\epsilon^2}\left(\phi^3-\phi,\varphi\right)&=0,\\
		\label{eq_variational_forms_ns}
		\left(\partial_t\mathbf{u},\mathbf{v}\right)+\nu\left(\nabla\mathbf{u},\nabla\mathbf{v}\right)
		+\boldsymbol{B}\left(\mathbf{u},\mathbf{u},\mathbf{v}\right)+\left(\nabla p,\mathbf{v}\right)-b\left(\phi,\mathbf{v},\mu\right)&=0,\\
		\label{eq_variational_forms_incompressible_condition}
		\left(\nabla\cdot\mathbf{u},q\right)&=0.
	\end{flalign}
	Let $N$ be a positive integer and $\left\{t^n=n \tau\right\}_{n=0}^N$ denotes a uniform partition of the time interval $[0, T]$ with a step size $\tau=T / N$. For $n\geq 0$, choosing 
	the initial conditions $\left(\phi_h^0,\mu_h^0,\mathbf{u}_h^0,p_h^0\right)$ and $\left(\phi_h^n,\mu_h^n,\mathbf{u}_h^n,p_h^n\right)$, and updating  $\left(\phi_h^{n+1},\mu_h^{n+1},\tilde{\mathbf{u}}_h^{n+1},\mathbf{u}_h^{n+1},p_h^{n+1}\right)$,
	we construct the following fully decoupled, energy stable scheme for the CHNS system.
	\\
	$\mathbf{Step 1.}$ Find $\left(\phi_h^{n+1},\mu_h^{n+1}\right)\in \left(S_h^r, S_h^r\right)$ such that
	\begin{flalign}
		\label{eq_fully_discrete_CHNS_scheme_phi}
		\left(\delta_\tau\phi_h^{n+1},w_h\right)+b\left(\phi_h^{n+1},\mathbf{u}_h^{n},w_h\right)+M\left(\nabla\mu_h^{n+1},\nabla w_h\right)=0,&\\
		\label{eq_fully_discrete_CHNS_scheme_mu}
		\left(\mu_h^{n+1},\varphi_h\right)-\lambda\left(\nabla\phi_h^{n+1},\nabla\varphi_h\right)-\frac{\lambda}{\epsilon^2}\left(\left(\phi_h^{n+1}\right)^3-\phi_h^n,\varphi_h\right)=0,&\\
		\label{eq_fully_discrete_CHNS_scheme_boundary_conditions}
		\partial_n\phi_h^{n+1}|_{\partial\Omega}=0,~	\partial_n\mu_h^{n+1}|_{\partial\Omega}=0.
	\end{flalign}
	$\mathbf{Step 2.}$ Find $\tilde{\mathbf{u}}_h^{n+1}\in \mathbf{X}_h^{r}$ such that
	\begin{flalign}
		\label{eq_fully_discrete_CHNS_scheme_tilde_u}
		\left(\frac{\tilde{\mathbf{u}}_h^{n+1}-\mathbf{u}_h^n}{ \tau},\mathbf{v}_h\right)
		+\boldsymbol{B}\left(\mathbf{u}_h^n,\tilde{\mathbf{u}}_h^{n+1},\mathbf{v}_h\right)+\nu\left(\nabla\tilde{\mathbf{u}}_h^{n+1},\nabla\mathbf{v}_h\right)
		+\left(\nabla p_h^n,\mathbf{v}_h\right)-b\left(\phi_h^{n+1},\mathbf{v}_h,\mu_h^{n+1}\right)=0,\\
		\tilde{\mathbf{u}}_h^{n+1}|_{\partial\Omega}=0.
	\end{flalign}
	$\mathbf{Step 3.}$
	Find $\left(p_h^{n+1},\mathbf{u}_h^{n+1}\right)\in \left(\mathring{S}_h^{r-1},\mathbf{X}_h^{r}\right)$ such that
	\begin{flalign}
		\label{eq_fully_discrete_CHNS_scheme_tilde_u_u1_p_p1}
		\frac{\mathbf{u}_h^{n+1}-\tilde{\mathbf{u}}_h^{n+1}}{\tau}+\nabla\left(p_h^{n+1}-p_h^n\right)=0,\\
		\label{eq_fully_discrete_CHNS_scheme_tilde_incompressible_condition}
		\left(\nabla\cdot\mathbf{u}_h^{n+1},q_h\right)=0,~\mathbf{u}_h^{n+1}\cdot\mathbf{n}|_{\partial\Omega}=0,
	\end{flalign}
	hold for all $\left(w_h,\varphi_h,\mathbf{v}_h,q_h\right)\in\mathbf{\mathcal{X}}_h^r$, and $n=0,1,2,\cdots,N-1$, where $\phi_h^0=R_h\phi^0,~\mathbf{u}_h^0=\boldsymbol{I}_h\mathbf{u}^0.$
	For simplicity of description, we perform the fully discretization to prove that our numerical scheme satisfies energy stabilization. For the convenience of the following theoretical analysis, we set the parameters $\lambda,M,\nu$ to 1.
	
	\subsection{Unique solvability and energy stability}
	For the fully discrete scheme \eqref{eq_fully_discrete_CHNS_scheme_phi}-\eqref{eq_fully_discrete_CHNS_scheme_tilde_incompressible_condition}, we assume that the initial solutions satisfly the following boundedness:
	\begin{equation}\label{eq_initial_value_boundedness}
		\left(F(\phi_h^0),1\right)+\|\phi_h^0\|_{H^1}+\|\mathbf{u}_h^0\|\leq C,
	\end{equation}
	where $C$ is a constant independent of mesh size $h$.
	
	In this subsection, we establish the existence of solutions for the fully discrete scheme \eqref{eq_fully_discrete_CHNS_scheme_phi}-\eqref{eq_fully_discrete_CHNS_scheme_tilde_incompressible_condition}. This is achieved by applying the Leray–Schauder fixed point theorem, and we also prove the uniqueness of the numerical solutions. Given that the scheme \eqref{eq_fully_discrete_CHNS_scheme_phi}-\eqref{eq_fully_discrete_CHNS_scheme_tilde_incompressible_condition} is implemented in a stepwise, decoupled manner, the existence of a solution is demonstrated incrementally. For the \textbf{Step 2} and \textbf{Step 3} of the discrete scheme, the proofs of existence and uniqueness are standard results; thus, we omit the detailed discussions here. Instead, our focus is on proving the existence and uniqueness of the solution for \textbf{Step 1}, specifically for the fully discrete scheme \eqref{eq_fully_discrete_CHNS_scheme_phi}-\eqref{eq_fully_discrete_CHNS_scheme_mu} corresponding to the CH equation.
		\begin{Lemma}
		[\cite{2019_HeXiaoming_A_diffuse_interface_model_and_semi_implicit_energy_stable_finite_element_method_for_two_phase_magnetohydrodynamic_flows}]
		\label{lemma_fixed_point}
		Let $\mathcal{G}$ denote a compact mapping from a Banach space $B$ into itself, and assume that there exists a constant $C$ such that 
		$
		\|x\|<C
		$ for all $x\in B$ and $\alpha\in [0,1]$ satisfying $x=\alpha\mathcal{G}x$. Thus,  $\mathcal{G}$ has a fixed point.
	\end{Lemma} 
	\begin{Lemma}\label{lemma_solutions_boundedness}
		Suppose the initial solutions boundedness \eqref{eq_initial_value_boundedness} holds. For given $\tau, h>0$,
		there exists a solution $\left(\phi_h^{n+1},\mu_h^{n+1},\mathbf{u}_h^{n+1},p_h^{n+1}\right)$ to the fully discrete scheme \eqref{eq_fully_discrete_CHNS_scheme_phi}-\eqref{eq_fully_discrete_CHNS_scheme_tilde_incompressible_condition}.
	\end{Lemma}
	\begin{Lemma}\label{lemma_unique_solvability}
		For $\tau,h>0$ and $\tau\leq C\alpha/{\|\mathbf{u}_h^n\|_{L^{\infty}}^2}$, the fully discrete scheme \eqref{eq_fully_discrete_CHNS_scheme_phi}-\eqref{eq_fully_discrete_CHNS_scheme_tilde_incompressible_condition} is unique solvability.
	\end{Lemma}
	The proofs of Lemmas \ref{lemma_solutions_boundedness} and \ref{lemma_unique_solvability} are given in Appendix \ref{subsection_appendix_proof_lemma0302} and \ref{subsection_appendix_proof_lemma0303}.
	
	\begin{Remark}
		According to the assumption $\tau\leq \frac{\alpha}{C\|\mathbf{u}_h^n\|_{L^{\infty}}^2}$, the above Lemma \eqref{lemma_unique_solvability} holds. Therefore, we can derive the boundedness of the nth-time-layer velocity as follows:
		\begin{equation}
			\label{eq_Linfty_u_boundedness}
			\|\mathbf{u}_h^n\|_{L^{\infty}}\leq C\|\mathbf{u}\|_{L^{\infty}(0,T;\mathbf{L}^{\infty}(\Omega))}\leq C,
		\end{equation}
		where $C$ is a positive constant independent of $\alpha$, $\tau$, $h$, and $n$.
		Later in this paper, we will analyze the reasonableness of this condition $\tau\leq \frac{\alpha}{C\|\mathbf{u}_h^n\|_{L^{\infty}}^2}$.
	\end{Remark}
	We explain later in Remark \ref{remark_uinfty_boundedness_explanation} the feasibility of showing that the inequality holds.
	For the above scheme, we can establish the theorem of energy stabilization as follows:
	\begin{Theorem}
		\label{theorem_unconditionally_energy_stable} 
		Let $\phi_h^n\in S_h^r $ and $\mathbf{u}_h^n\in \mathbf{X}_h^{r}$ 
		and the initial values $|\phi_h^0|\leq C_1$ and $E(\phi_h^0,\mathbf{u}_h^0)\leq C_0$, i.e., the initial energy is bounded.
		For all time partitions $\tau$ and mesh size $h$ such that $ \tau\leq C\alpha/{\|\mathbf{u}_h^n\|_{L^{\infty}}^2}$ and the second step satisfies $\tau\leq Ch/\|\phi_h^{n+1}\|_{L^4}^2$, the decoupled discrete scheme \eqref{eq_fully_discrete_CHNS_scheme_phi}-\eqref{eq_fully_discrete_CHNS_scheme_tilde_incompressible_condition} is uniquely solvable and satisfies discrete energy law as follows:
		\begin{equation}
			\label{eq_semi_discrete_scheme_energy_law}
			\begin{aligned}
				\tilde{E}_h^{n+1}-\tilde{E}_h^{n}\leq -M\tau \|\nabla\mu_h^{n+1}\|^2-2\nu\tau\|\nabla\tilde{\mathbf{u}}_h^{n+1}\|^2,
			\end{aligned}
		\end{equation}
		where $\tilde{E}_h^{n+1}$ is defined by
		\begin{equation}
			\tilde{E}_h^{n+1}=\lambda\|\nabla\phi_h^{n+1}\|^2+\frac{\lambda}{2\epsilon^2}\|(\phi_h^{n+1})^2-1\|^2+\|\mathbf{u}_h^{n+1}\|^2+\tau^2\|\nabla p_h^{n+1}\|^2,
		\end{equation}
	and $C_0$ and $C_1$ are constant that does  depend on the initial value of $\phi$ and $\mathbf{u}$.
	\end{Theorem}
	\begin{proof}
		Letting $w_h=2\tau\mu_h^{n+1}$ in \eqref{eq_fully_discrete_CHNS_scheme_phi}, we can obtain
		\begin{equation}
			\label{eq_semi_discrete_CHNS_scheme_phi_inner_2tau_mu}
			2\left(\phi_h^{n+1}-\phi_h^n,\mu_h^{n+1}\right)+2\tau\left(\mathbf{u}_h^n\cdot\nabla\phi_h^{n+1},\mu_h^{n+1}\right)+2M\tau\|\nabla\mu_h^{n+1}\|^2=0.
		\end{equation}
		Letting  $\varphi_h=-2\left(\phi_h^{n+1}-\phi_h^n\right)$ in \eqref{eq_fully_discrete_CHNS_scheme_mu}, we can obtain
		\begin{equation}
			\label{eq_inner_product_eq_semi_discrete_CHNS_scheme_mu}
			\begin{aligned}
				-2\left(\mu_h^{n+1},\phi_h^{n+1}-\phi_h^n\right)+\lambda\left(\|\nabla\phi_h^{n+1}\|^2-\|\nabla\phi_h^{n}\|^2+\|\nabla\phi_h^{n+1}-\nabla\phi_h^n\|^2\right)&\\
				+\frac{2\lambda\tau}{\epsilon^2}
				\left(\frac{1}{4}\delta_\tau\|\left(\phi_h^{n+1}\right)^2-1\|^2+\frac{\tau}{4}
				\left(\|\delta_\tau(\phi_h^{n+1})^2\|^2+2\|\phi_h^{n+1}\delta_\tau\phi_h^{n+1}\|^2
				+2\|\delta_\tau\phi_h^{n+1}\|^2\right)\right)&=0.
			\end{aligned}
		\end{equation}
		Letting  $\mathbf{v}_h=2\tau\tilde{\mathbf{u}}_h^{n+1}$ in \eqref{eq_fully_discrete_CHNS_scheme_tilde_u},
		we derive that
		\begin{equation}
			\label{eq_semi_discrete_CHNS_scheme_tilde_u_inner_product_2delta_t}
			\|\tilde{\mathbf{u}}_h^{n+1}\|^2-\|\mathbf{u}_h^n\|^2+\|\tilde{\mathbf{u}}_h^{n+1}-\mathbf{u}_h^n\|^2+2\nu\tau
			\|\nabla\tilde{\mathbf{u}}_h^{n+1}\|^2+2\tau\left(\nabla p_h^n,\tilde{\mathbf{u}}_h^{n+1}\right)-2\tau\left(\mu_h^{n+1}\nabla\phi_h^{n+1},\tilde{\mathbf{u}}_h^{n+1}\right)=0.
		\end{equation}
		Next, we rewrite \eqref{eq_fully_discrete_CHNS_scheme_tilde_u_u1_p_p1} as
		\begin{equation}
			\label{eq_semi_discrete_CHNS_scheme_tilde_u_u1_p_p1_rewrite}
			\mathbf{u}_h^{n+1}+\tau\nabla p_h^{n+1}=\tilde{\mathbf{u}}_h^{n+1}+\tau\nabla p_h^n.
		\end{equation}
		Then we derive from \eqref{eq_semi_discrete_CHNS_scheme_tilde_u_u1_p_p1_rewrite} that
		\begin{equation}
			\label{eq_semi_discrete_CHNS_scheme_tilde_u_u1_p_p1_rewrite_d}
			\|\mathbf{u}_h^{n+1}\|^2-\|\tilde{\mathbf{u}}_h^{n+1}\|^2-2\tau\left(\nabla p_h^n,\tilde{\mathbf{u}}_h^{n+1}\right)+\tau^2\left(\|\nabla p_h^{n+1}\|^2-\|\nabla p_h^n\|^2\right)=0.
		\end{equation}
		Adding the \eqref{eq_semi_discrete_CHNS_scheme_tilde_u_u1_p_p1_rewrite_d} and \eqref{eq_semi_discrete_CHNS_scheme_tilde_u_inner_product_2delta_t}, we can obtain
		\begin{equation}
			\label{eq_combing_two_equations}
			\begin{aligned}	
				\|\mathbf{u}_h^{n+1}\|^2&-\|\mathbf{u}_h^n\|^2+\|\tilde{\mathbf{u}}_h^{n+1}-\mathbf{u}_h^n\|^2+2\tau
				\|\nabla\tilde{\mathbf{u}}_h^{n+1}\|^2+\tau^2\left(\|\nabla p_h^{n+1}\|^2-\|\nabla p_h^n\|^2\right)\\
				&=2\tau\left(\mu_h^{n+1}\nabla\phi_h^{n+1},\tilde{\mathbf{u}}_h^{n+1}\right).
			\end{aligned}
		\end{equation}
		Combining the \eqref{eq_semi_discrete_CHNS_scheme_phi_inner_2tau_mu},  \eqref{eq_inner_product_eq_semi_discrete_CHNS_scheme_mu} with \eqref{eq_combing_two_equations}, we get
		\begin{equation}
			\label{eq_combining_three_terms}
			\begin{aligned}
				\lambda&\left(\|\nabla\phi_h^{n+1}\|^2-\|\nabla\phi_h^{n}\|^2+\|\nabla\phi_h^{n+1}-\nabla\phi_h^n\|^2\right)\\
				&+\frac{2\lambda\tau}{\epsilon^2}
				\left(\frac{1}{4}\delta_\tau\|\left(\phi_h^{n+1}\right)^2-1\|^2+\frac{\tau}{4}
				\left(\|\delta_\tau(\phi_h^{n+1})^2\|^2\right.\right.\\
				&\left.\left.\qquad+2\|\phi_h^{n+1}\delta_\tau\phi_h^{n+1}\|^2
				+2\|\delta_\tau\phi_h^{n+1}\|^2\right)\right)\\
				&+\|\mathbf{u}_h^{n+1}\|^2-\|\mathbf{u}_h^n\|^2+\|\tilde{\mathbf{u}}_h^{n+1}-\mathbf{u}_h^n\|^2+\tau^2\left(\|\nabla p_h^{n+1}\|^2-\|\nabla p_h^n\|^2\right)\\
				=&-2M\tau\|\nabla\mu_h^{n+1}\|^2-2\nu\tau
				\|\nabla\tilde{\mathbf{u}}_h^{n+1}\|^2
				+2\tau\left(\mu_h^{n+1}\nabla\phi_h^{n+1},\tilde{\mathbf{u}}_h^{n+1}-\mathbf{u}_h^n\right).
				\end{aligned}
			\end{equation}
		According to \eqref{lemma_unique_solvability}, we obtain $\|\phi_h^{n+1}\|_{H^1}\leq C$.
	Recalling inverse inequality  and
choosing $\tau\leq Ch$, we have
\begin{equation}\label{eq_boundness_mu_phi_u}
\begin{aligned}
&2\tau\bigg|\left(\mu_h^{n+1}\nabla\phi_h^{n+1},\tilde{\mathbf{u}}_h^{n+1}-\mathbf{u}_h^n\right)\bigg|\\
\leq&~C\tau\|\nabla\mu_h^{n+1}\|\|\phi_h^{n+1}\|_{L^4}\|\tilde{\mathbf{u}}_h^{n+1}-\mathbf{u}_h^n\|_{L^4}\\
\leq&~\frac{C_2\tau}{h}\|\phi_h^{n+1}\|_{L^4}^2\|\tilde{\mathbf{u}}_h^{n+1}-\mathbf{u}_h^n\|^2+\tau\|\nabla\mu_h^{n+1}\|^2\\
\leq &~\frac{1}{2}\|\tilde{\mathbf{u}}_h^{n+1}-\mathbf{u}_h^n\|^2+M\tau\|\nabla\mu_h^{n+1}\|^2,
\end{aligned}
\end{equation}
	where $C$ and $C_2$ are constant that does not depend on $\tau$ and $h$. Combining \eqref{eq_boundness_mu_phi_u} with \eqref{eq_combining_three_terms}, and
	summing up form $n=0$ to $m$,
	we have
	\begin{equation}\label{eq_combining_summing_equation}
		\tilde{E}_h^{m+1}-\tilde{E}_h^0\leq -\tau\sum_{n=0}^{m}\left(2\|\nabla\tilde{\mathbf{u}}_h^{n+1}\|^2+\|\nabla\mu_h^{n+1}\|^2\right)-\sum_{n=0}^{m}\frac{1}{2}\|\tilde{\mathbf{u}}_h^{n+1}-\mathbf{u}_h^n\|^2.
	\end{equation}
	Thus, we can always obtain a positive number $C$ that does not mesh size $h$ and time step $\tau$ such that $\|\phi_h^{n+1}\|_{L^4}\leq \|\phi_h^{n+1}\|_{H^1}\leq C$ is bounded, and derive the desired result \eqref{eq_semi_discrete_scheme_energy_law}. 
	\end{proof}
	\begin{Remark}\label{remark_existence_solution}
		The existence of the solution for scheme \eqref{eq_fully_discrete_CHNS_scheme_phi}-\eqref{eq_fully_discrete_CHNS_scheme_tilde_incompressible_condition} 
		can be proved using the standard argument of the Leray-Schauder fixed point theorem (see \cite{2015_ShenJie_Decoupled_energy_stable_schemes_for_phase_field_models_of_two_phase_incompressible_flows,2008_Kay_David_and_Welford_Richard_Finite_element_approximation_of_a_Cahn_Hilliard_Navier_Stokes_system} for details of similar problems).
	\end{Remark}
	
	Then, we can derive the following lemma regarding the boundedness of the numerical solutions obtained from the scheme \eqref{eq_fully_discrete_CHNS_scheme_phi}-\eqref{eq_fully_discrete_CHNS_scheme_tilde_incompressible_condition}.
	\begin{Lemma}
		\label{Lemma_0301}
		Let $\left(\phi_h^{n+1},\mu_h^{n+1},\tilde{\mathbf{u}}_h^{n+1},\mathbf{u}_h^{n+1},p_h^{n+1}\right)\in \mathbf{\mathcal{X}}_h^r $ be the unique solution of  \eqref{eq_fully_discrete_CHNS_scheme_phi}-\eqref{eq_fully_discrete_CHNS_scheme_tilde_incompressible_condition}. Suppose 
		that $E(\phi_h^0,\mathbf{u}_h^0)<C_0$, 
		for any $h,\tau>0$ and the time step
		$\tau$ and mesh size $h$ satisfy the conditions of Theorem \ref{theorem_unconditionally_energy_stable}, the following boundednesses hold
		\begin{flalign}
			\label{eq_boundedness_estimates_phi}
			\max_{0\leq n\leq N-1}\left(\|\mathbf{u}_h^{n+1}\|^2+\|\nabla\phi_h^{n+1}\|^2
			+\frac{\tau}{2\epsilon^2}\|\left(\phi_h^{n+1}\right)^2-1\|^2+\tau^2\|\nabla p_h^{n+1}\|^2\right)&\leq C_0,\\
			\tau\sum_{n=0}^{N-1}\left(\|\nabla\mu_h^{n+1}\|^2+\|\nabla\tilde{\mathbf{u}}_h^{n+1}\|^2\right)&\leq C_0,\\
			\sum_{n=0}^{N-1}\left(\|\nabla\phi_h^{n+1}-\nabla\phi_h^n\|^2
			+\|\tilde{\mathbf{u}}_h^{n+1}-\mathbf{u}_h^n\|^2\right)&\leq C_0,\\
			\frac{\tau^2}{\epsilon^2}\sum_{n=0}^{N-1}
			\left[\frac{1}{2}
			\|\delta_\tau(\phi_h^{n+1})^2\|^2+\|\phi_h^{n+1}\delta_\tau\phi_h^{n+1}\|^2
			+\|\delta_\tau\phi_h^{n+1}\|^2\right]&\leq C_0,
		\end{flalign}
		where $C_0$ is a positive constant dependent on the initial values of $\phi^0$ and $\mathbf{u}^0$.
	\end{Lemma}
The proof of this lemma can be derived from the result of Theorem \ref{theorem_unconditionally_energy_stable}, the detailed steps of which we omit.
\begin{Lemma}\label{lemma_boundness_phi_mu}
	For all $m\geq 0$, there exists positive constant $C$ such that
	\begin{equation}\label{eq_lemma_first_result}
		\|\phi_h^{m+1}\|_{H^1}^2+\tau\sum_{n=0}^{m}\|\mu_h^{n+1}\|^2\leq C\tau,
	\end{equation}
and 
\begin{equation}\label{eq_lemma_second_result}
	\|\Delta_h\phi_h^{m+1}\|^2+\tau\sum_{n=0}^{m}\|\Delta_h\mu_h^{n+1}\|^2\leq C\tau.
\end{equation}
\end{Lemma}
\begin{proof}
	Letting $w_h=2\tau\phi_h^{n+1}$ in \eqref{eq_fully_discrete_CHNS_scheme_phi}, we obtain
	\begin{equation}
		\label{eq_TIP_2tau_phi}
		\begin{aligned}
			&\|\phi_h^{n+1}\|^2-\|\phi_h^n\|^2+\|\phi_h^{n+1}-\phi_h^n\|^2=-2\tau\left(\mathbf{u}_h^n\cdot\nabla\phi_h^{n+1},\phi_h^{n+1}\right)
			-2\tau\left(\nabla\mu_h^{n+1},\nabla\phi_h^{n+1}\right).
		\end{aligned}
	\end{equation}
Using the Young inequality and Cauchy-Sachwarz inequality, we can obtain
\begin{equation}
	\begin{aligned}
		\big|-2\tau\left(\mathbf{u}_h^n\cdot\nabla\phi_h^{n+1},\phi_h^{n+1}\right)\big|
		\leq &~2\tau\|\mathbf{u}_h^n\|_{L^4}\|\nabla\phi_h^{n+1}\|\|\phi_h^{n+1}\|_{L^4}\\
		\leq&~ \tau\|\nabla\phi_h^{n+1}\|^2+C\tau\|\nabla\mathbf{u}_h^n\|^2\left(\|\phi_h^{n+1}\|^2+\|\nabla\phi_h^{n+1}\|^2\right)
	\end{aligned}
\end{equation}
and
\begin{equation}
	\big|-2\tau\left(\nabla\mu_h^{n+1},\nabla\phi_h^{n+1}\right)\big|\leq \tau\left(\|\nabla\mu_h^{n+1}\|^2+\|\nabla\phi_h^{n+1}\|^2\right).
\end{equation}
Combining the above two inequalities with \eqref{eq_TIP_2tau_phi} and summing up for $n$ from $0$ to $m$, we obtain
\begin{equation}
	\label{eq_combining_the_two_inequalities_summing_up}
	\|\phi_h^{m+1}\|^2\leq \|\phi_h^0\|^2+\tau\sum_{n=0}^{m}\left(\|\nabla\mu_h^{n+1}\|^2+2\|\nabla\phi_h^{n+1}\|^2+C\|\nabla\mathbf{u}_h^n\|^2\left(\|\phi_h^{n+1}\|^2+\|\nabla\phi_h^{n+1}\|^2\right)\right).
\end{equation}
Using the lemma \ref{lemma_discrete_Gronwall_inequation} and \ref{Lemma_0301}, we have
\begin{equation}
	\label{eq_phi_h_L2_norm_boundness}
	\|\phi_h^{m+1}\|^2\leq \left(\|\phi_h^{0}\|^2+c_1\left(1+4\tau+c_2\right)\right)\exp(c)\leq C\tau.
\end{equation}
Thererfore, we can derive 
\begin{equation}
	\label{eq_phi_h_H1_norm_boundness}
	\|\phi_h^{n+1}\|_{H^1}^2\leq \|\phi_h^{n+1}\|^2+\|\nabla\phi_h^{n+1}\|^2\leq C\tau.
\end{equation}
Testing $\varphi_h=\tau\mu_h^{n+1}$ in \eqref{eq_fully_discrete_CHNS_scheme_mu}, we obtain
\begin{equation}
	\label{eq_testing_tau_mu_n1}
	\begin{aligned}
		\tau\|\mu_h^{n+1}\|^2=&~\tau\left(\nabla\phi_h^{n+1},\nabla\mu_h^{n+1}\right)+\tau\left((\phi_h^{n+1})^3-\phi_h^n,\mu_h^{n+1}\right)\\
		\leq&~\tau\left(\|\nabla\phi_h^{n+1}\|^2+\|\nabla\mu_h^{n+1}\|^2+\|\mu_h^{n+1}\|^2+\|(\phi_h^{n+1})^3-\phi_h^n\|^2\right).
	\end{aligned}
\end{equation}
Now using the (2.37) in \cite{2015_Diegel_Analysis_of_a_mixed_finite_element_method_for_a_Cahn_Hilliard_Darcy_Stokes_system}, 
we have
\begin{equation}\label{eq_phi3_error}
	\|(\phi_h^{n+1})^3-\phi_h^n\|^2\leq 2\left(\|\phi_h^{n+1}\|_{L^6}^6+\|\phi_h^n\|^2\right)\leq C\left(\|\phi_h^{n+1}\|_{H^1}^2+\|\phi_h^n\|^2\right)
	\leq C.
\end{equation}
Thus, combining the above inequalities with \eqref{eq_testing_tau_mu_n1} and summing up for $n$ from $0$ to $m$, we obtain
\begin{equation}\label{eq_combin_sum_mu}
	\tau\sum_{n=0}^{m}\|\mu_h^{n+1}\|^2\leq C\tau.
\end{equation}
From \eqref{eq_phi_h_H1_norm_boundness} and \eqref{eq_combin_sum_mu}, we obtain the result of \eqref{eq_lemma_first_result}.
We nextly prove the second term of this lemma.
Letting $w_h=2\tau\Delta_h\mu_h^{n+1}$  in \eqref{eq_fully_discrete_CHNS_scheme_phi}, and $\varphi_h=\Delta_h\delta_{\tau}\phi_h^{n+1}$
in \eqref{eq_fully_discrete_CHNS_scheme_mu}, and combining the results of two terms, 
we obtain
\begin{equation}
	\label{eq_Delta_phi_mu}
	\begin{aligned}
		&~\|\Delta_h\phi_h^{n+1}\|^2-\|\Delta_h\phi_h^n\|^2+\|\Delta_h\left(\phi_h^{n+1}-\phi_h^n\right)\|^2+2\tau\|\Delta_h\mu_h^{n+1}\|^2\\
		=&~2\tau\left(\mathbf{u}_h^n\cdot\nabla\phi_h^{n+1},\Delta_h\mu_h^{n+1}\right)
			+2\tau\left((\phi_h^{n+1})^3-\phi_h^n,\Delta_h\delta_{\tau}\phi_h^{n+1}\right).
	\end{aligned}
\end{equation}
Using the Lemma 2.14 in \cite{2015_Diegel_Analysis_of_a_mixed_finite_element_method_for_a_Cahn_Hilliard_Darcy_Stokes_system},
we can estimate the right-hand side of \eqref{eq_Delta_phi_mu} as follows:
\begin{equation}
	\label{eq_estimate_right_hand_terms_first}
	\begin{aligned}
		\big|2\tau\left(\mathbf{u}_h^n\cdot\nabla\phi_h^{n+1},\Delta_h\mu_h^{n+1}\right)\big|
		\leq &~2\tau\|\mathbf{u}_h^n\|_{L^4}^2\|\nabla\phi_h^{n+1}\|_{L^4}\|\Delta_h\mu_h^{n+1}\|\\
		\leq &~C\tau\|\nabla\mathbf{u}_h^n\|\left(\|\nabla\phi_h^{n+1}\|+\|\Delta_h\phi_h^{n+1}\|\right)\|\Delta_h\mu_h^{n+1}\|\\
		\leq &~C\tau\|\nabla\mathbf{u}_h^n\|^2\left(\|\nabla\phi_h^{n+1}\|^2+\|\Delta_h\phi_h^{n+1}\|^2\right)+\tau\|\Delta_h\mu_h^{n+1}\|^2.
	\end{aligned}
\end{equation}
and
\begin{equation}
	\label{eq_estimate_right_hand_terms_second}
	\begin{aligned}
		\big|2\tau\left((\phi_h^{n+1})^3-\phi_h^n,\Delta_h\delta_{\tau}\phi_h^{n+1}\right)\big|
		\leq&~ 2\tau\|(\phi_h^{n+1})^3-\phi_h^n\|\|\Delta_h\delta_{\tau}\phi_h^{n+1}\|\\
		\leq&~c\tau^2\|\Delta_h\delta_{\tau}\phi_h^{n+1}\|^2+C.
	\end{aligned}
\end{equation}
Combining the above two inequalities \eqref{eq_estimate_right_hand_terms_first} and
\eqref{eq_estimate_right_hand_terms_second} with \eqref{eq_Delta_phi_mu} and summing up for $n$ from $0$ to $m$, we obtain
\begin{equation}
		\|\Delta_h\phi_h^{m+1}\|^2+\tau\sum_{n=0}^{m}\|\Delta_h\mu_h^{n+1}\|^2\leq \|\Delta_h\phi_h^0\|^2
			+\tau\sum_{n=0}^{m}\|\nabla\mathbf{u}_h^n\|^2\left(\|\nabla\phi_h^{n+1}\|^2+\|\Delta_h\phi_h^{n+1}\|^2\right).
\end{equation}
Using the lemma \ref{lemma_discrete_Gronwall_inequation} and \ref{Lemma_0301}, we
obtain the result of this lemma.
\end{proof}
\begin{Remark}
	 We note that the bound $\tau\leq Ch/\left(2\|\phi_h^{n+1}\|_{L^4}^2\right)$ in Lemma \ref{Lemma_0301} and 
	 \ref{lemma_boundness_phi_mu} is attainable, since from \eqref{eq_basic_inequalities_001} and Lemma \ref{lemma_boundness_phi_mu} and \eqref{eq_H1_phi_mu_boundedness}, we obtain
	 $\|\phi_h^{n+1}\|_{L^4}\leq C\|\phi_h^{n+1}\|_{H^1}\leq C$.  
	More similar details about energy stabilization under this discrete scheme can be found in reference \cite{2008_Kay_David_and_Welford_Richard_Finite_element_approximation_of_a_Cahn_Hilliard_Navier_Stokes_system}.
\end{Remark}

\section{Error analysis}\label{section_error_analysis}

In this section, we establish the error analysis for the scheme \eqref{eq_fully_discrete_CHNS_scheme_phi}-\eqref{eq_fully_discrete_CHNS_scheme_tilde_incompressible_condition}. 
We denote
$$
	\begin{aligned}
		&e_\phi^{n+1}=\phi^{n+1}-\phi_h^{n+1}=\phi^{n+1}-R_h\phi^{n+1}+R_h\phi^{n+1}-\phi_h^{n+1}
			=\Phi_{\phi h}^{n+1}+\Theta_{\phi h}^{n+1},\\
	   &e_\mu^{n+1}=\mu^{n+1}-\mu_h^{n+1}=\mu^{n+1}-R_h\mu^{n+1}+R_h\mu^{n+1}-\mu_h^{n+1}	=\Phi_{\mu h}^{n+1}+\Theta_{\mu h}^{n+1},
	\end{aligned}
$$
and
$$
	\begin{aligned}
		&e_{\tilde{\mathbf{u}}}^{n+1}=\mathbf{u}^{n+1}-\tilde{\mathbf{u}}_h^{n+1}=\mathbf{u}^{n+1}-\boldsymbol{P}_h\mathbf{u}^{n+1}
			+\boldsymbol{P}_h\mathbf{u}^{n+1}-\tilde{\mathbf{u}}_h^{n+1}=\Phi_{\mathbf{u} h}^{n+1}+\Theta_{\tilde{\mathbf{u}} h}^{n+1},\\
		&e_\mathbf{u}^{n+1}=\mathbf{u}^{n+1}-\mathbf{u}_h^{n+1}=\mathbf{u}^{n+1}-\boldsymbol{P}_h\mathbf{u}^{n+1}
			+\boldsymbol{P}_h\mathbf{u}^{n+1}-\mathbf{u}_h^{n+1}=\Phi_{\mathbf{u} h}^{n+1}+\Theta_{\mathbf{u} h}^{n+1},\\
		&e_p^{n+1}=p^{n+1}-p_h^{n+1}=p^{n+1}-P_hp^{n+1}+P_hp^{n+1}-p_h^{n+1}=\Phi_{p h}^{n+1}+\Theta_{p h}^{n+1}.
	\end{aligned}
$$
For $\left(e_\phi^{n+1},e_\mu^{n+1},e_{\tilde{\mathbf{u}}}^{n+1},e_\mathbf{u}^{n+1},e_p^{n+1}\right)$ 
and the definition of Ritz projection \eqref{eq_Ritz_projection} and 
Stokes projection \eqref{eq_stokes_quasi_proojection_equation0001}-\eqref{eq_stokes_quasi_proojection_equation0002}, 
we subtract \eqref{eq_variational_forms_phi}-\eqref{eq_variational_forms_incompressible_condition} 
from \eqref{eq_fully_discrete_CHNS_scheme_phi}-\eqref{eq_fully_discrete_CHNS_scheme_tilde_incompressible_condition} to
obtain the following error equations,
\begin{flalign}
	\label{eq_error_equations_phi}
	\left(\delta_\tau e_\phi^{n+1},w_h\right)+M\left(\nabla e_\mu^{n+1},\nabla w_h\right)+b\left(\phi^{n+1},\mathbf{u}^{n+1},w_h\right)-b\left(\phi_h^{n+1},\mathbf{u}_h^n,w_h\right)
		-\left(R_1^{n+1},w_h\right)&=0,\\
	\label{eq_error_equations_mu}
	\left(e_\mu^{n+1},\varphi_h\right)-\lambda\left(\nabla e_\phi^{n+1},\nabla\varphi_h\right)
	-\frac{\lambda}{\epsilon^2}\left((\phi_h^{n+1})^3-(\phi^{n+1})^3-e_\phi^{n},\varphi_h\right)&=0,\\
	\label{eq_error_equations_ns}
	\left( \frac{e_{\tilde{\mathbf{u}}}^{n+1}-e_\mathbf{u}^n}{\tau},\mathbf{v}_h\right)+\nu\left(\nabla e_{\tilde{\mathbf{u}}}^{n+1},\nabla\mathbf{v}_h\right)
	+\left(\nabla e_p^{n},\mathbf{v}_h\right)
	+\left(\boldsymbol{B}\left(\mathbf{u}^{n+1},\mathbf{u}^{n+1},\mathbf{v}_h\right)-\boldsymbol{B}\left(\mathbf{u}_h^{n},\tilde{\mathbf{u}}_h^{n+1},\mathbf{v}_h\right)\right)\notag\\
	+\left(b\left(\phi_h^{n+1},\mathbf{v}_h,\mu_h^{n+1}\right)-b\left(\phi^{n+1},\mathbf{v}_h,\mu^{n+1}\right)\right)-
	\left(R_2^{n+1},\mathbf{v}_h\right)&=0,\\
	\label{eq_error_equations_u_p_cor}
	\frac{e_\mathbf{u}^{n+1}-e_{\tilde{\mathbf{u}}}^{n+1}}{\tau}+\nabla\left(e_p^{n+1}-e_p^{n}\right)-R_3^{n+1}&=0,\\
	\label{eq_error_equations_incompressible}
	\left(\nabla\cdot e_\mathbf{u}^{n+1},q_h\right)&=0,
\end{flalign}
where the truncation errors $R_1^{n+1}$, $R_2^{n+1}$  and $R_3^{n+1}$ is denoted by
\begin{equation}
	\begin{aligned}
		R_1^{n+1}=\delta_\tau R_h\phi^{n+1}-\partial_t\phi^{n+1},
		R_2^{n+1}=\delta_\tau\boldsymbol{P}_h\mathbf{u}^{n+1}-\partial_t\mathbf{u}^{n+1},
		R_3^{n+1}=\nabla\left(P_h p^{n+1}-P_h p^n\right),
	\end{aligned}
\end{equation}
and the  $\delta_\tau$ in time scheme is defined by \eqref{eq_d_t}.
Given the above error equations, we aim to derive the optimal $L^2$
error estimate for the CHNS system, which is outlined in the following main result.
\begin{Theorem}\label{theorem_error_estimates_u_phi_mu_r}
	Suppose that the model \eqref{eq_chns_equations}-\eqref{eq_boundary_initial_conditions_equations} has a
	unique solution $\left(\phi,\mu,\mathbf{u},p\right)$ 
	satisfying the regularities \eqref{eq_varibles_satisfied_regularities}. Then $\left(\phi_h^{n+1},\mu_h^{n+1},\mathbf{u}_h^{n+1},p_h^{n+1}\right)\in \mathbf{\mathcal{X}}_h^r $ is 
	unique solution in the fully discrete scheme 
	\eqref{eq_fully_discrete_CHNS_scheme_phi}-\eqref{eq_fully_discrete_CHNS_scheme_tilde_incompressible_condition}.
	for a positive constant $\tau_0$ such that $0<\tau\leq \tau_0$, the numerical solutions satisfy the error estimates as follows:
	\begin{flalign}
		\max_{0\leq n\leq N-1}\|\phi^{n+1}-\phi_h^{n+1}\|+\left(\tau\sum_{n=0}^{N-1}\|\mu^{n+1}-\mu_h^{n+1}\|^2\right)^{\frac{1}{2}}\leq C(\tau+h^{r+1}),&\\
		\max_{0\leq n\leq N-1}\|\mathbf{u}^{n+1}-\mathbf{u}_h^{n+1}\|\leq C(\tau+h^{r+1}),&\\
		\tau\sum_{n=0}^{N-1}\|\nabla\left(\mathbf{u}^{n+1}-\mathbf{u}_h^{n+1}\right)\|^2\leq C(\tau^2+h^{2(r+1)}),&\\
		\tau\sum_{n=0}^{N-1}\|p^{n+1}-p_h^{n+1}\|^2\leq C\left(\tau^2+h^{2r}\right),
	\end{flalign}
	where $C$ is a positive constant independent of  $\tau$ and $h$.
\end{Theorem}
To prove the theorem, we need the following lemmas.
\begin{Lemma}\label{lemma_truncation_errors}
	Under the assumption of regularity \eqref{eq_varibles_satisfied_regularities}, for any $\tau \geq 0$ and $n=0,1,2,\cdots, N-1$, the truncation errors satisfy
	\begin{equation}
		\tau\sum_{n=0}^{n}\left(\|R_1^{n+1}\|^2+\|R_2^{n+1}\|^2+\|R_3^{n+1}\|^2+\|R_4^{n+1}\|^2\right)\leq C\tau^2.
	\end{equation}
\end{Lemma}
The detailed proof of this Lemma can refer to \cite{2006_FengXiaobing_FullydiscretefiniteelementapproximationsoftheNavierStokesCahnHilliarddiffuseinterfacemodelfortwophasefluidflows,2015_Diegel_Analysis_of_a_mixed_finite_element_method_for_a_Cahn_Hilliard_Darcy_Stokes_system}. 
\begin{Lemma}\label{Lemma_0403}
	For any $h,\tau >0$, we have 
	\begin{equation}
		\|\nabla\left((\phi^{n+1})^3-(\phi_h^{n+1})^3\right)\|\leq C\|\nabla e_{\phi}^{n+1}\|.
	\end{equation}
	where $C$ is independent of $\tau$ and $h$, $~n=0,1,2,\cdots, N-1.$
\end{Lemma}
\begin{proof}
	According to \cite{2008_Kay_David_and_Welford_Richard_Finite_element_approximation_of_a_Cahn_Hilliard_Navier_Stokes_system,2015_Diegel_Analysis_of_a_mixed_finite_element_method_for_a_Cahn_Hilliard_Darcy_Stokes_system}, we have
	\begin{equation}
		\begin{aligned}
			\|\nabla\left((\phi^{n+1})^3-(\phi_h^{n+1})^3\right)\|&\leq3\|(\phi_h^{n+1})^2\nabla e_{\phi}^{n+1}\|+3\|\nabla\phi^{n+1}\left(\phi^{n+1}+\phi_h^{n+1}\right)e_{\phi}^{n+1}\|\\
			&\leq 3\|\phi_h^{n+1}\|_{L^{\infty}}^2\|\nabla e_{\phi}^{n+1}\|+
			3\|\nabla \phi^{n+1}\|_{L^6}\|\phi^{n+1}+\phi_h^{n+1}\|_{L^6}\|e_{\phi}^{n+1}\|_{L^6}\\
			&\leq 3\left(\|\phi_h^{n+1}\|_{L^{\infty}}^2+C\|\nabla \phi^{n+1}\|_{L^6}\|\phi^{n+1}+\phi_h^{n+1}\|_{H^1}\right)\|\nabla e_{\phi}^{n+1}\|\\
			&\leq C\|\nabla e_{\phi}^{n+1}\|.
		\end{aligned}
	\end{equation}
	Then, using the Lemma \ref{Lemma_0301} and the regularity \eqref{eq_varibles_satisfied_regularities}, we get the result.
\end{proof}
\begin{Lemma}
	\label{lemma_ephi_error}
	Under the assumption of regularity \eqref{eq_varibles_satisfied_regularities}, for all $m\geq 0$, there exists a positive constant that does not depend on $\tau$ and $h$ such that
	\begin{equation}
		\begin{aligned}
			\| \Theta_{\phi h}^{m+1}\|_{H^1}^2&+\tau\sum_{n=0}^{m}\|\nabla e_\mu^{n+1}\|_{H^1}^2\\
			\leq &~C\left(\tau^2+h^{2r}+\tau\sum_{n=0}^{m}\left(\|\phi^{n+1}\|_{H^2}^2
				+\|\nabla\mathbf{u}_h^n\|^2+\|\nabla\mathbf{u}_h^{n}\|^2\|\nabla\phi_h^{n+1}\|^2\right)\|\Theta_{\mathbf{u} h}^{n+1}\|^2\right).
		\end{aligned}
	\end{equation} 
\end{Lemma}
\begin{proof}
	To obtain estimates for $e_\phi^{n+1}$ and $e_\mu^{n+1}$, testing $w_h=2\tau\left(\Theta_{\mu h}^{n+1},\Theta_{\phi h}^{n+1}\right)$ and $\varphi_h=2\tau\left(\Theta_{\mu h}^{n+1},\delta_\tau \Theta_{\phi h}^{n+1}\right)$ in \eqref{eq_error_equations_phi} and \eqref{eq_error_equations_mu}, we get
	\begin{equation}
		\label{eq_testing_four_terms}
		\begin{aligned}
			&\|\Theta_{\phi h}^{n+1}\|_{H^1}^2-\|\Theta_{\phi h}^{n}\|_{H^1}^2+\|\Theta_{\phi h}^{n+1}-\Theta_{\phi h}^{n}\|_{H^1}^2
				+\tau\left(\|\Theta_{\phi h}^{n+1}\|_{H^1}^2+\|e_{\mu}^{n+1}\|_{H^1}^2\right)\\
			=&~\tau\|\Phi_{\mu h}\|_{H^1}^2+2\tau\left(R_1^{n+1},\Theta_{\mu h}^{n+1}+\Theta_{\phi h}^{n+1}\right)
				-2\tau\left(\mathbf{u}^{n+1}\cdot\nabla\phi^{n+1}-\mathbf{u}_h^n\cdot\nabla\phi_h^{n+1},\Theta_{\mu h}^{n+1}+\Theta_{\phi h}^{n+1}\right)\\
				&+2\tau\left((\phi^{n+1})^3-\phi^n-\left((\phi_h^{n+1})^3-\phi_h^n\right),\Theta_{\mu h}^{n+1}-\delta_\tau \Theta_{\phi h}^{n+1}\right)\\
					&-2\tau\left(\delta_\tau \Phi_{\phi h}^{n+1},\Theta_{\mu h}^{n+1}+\Theta_{\phi h}^{n+1}\right)
				    +2\tau\left( \Phi_{\mu h}^{n+1},\delta_\tau \Theta_{\phi h}^{n+1}\right)\\
			=&~\tau\|\Phi_{\mu h}^{n+1}\|_{H^1}^2+\sum_{n=1}^{5}T_i.
		\end{aligned}
	\end{equation}
We now estimate the right-hand side terms of \eqref{eq_testing_four_terms}.
\begin{equation}
	\label{eq_estimate_error_T1}
	\begin{aligned}
		\big|T_1\big|=\bigg|2\tau\left(R_1^{n+1},\Theta_{\mu h}^{n+1}+\Theta_{\phi h}^{n+1}\right)\bigg|
		\leq \frac{\tau}{7}\left(\|\Theta_{\mu h}^{n+1}\|_{H^1}^2+\|\Theta_{\phi h}\|_{H^1}^2\right)+C\tau\|R_1^{n+1}\|^2,
	\end{aligned}
\end{equation}
and
\begin{equation}
	\label{eq_estimate_error_T2}
	\begin{aligned}
		\big|T_2\big|=&~\bigg|-2\tau\left(\mathbf{u}^{n+1}\cdot\nabla\phi^{n+1}-\mathbf{u}_h^n\cdot\nabla\phi_h^{n+1},\Theta_{\mu h}^{n+1}+\Theta_{\phi h}^{n+1}\right)\bigg|\\
		= &~2\tau\bigg|\left(\left(\mathbf{u}^{n+1}-\mathbf{u}_h^{n+1}\right)\cdot\nabla\phi^{n+1},\Theta_{\mu h}^{n+1}+\Theta_{\phi h}^{n+1}\right)
		+\left(\mathbf{u}_h^n\cdot\nabla\left(\phi^{n+1}-\phi_h^{n+1}\right),\Theta_{\mu h}^{n+1}+\Theta_{\phi h}^{n+1}\right)\\
		&+\left(\left(\mathbf{u}_h^{n+1}-\mathbf{u}_h^n\right)\cdot\nabla\phi^{n+1},\Theta_{\mu h}^{n+1}+\Theta_{\phi h}^{n+1}\right)\bigg|\\
		=&~2\tau\bigg|\left(e_{\mathbf{u}}^{n+1}\cdot\nabla\phi^{n+1},\Theta_{\mu h}^{n+1}+\Theta_{\phi h}^{n+1}\right)
			+\left(\mathbf{u}_h^{n}\cdot\nabla e_{\phi}^{n+1},\Theta_{\mu h}^{n+1}+\Theta_{\phi h}^{n+1}\right)\\
			&+\left(\left(\mathbf{u}_h^{n+1}-\mathbf{u}_h^n\right)\cdot\nabla\phi^{n+1},\Theta_{\mu h}^{n+1}+\Theta_{\phi h}^{n+1}\right)
			\bigg|\\
		\leq&~\frac{\tau}{7}\left(\|\Theta_{\mu h}^{n+1}\|_{H^1}^2+\|\Theta_{\phi h}^{n+1}\|_{H^1}^2\right)
			+C\tau\|\phi^{n+1}\|_{H^2}^2\left(\|\Theta_{\mathbf{u} h}^{n+1}\|^2+\|\Phi_{\mathbf{u} h}^{n+1}\|^2\right)\\
			&+C\tau\|\nabla\mathbf{u}_h^{n}\|^2\left(\|\nabla\Phi_{\phi h}^{n+1}\|^2+\|\nabla\Theta_{\phi h}^{n+1}\|^2
			\right)+C\tau\left(\|\mathbf{u}_t\|_{L^{\infty}(0,T;L^2(\Omega))}^2\|\nabla\phi^{n+1}\|^2\right).
	\end{aligned}
\end{equation}
 We can recast  \eqref{eq_error_equations_phi} by the 
following
\begin{equation}
	\label{eq_delta_tau_Theta_phin1}
	\delta_\tau\Theta_{\phi h}^{n+1}=\Delta_he_{\mu}^{n+1}-\delta_\tau\Phi_{\phi h}^{n+1}
		-\left(\mathbf{u}^{n+1}\cdot\nabla\phi^{n+1}-\mathbf{u}_h^n\cdot\nabla\phi_h^{n+1}\right)
		+R_1^{n+1}.
\end{equation}
Using the \eqref{eq_basic_inequalities_001}, \eqref{eq_basic_inequalities_003}, Lemma \ref{lemma_stokes_projection001} 
and \ref{lemma_stokes_projection002}, we can obtain
\begin{equation}
	\label{eq_Delta_Theta_phih_n001}
	\begin{aligned}
		&\|\delta_\tau \Theta_{\phi h}^{n+1}\|_{H^{-1}}\leq C\|\nabla e_{\mu}^{n+1}\|+\|R_1^{n+1}\|\\
			&+C\tau\left(\|\phi^{n+1}\|_{H^2}+\|\nabla\mathbf{u}_h^n\|+\|\nabla\mathbf{u}_h^n\|\|\phi_h^{n+1}\|_{H^1}\right)
			\left(\|\Phi_{\mathbf{u} h}^{n+1}\|+\|\Theta_{\mathbf{u} h}^{n+1}\|+\|\Phi_{\phi h}^{n+1}\|_{H^1}+\|\Theta_{\phi h}^{n+1}\|_{H^1}\right).
	\end{aligned}
\end{equation}
According to the Lemma \ref{Lemma_0403} and the inequality \eqref{eq_phi3_error}, we estimate the right-hand side of \eqref{eq_testing_four_terms} as follows
\begin{equation}
	\label{eq_estimate_error_T3}
	\begin{aligned}
		\big|T_3\big|=&~2\tau\bigg|\left((\phi^{n+1})^3-\phi^n-\left((\phi_h^{n+1})^3-\phi_h^n\right),\Theta_{\mu h}^{n+1}-\delta_\tau \Theta_{\phi h}^{n+1}\right)\bigg|\\
		\leq &~C\tau\bigg|\left( (\phi^{n+1})^3-(\phi_h^{n+1})^3,\Theta_{\mu h}^{n+1}-\delta_\tau \Theta_{\phi h}^{n+1}\right)+\left(\phi^n-\phi_h^n,\Theta_{\mu h}^{n+1}-\delta_\tau \Theta_{\phi h}^{n+1}\right)\bigg|\\
		\leq &~C\tau\left(\|(\phi^{n+1})^3-(\phi_h^{n+1})^3\|\|\Theta_{\mu h}^{n+1}\|+\|\nabla\left((\phi^{n+1})^3-(\phi_h^{n+1})^3\right)\|\|\delta_\tau\Theta_{\phi h}^{n+1}\|_{H^{-1}}\right)\\
			&+C\tau\left(\|\phi^{n}-\phi_h^{n}\|\|\Theta_{\mu h}^{n+1}\|+\|\nabla\left(\phi^{n}-\phi_h^{n}\right)\|\|\delta_\tau\Theta_{\phi h}^{n+1}\|_{H^{-1}}\right)\\
		\leq&~C\tau\left(\|\Theta_{\mu h}^{n+1}\|^2+\|\nabla e_{\mu}^{n+1}\|^2+\|R_1^{n+1}\|^2\right)
			+C\tau\left(\|\phi^{n+1}\|_{H^2}^2+\|\nabla\mathbf{u}_h^n\|^2\right.\\
			&\left.+\|\nabla\mathbf{u}_h^n\|^2\|\phi_h^{n+1}\|_{H^1}^2\right)
			\left(\|\Theta_{\mathbf{u} h}^{n+1}\|^2+\|\Phi_{\mathbf{u} h}^{n+1}\|^2
			+\|\Theta_{\phi h}^{n+1}\|_{H^1}^2+\|\Phi_{\phi h}^{n+1}\|_{H^1}^2\right)+C h^{2(r+1)}.
	\end{aligned}
\end{equation}
Using the \eqref{eq_psi_Rhpsi_Hneg1_norm_inequation} and \eqref{eq_Delta_Theta_phih_n001}, we can obtain error estimates of $T_4$ and $T_5$ as follows:  
\begin{equation}
	\label{eq_estimate_error_T4}
	\begin{aligned}
		\big|T_4\big|=&~\big|-2\tau\left(\delta_\tau \Phi_{\phi h}^{n+1},\Theta_{\mu h}^{n+1}+\Theta_{\phi h}^{n+1}\right)\big|\\
		\leq &~2\tau\left(\|\Theta_{\mu h}^{n+1}\|_{H^1}+\|\Theta_{\phi h}^{n+1}\|_{H^1}\right)\|\delta_\tau \Phi_{\phi h}^{n+1}\|_{H^{-1}}\\
		\leq&~2\tau\left(\|\Theta_{\mu h}^{n+1}\|_{H^1}^2+\|\Theta_{\phi h}^{n+1}\|_{H^1}^2\right)+C h^{2(r+1)},
	\end{aligned}
\end{equation}
and
\begin{equation}
	\label{eq_estimate_error_T5}
	\begin{aligned}
		\big|T_5\big|=&\big| 2\tau\left( \Phi_{\mu h}^{n+1},\delta_\tau \Theta_{\phi h}^{n+1}\right)\big|\leq 2\tau\| \Phi_{\mu h}^{n+1}\|_{H^1}\|\delta_\tau \Theta_{\phi h}^{n+1}\|_{H^{-1}}\\
		\leq&~C\|\nabla e_{\mu}^{n+1}\|^2+C\tau\|R_1^{n+1}\|^2+C\tau\left(\|\phi^{n+1}\|_{H^2}^2+\|\nabla\mathbf{u}_h^n\|^2\right.\\
		&\left.+\|\nabla\mathbf{u}_h^n\|^2\|\phi_h^{n+1}\|_{H^1}^2\right)
		\left(\|\Phi_{\mathbf{u} h}^{n+1}\|^2+\|\Theta_{\mathbf{u} h}^{n+1}\|^2+\|\Phi_{\phi h}^{n+1}\|_{H^1}^2+\|\Theta_{\phi h}^{n+1}\|_{H^1}^2\right)
		+C h^{2(r+1)}.
	\end{aligned}
\end{equation}
Combining the inequalities $T_1,\cdots,T_5$ with \eqref{eq_testing_four_terms}, and
summing from $n=0$ to $m$, we can obtain
\begin{equation}
	\label{eq_T1_T5_error_estimates_results}
	\begin{aligned}
		\|\Theta_{\phi h}^{m+1}\|_{H^1}^2+&\frac{\tau}{2}\sum_{n=0}^{m}\|e_{\mu}^{n+1}\|_{H^1}^2\leq 
		\|\Theta_{\phi h}^{0}\|_{H^1}^2+C\tau\sum_{n=0}^{m}\|R_1^{n+1}\|^2+C\tau h^{2(r+1)}\\
		&+C\tau\sum_{n=0}^{m}\left(\|\phi^{n+1}\|_{H^2}^2+\|\nabla\mathbf{u}_h^n\|^2+\|\nabla\mathbf{u}_h^n\|^2\|\phi_h^{n+1}\|_{H^1}^2\right)
		\left(\|\Phi_{\mathbf{u} h}^{n+1}\|^2\right.\\
		&\left.\qquad+\|\Theta_{\mathbf{u} h}^{n+1}\|^2+\|\Phi_{\phi h}^{n+1}\|_{H^1}^2+\|\Theta_{\phi h}^{n+1}\|_{H^1}^2\right).
	\end{aligned}
\end{equation}
Using the \eqref{eq_psi_Rhpsi_Ls_norm_inequation}, and Lemma \ref{lemma_stokes_projection001}, \ref{Lemma_0301}, 
\ref{lemma_truncation_errors}, and \ref{lemma_discrete_Gronwall_inequation}, we can obtain the result of this Lemma. 
\end{proof}
\begin{Lemma}\label{lemma_u_p_L2_norm_boundness}
	Under the assumption of regularity \eqref{eq_varibles_satisfied_regularities}, for all $m\geq 0$,  there exists a positive constant that does not depend on $\tau$ and $h$ such that
	\begin{equation}
		\begin{aligned}
			\|\Theta_{\mathbf{u} h}^{m+1}\|^2&+\tau^2\|\nabla\Theta_{p h}^{m+1}\|^2+\tau\sum_{n=0}^{m}\|\nabla\Theta_{\tilde{\mathbf{u}} h}^{n+1}\|^2
			\leq C\left(\tau^2+h^{2r}+\frac{\tau}{4}\sum_{n=0}^{m}\|e_{\mu}^{n+1}\|_{H^1}^2\right)\\
			&+C\tau\sum_{n=0}^{m}\left(\|\phi^{n+1}\|_{H^2}^2+\|\mu_h^{n+1}\|_{H^1}^2+\|\mu_h^{n+1}\|_{H^1}^2\|\nabla\phi_h^{n+1}\|^2+\|\nabla\mathbf{u}_h^{n+1}\|^2\right.\\
			&\left.+\|\nabla\mathbf{u}_h^n\|^2\|\mathbf{u}_h^n\|^2
			\right)\|\Theta_{\phi h}^{n+1}\|_{H^1}^2.
		\end{aligned}
	\end{equation}
\end{Lemma}
\begin{proof}
	Recasting the projection step \eqref{eq_error_equations_u_p_cor}, we can obtain
	\begin{equation}
		\label{eq_recast_error_equation_projection_step}
		\Theta_{\mathbf{u} h}^{n+1}+\tau\nabla\Theta_{p h}^{n+1}=\Theta_{\tilde{\mathbf{u}} h}^{n+1}
			+\tau\nabla\Theta_{p h}^n+\tau\left(R_3^{n+1}-\tau\delta_{\tau}\nabla\Phi_{p h}^{n+1}\right).
	\end{equation}
Taking the inner product of \eqref{eq_recast_error_equation_projection_step} with itself on both sides, we can derive
\begin{equation}
	\label{eq_recast_error_equation_projection_step_squaring}
	\begin{aligned}
		\|\Theta_{\mathbf{u} h}^{n+1}\|^2+\tau^2\|\nabla\Theta_{p h}^{n+1}\|^2=&~\|\Theta_{\tilde{\mathbf{u}} h}^{n+1}\|^2+\tau^2\|\nabla\Theta_{p h}^n\|^2
			+\tau^2\|R_3^{n+1}-\tau\delta_{\tau}\nabla\Phi_{p h}^{n+1}\|^2\\
			&+2\tau\left(\Theta_{\tilde{\mathbf{u}} h}^{n+1},\nabla\Theta_{p h}^n\right)
			+2\tau\left(\Theta_{\tilde{\mathbf{u}} h}^{n+1},R_3^{n+1}-\tau\delta_{\tau}\nabla\Phi_{p h}^{n+1}\right)\\
			&+2\tau^2\left(\nabla\Theta_{p h}^n,R_3^{n+1}-\tau\delta_{\tau}\nabla\Phi_{p h}^{n+1}\right).
	\end{aligned}
\end{equation}
Testing $\mathbf{v}_h=2\tau\Theta_{\tilde{\mathbf{u}} h}^{n+1}$ in \eqref{eq_error_equations_ns}, and combining the \eqref{eq_recast_error_equation_projection_step_squaring}
and then using the Stokes projection \eqref{eq_stokes_quasi_proojection_equation0001}, we obtain
\begin{equation}
	\label{eq_testing_combining_using_the_above_equations}
	\begin{aligned}
		&\|\Theta_{\mathbf{u} h}^{n+1}\|^2-\|\Theta_{\mathbf{u} h}^n\|^2+\|\Theta_{\tilde{\mathbf{u}} h}^{n+1}-\Theta_{\mathbf{u} h}^n\|^2
			+\tau^2\left(\|\Theta_{p h}^{n+1}\|^2-\|\Theta_{p h}^n\|^2\right)+2\tau\|\nabla\Theta_{\tilde{\mathbf{u}} h}^{n+1}\|^2\\
		=&2\tau\left(\delta_{\tau}\Phi_{\mathbf{u} h}^{n+1},\Theta_{\tilde{\mathbf{u}} h}^{n+1}\right)
			+2\tau^2\left(\nabla\Theta_{p h}^{n+1},R_3^{n+1}-\tau\delta_{\tau}\nabla\Phi_{p h}^{n+1}\right)
			+2\tau\left(\Theta_{\tilde{\mathbf{u}} h}^{n+1},\nabla\left(\Phi_{p h}^{n+1}-\Phi_{p h}^{n}\right)\right)\\
			&+2\tau\left(\Theta_{\tilde{\mathbf{u}} h}^{n+1},R_3^{n+1}-\tau\delta_{\tau}\nabla\Phi_{p h}^{n+1}\right)
			+\tau^2\|R_3^{n+1}-\tau\delta_{\tau}\nabla\Phi_{p h}^{n+1}\|^2+2\tau\left(R_2^{n+1},\Theta_{\tilde{\mathbf{u}} h}^{n+1}\right)\\
			&+2\tau\left(\mu^{n+1}\nabla\phi^{n+1}-\mu_h^{n+1}\nabla\phi_h^{n+1},\Theta_{\tilde{\mathbf{u}} h}^{n+1}\right)
				-2\tau\left(\mathbf{u}^{n+1}\cdot\nabla\mathbf{u}^{n+1}-\mathbf{u}_h^n\cdot\nabla\tilde{\mathbf{u}}_h^{n+1},\Theta_{\tilde{\mathbf{u}} h}^{n+1}\right).
	\end{aligned}
\end{equation}
Now, we need to estimate each term on the right-hand sides of \eqref{eq_testing_combining_using_the_above_equations}.
\begin{equation}
	\label{eq_error_estimates_right_term0001}
	\begin{aligned}
		\big|2\tau\left(\delta_{\tau}\Phi_{\mathbf{u} h}^{n+1},\Theta_{\tilde{\mathbf{u}} h}^{n+1}\right)\big|
		\leq \frac{\tau}{8}\|\nabla\Theta_{\tilde{\mathbf{u}} h}^{n+1}\|^2+C\tau\|\delta_{\tau}\Phi_{\mathbf{u} h}^{n+1}\|^2,
	\end{aligned}
\end{equation}
\begin{equation}
	\label{eq_error_estimates_right_term0002}
	\begin{aligned}
		&\bigg|2\tau^2\left(\nabla\Theta_{p h}^{n+1},R_3^{n+1}-\tau\delta_{\tau}\nabla\Phi_{p h}^{n+1}\right)+2\tau\left(\Theta_{\tilde{\mathbf{u}} h}^{n+1},R_3^{n+1}-\tau\delta_{\tau}\nabla\Phi_{p h}^{n+1}\right)\bigg|\\
		\leq &~C\tau^3\|\nabla\Theta_{p h}^{n+1}\|^2+C\tau\|R_3^{n+1}-\tau\delta_{\tau}\nabla\Phi_{p h}^{n+1}\|^2+\frac{\tau}{8}\|\nabla\Theta_{\tilde{\mathbf{u}} h}^{n+1}\|^2\\
		\leq&~C\tau^3\|\nabla\Theta_{p h}^{n+1}\|^2+\frac{\tau}{8}\|\nabla\Theta_{\tilde{\mathbf{u}} h}^{n+1}\|^2+C\tau\left(\|R_3^{n+1}\|^2+\tau^2\|\nabla\delta_{\tau}\Phi_{p h}^{n+1}\|^2\right),
	\end{aligned}
\end{equation}
\begin{equation}
	\label{eq_error_estimates_right_term0003}
	\begin{aligned}
		\bigg|2\tau\left(\Theta_{\tilde{\mathbf{u}} h}^{n+1},\nabla\left(\Phi_{p h}^{n+1}-\Phi_{p h}^{n}\right)\right)\bigg|
		=&~\bigg|2\tau^2\left(\delta_{\tau}\Phi_{p h}^{n+1},\nabla\cdot\Theta_{\tilde{\mathbf{u}} h}^{n+1}\right)\bigg|
		\leq C\tau^3\|\delta_{\tau}\Phi_{p h}^{n+1}\|^2+\frac{\tau}{8}\|\nabla\Theta_{\tilde{\mathbf{u}} h}^{n+1}\|^2,
	\end{aligned}
\end{equation}
\begin{equation}
	\label{eq_error_estimates_right_term0004}
	\begin{aligned}
		\bigg|\tau^2\|R_3^{n+1}-\tau\delta_{\tau}\nabla\Phi_{p h}^{n+1}\|^2\bigg|
		\leq\tau^2\|R_3^{n+1}\|^2+C\tau^4\|\delta_{\tau}\nabla\Phi_{p h}^{n+1}\|^2,
	\end{aligned}
\end{equation}
\begin{equation}
	\bigg|2\tau\left(R_2^{n+1},\Theta_{\tilde{\mathbf{u}} h}^{n+1}\right)\bigg|
	\leq C\tau\|R_2^{n+1}\|^2+\frac{\tau}{8}\|\nabla\Theta_{\tilde{\mathbf{u}} h}^{n+1}\|^2,
\end{equation}
\begin{equation}
	\label{eq_error_estimates_right_term0005}
	\begin{aligned}
		&\bigg|2\tau\left(\mu^{n+1}\nabla\phi^{n+1}-\mu_h^{n+1}\nabla\phi_h^{n+1},\Theta_{\tilde{\mathbf{u}} h}^{n+1}\right)\bigg|
		=2\tau\bigg|\left(e_{\mu}^{n+1}\nabla\phi^{n+1}+\mu_h^{n+1}\nabla e_{\phi}^{n+1},\Theta_{\tilde{\mathbf{u}} h}^{n+1}\right)\bigg|\\
		\leq&~2\tau\bigg|\left(e_{\mu}\nabla\phi^{n+1},\Theta_{\mathbf{u} h}^{n+1}+\tau\nabla\left(\Theta_{p h}^{n+1}-\Theta_{p h}^n\right)
			-\tau\left(R_3^{n+1}-\tau\delta_{\tau}\nabla\Phi_{p h}^{n+1}\right)\right)\bigg|\\
			&+2\tau\bigg|\left(\mu_h^{n+1}\nabla e_{\phi}^{n+1},\Theta_{\tilde{\mathbf{u}} h}^{n+1}\right)\bigg|\\
		\leq &~2\tau\|e_{\mu}^{n+1}\|_{L^4}\|\nabla\phi^{n+1}\|_{L^4}\left(\|\Theta_{\mathbf{u} h}^{n+1}\|+\tau\left(\|\nabla\Theta_{p h}^{n+1}\|+\|\nabla\Theta_{p h}^n\|\right)
		+\tau\|R_3^{n+1}-\tau\delta_{\tau}\nabla\Phi_{p h}^{n+1}\|\right)\\
		&+2\tau\|\mu_h^{n+1}\|_{L^4}\|\nabla e_{\phi}^{n+1}\|\|\Theta_{\tilde{\mathbf{u}} h}^{n+1}\|_{L^4}\\
		\leq &~\frac{\tau}{8}\|\nabla\Theta_{\tilde{\mathbf{u}} h}^{n+1}\|^2+\frac{\tau}{4}\|e_\mu^{n+1}\|_{H^1}^2\\
		&+C\tau^3\|\phi^{n+1}\|_{H^2}^2
			\left(\|\nabla\Theta_{p h}^{n+1}\|^2+\|\nabla\Theta_{p h}^n\|^2+\|R_3^{n+1}-\tau\delta_{\tau}\nabla\Phi_{p h}^{n+1}\|^2\right)\\
			&+C\tau\left(\|\mu_h^{n+1}\|_{H^1}^2+\|\phi^{n+1}\|_{H^2}^2\right)\left(\|\nabla\Theta_{\phi h}^{n+1}\|^2+\|\nabla \Phi_{\phi h}^{n+1}\|^2+\|\nabla\Theta_{\mathbf{u} h}^{n+1}\|^2\right),
	\end{aligned}
\end{equation}
and
\begin{equation}
	\label{eq_error_estimates_right_term0006}
	\begin{aligned}
		&\bigg|-2\tau\left(\mathbf{u}^{n+1}\cdot\nabla\mathbf{u}^{n+1}-\mathbf{u}_h^n\cdot\nabla\tilde{\mathbf{u}}_h^{n+1},\Theta_{\tilde{\mathbf{u}} h}^{n+1}\right)\bigg|\\
		=&~2\tau\bigg|\left(\mathbf{u}^{n+1}\cdot\nabla\mathbf{u}^{n+1}-\mathbf{u}^n\cdot\nabla\mathbf{u}^n
			+e_{\mathbf{u}}^n\cdot\nabla\mathbf{u}^n+\mathbf{u}_h^n\cdot\nabla e_{\tilde{\mathbf{u}}}^{n+1}+\mathbf{u}_h^n\cdot\nabla\left(\mathbf{u}^n-\mathbf{u}^{n+1}\right),\Theta_{\tilde{\mathbf{u}} h}^{n+1}\right)\bigg|.
	\end{aligned}
\end{equation}
According to the Taylor expansion and the inequatilies \eqref{eq_basic_inequalities_001}-\eqref{eq_basic_inequalities_004}, we obtain
\begin{equation}
	\label{eq_error_estimates_right_term0007}
	\begin{aligned}
		&\bigg|2\tau\left(\mathbf{u}^{n+1}\cdot\nabla\mathbf{u}^{n+1}-\mathbf{u}^n\cdot\nabla\mathbf{u}^n,\Theta_{\tilde{\mathbf{u}} h}^{n+1}\right)\bigg|\\
		= &~2\tau\bigg|\left(\left(\mathbf{u}^{n+1}-\mathbf{u}^n\right)\cdot\nabla\mathbf{u}^{n+1}+\mathbf{u}^n\cdot\nabla\left(\mathbf{u}^{n+1}-\mathbf{u}^n\right),\Theta_{\tilde{\mathbf{u}} h}^{n+1}\right)\bigg|\\
		\leq&~\frac{\tau}{8}\|\nabla\Theta_{\tilde{\mathbf{u}} h}^{n+1}\|^2+C\tau^3\|\mathbf{u}_t\|_{L^{\infty}(0,T;\mathbf{H}^1(\Omega))}^2\left(\|A_h\mathbf{u}^{n+1}\|^2+\|A_h\mathbf{u}^{n}\|^2\right),
	\end{aligned}
\end{equation}
and
\begin{equation}
	\label{eq_error_estimates_right_term0008}
	\begin{aligned}
		&\bigg|2\tau\left(e_{\mathbf{u}}^n\cdot\nabla\mathbf{u}^n+\mathbf{u}_h^n\cdot\nabla e_{\tilde{\mathbf{u}}}^{n+1},\Theta_{\tilde{\mathbf{u}} h}^{n+1}\right)\bigg|\\
		\leq&~2\tau\|\nabla\Theta_{\tilde{\mathbf{u}} h}^{n+1}\|\left(\|e_{\mathbf{u}}^n\|_{L^4}\|\nabla\mathbf{u}^n\|_{L^4} +\|\nabla\mathbf{u}_h^n\|_{L^4}\|e_{\tilde{\mathbf{u}}}^{n+1}\|_{L^4}\right)\\
		\leq&~\frac{\tau}{8}\|\nabla\Theta_{\tilde{\mathbf{u}} h}^{n+1}\|^2+C\tau\left(\|\nabla\mathbf{u}^n\|^2+\|\nabla\mathbf{u}_h^n\|^2 \right)\left(\|\Phi_{\mathbf{u} h}^n\|^2+\|\Theta_{\mathbf{u} h}^n\|^2
				+\|\Phi_{\mathbf{u} h}^{n+1}\|^2+\|\Theta_{\tilde{\mathbf{u}} h}^{n+1}\|^2\right),
	\end{aligned}
\end{equation}
and
\begin{equation}
	\label{eq_error_estimates_right_term0009}
	\begin{aligned}
		&\bigg|2\tau\left(\mathbf{u}_h^n\cdot\nabla\left(\mathbf{u}^n-\mathbf{u}^{n+1}\right),\Theta_{\tilde{\mathbf{u}} h}^{n+1}\right)\bigg|
		\leq 2\tau\|\mathbf{u}_h^n\|_{L^4}\|\nabla\left(\mathbf{u}^n-\mathbf{u}^{n+1}\right)\|_{L^4}\|\nabla\Theta_{\tilde{\mathbf{u}} h}^{n+1}\|\\
		\leq&~\frac{\tau}{8}\|\nabla\Theta_{\tilde{\mathbf{u}} h}^{n+1}\|^2+C\tau\|\mathbf{u}_h^n\|^2\left(\|\nabla\mathbf{u}^{n+1}\|^2+\|\nabla\mathbf{u}^n\|^2\right).
	\end{aligned}
\end{equation}
Combining the above inequalities \eqref{eq_error_estimates_right_term0001}-\eqref{eq_error_estimates_right_term0009} with \eqref{eq_testing_combining_using_the_above_equations},
and summing up for $n$ from $0$ to $m$, we obtain
\begin{equation}
	\begin{aligned}
		&\|\Theta_{\mathbf{u} h}^{m+1}\|^2
		+\tau^2\|\Theta_{p h}^{m+1}\|^2+\tau\sum_{n=0}^{m}\|\nabla\Theta_{\tilde{\mathbf{u}} h}^{n+1}\|^2\\
		\leq &~\|\Theta_{\mathbf{u} h}^{0}\|^2+\tau^2\|\Theta_{p h}^{0}\|^2
			+C\tau\sum_{n=0}^{m}\left(\|\delta_{\tau}\Phi_{\mathbf{u} h}^{n+1}\|^2+\|R_2^{n+1}\|^2+\|R_3^{n+1}\|^2\right)\\
			&+\frac{\tau}{}\sum_{n=0}^{m}\|e_\mu^{n+1}\|_{H^1}^2+C\tau^3\sum_{n=0}^{m}\left(\|\nabla\Theta_{p h}^{n+1}\|^2+\|\nabla\delta_{\tau}\Phi_{p h}^{n+1}\|^2
				+\|\delta_{\tau}\Phi_{p h}^{n+1}\|^2 \right)\\
			&+C\tau\sum_{n=0}^{m}\left(\|\mu_h^{n+1}\|_{H^1}^2+\|\phi^{n+1}\|_{H^2}^2\right)\left(\|\nabla\Theta_{\phi h}^{n+1}\|^2+\|\Phi_{\phi h}^{n+1}\|^2+\|\nabla\Theta_{\mathbf{u} h}^{n+1}\|^2\right)\\
			&+C\tau\sum_{n=0}^{m}\left(\|\nabla\mathbf{u}^n\|^2+\|\nabla\mathbf{u}_h^n\|^2\right)\left(\|\Phi_{\mathbf{u} h}^n\|^2+\|\Theta_{\mathbf{u} h}^n\|^2
				+\|\Phi_{\mathbf{u} h}^{n+1}\|^2+\|\Theta_{\tilde{\mathbf{u}} h}^{n+1}\|^2\right)\\
			&+C\tau\sum_{n=0}^{m}\|\mathbf{u}_h^n\|^2\left(\|\nabla\mathbf{u}^{n+1}\|^2+\|\nabla\mathbf{u}^n\|^2\right)\\
			&+C\tau^3\sum_{n=0}^{m}\|\phi^{n+1}\|_{H^2}^2\left(\|\nabla\Theta_{p h}^{n+1}\|^2+\|\nabla\Theta_{p h}^n\|^2+\|\mathbf{u}_t\|_{L^\infty(0,T;\mathbf{H}^1(\Omega))}^2\left(\|A_h\mathbf{u}^{n+1}\|^2+\|A_h\mathbf{u}^{n}\|^2\right)\right).
	\end{aligned}
\end{equation}
According to the Lemma \ref{lemma_stokes_projection001} and \ref{lemma_stokes_projection002}, we can derive
\begin{equation}
	\tau^3\sum_{n=0}^{m}\|\delta_{\tau}\Phi_{p h}^{n+1}\|_{H^1}^2\leq C\tau^2\int_{0}^{T}\left(\|\mathbf{u}_t\|_{H^2}^2+\|p_t\|_{H^1}^2\right)dt,
\end{equation}
and
\begin{equation}
	\tau\sum_{n=0}^{m}\|\delta_{\tau}\Phi_{\mathbf{u} h}^{n+1}\|^2\leq Ch^{2r}\int_{0}^{T}\left(\|\mathbf{u}_t\|_{H^{r+1}}^2+\|p_t\|_{H^r}^2\right)dt.
\end{equation}
Finally, using the Lemma \ref{lemma_stokes_projection001}, \ref{Lemma_0301}, \ref{lemma_boundness_phi_mu}, \ref{lemma_truncation_errors}, 
 the inequalities \eqref{eq_psi_Rhpsi_Ls_norm_inequation}-\eqref{eq_Dtau_psi_Rhpsi_Hneg1_norm_inequation} and the discrete Gr\"{o}nwall's lemma \ref{lemma_discrete_Gronwall_inequation},
  we can ontain the proof of this lemma.
\end{proof}
Therefore, we can obtain the following theorem based on the above lemmas.
\begin{Theorem}\label{theorem_e_phi_L2}
	Under the assumption of regularity \eqref{eq_varibles_satisfied_regularities}, for all $m\geq 0$, there exists a positive constant that does not depend on $\tau$ and $h$ such that
	\begin{equation}
		\|e_{\phi}^{m+1}\|_{H^1}^2+\|e_{\mathbf{u}}^{m+1}\|^2+\tau^2\|\nabla e_{p}^{m+1}\|^2+\tau\sum_{n=0}^{m}\left(\|e_{\mu}^{n+1}\|_{H^1}^2+\|\nabla e_{\tilde{\mathbf{u}}}^{n+1}\|^2\right)\leq C\left(\tau^2+h^{2r}\right).
	\end{equation}
\end{Theorem}
Next, in order to prove the optimal $L^2$ error estimates for the numerical solutions of the fully discrete scheme 
\eqref{eq_fully_discrete_CHNS_scheme_phi}-\eqref{eq_fully_discrete_CHNS_scheme_tilde_incompressible_condition},
we need to prove the following lemma.
\begin{Lemma}\label{lemma_L2_e_phi_u_mu}
	Under the assumption of regularity \eqref{eq_varibles_satisfied_regularities}, for all $m\geq 0$, there exists a positive constant that does not depend on $\tau$ and $h$ such that
	\begin{equation}
		\|e_\phi^{m+1}\|^2+\sum_{n=0}^{m}\|e_{\mu}^{n+1}\|^2\leq C\left(\tau^2+h^{2(r+1)}\right),
	\end{equation}
and
\begin{equation}
	\|e_{\mathbf{u}}^{m+1}\|^2\leq C\left(\tau^2+h^{2(r+1)}\right).
\end{equation}
\end{Lemma}
\begin{proof}
	Testing $w_h=2\tau\Delta_h^{-1}\Theta_{\mu h}^{n+1}$,
	$\varphi_h=\delta_{\tau}\Delta_h^{-1}\Theta_{\phi h}^{n+1}$
	and $\mathbf{v}_h=2\tau\Theta_{\tilde{\mathbf{u}} h}^{n+1}$ in 
	the equations \eqref{eq_error_equations_phi}-\eqref{eq_error_equations_ns},
	and using \eqref{eq_discrete_Laplacian_operator}-\eqref{eq_discrete_Laplacian_operator11} and
	\eqref{eq_recast_error_equation_projection_step_squaring}, we have
	\begin{equation}
		\label{eq_combining_additional00001}
		\begin{aligned}
			&\|\Theta_{\phi h}^{n+1}\|^2-\|\Theta_{\phi h}^{n}\|^2+\|\Theta_{\phi h}^{n+1}-\Theta_{\phi h}^{n}\|^2
			+\tau\left(\|\Theta_{\mu h}^{n+1}\|^2+\|e_{\mu}^{n+1}\|^2\right)\\
			=&~\tau\|\Phi_{\mu h}^{n+1}\|^2+2\tau\left(\Phi_{\phi h}^{n+1},\delta_{\tau}\Theta_{\phi h}^{n+1}\right)
			+2\tau\left(R_1^{n+1},\Delta_h^{-1}\Theta_{\mu h}^{n+1}\right)\\
			&+2\tau\left((\phi_h^{n+1})^3-(\phi^{n+1})^3-e_{\phi}^n,\delta_{\tau}\Delta_h^{-1}\Theta_{\phi h}^{n+1}\right)\\
			=&~\tau\|\Phi_{\mu h}^{n+1}\|^2+\sum_{i=1}^{3}Tb_i,
		\end{aligned}
	\end{equation}
and 
\begin{equation}
	\begin{aligned}
		&\|\Theta_{{\mathbf{u}} h}^{n+1}\|^2-\|\Theta_{\mathbf{u} h}^n\|^2+\|\Theta_{\tilde{\mathbf{u}} h}^{n+1}-\Theta_{\mathbf{u} h}^n\|^2+\tau^2\left(\|\nabla\Theta_{p h}^{n+1}\|^2-\|\nabla\Theta_{p h}^n\|^2\right)+2\tau\|\nabla\Theta_{\tilde{\mathbf{u}} h}^{n+1}\|^2\\
		=&~	2\tau\left(\delta_{\tau}\Phi_{\mathbf{u} h}^{n+1},\Theta_{\tilde{\mathbf{u}} h}^{n+1}\right)+2\tau\left(R_2^{n+1},\Theta_{\tilde{\mathbf{u}} h}^{n+1}\right)
		+2\tau^2\left(\nabla\Theta_{p h}^{n+1},R_3^{n+1}-\tau\delta_{\tau}\nabla\Phi_{p h}^{n+1}\right)\\
		&+2\tau\left(\Theta_{\tilde{\mathbf{u}} h}^{n+1},\nabla\left(\Phi_{p h}^{n+1}-\Phi_{p h}^{n}\right)\right)
		+2\tau\left(\Theta_{\tilde{\mathbf{u}} h}^{n+1},R_3^{n+1}-\tau\delta_{\tau}\nabla\Phi_{p h}^{n+1}\right)
		+\tau^2\|R_3^{n+1}-\tau\delta_{\tau}\nabla\Phi_{p h}^{n+1}\|^2\\
		&+2\tau\left(\mathbf{u}^{n+1}\cdot\nabla\phi^{n+1}-\mathbf{u}_h^n\cdot\nabla\phi_h^{n+1},\Delta_h^{-1}\Theta_{\mu h}^{n+1}\right)\\
		&+2\tau\left(\mu^{n+1}\nabla\phi^{n+1}-\mu_h^{n+1}\nabla\phi_h^{n+1},\Theta_{\tilde{\mathbf{u}} h}^{n+1}\right)
		-2\tau\left(\mathbf{u}^{n+1}\cdot\nabla\mathbf{u}^{n+1}-\mathbf{u}_h^n\cdot\nabla\tilde{\mathbf{u}}_h^{n+1},\Theta_{\tilde{\mathbf{u}} h}^{n+1}\right)\\
		=&~\sum_{i=4}^{12}Tb_i.
	\end{aligned}
\end{equation}
We now need to estimate the right-hand side terms $Tb_1,\cdots,Tb_{12}$ of \eqref{eq_combining_additional00001}.
\begin{equation}
	\label{eq_boundness_Tb1_2}
	\begin{aligned}
		\bigg|Tb_1+Tb_2\bigg|=&~\bigg|2\tau\left(\Phi_{\phi h}^{n+1},\delta_{\tau}\Theta_{\phi h}^{n+1}\right)+2\tau\left(R_1^{n+1},\Delta_h^{-1}\Theta_{\mu h}^{n+1}\right)\bigg|\\
		\leq &~C\tau\left(\|\Phi_{\phi h}^{n+1}\|^2+\tau\|\delta_{\tau}\Theta_{\phi h}^{n+1}\|^2+\frac{1}{2}\|R_1^{n+1}\|^2+\frac{1}{2}\|\Theta_{\mu h}^{n+1}\|^2\right).
	\end{aligned}
\end{equation}
According to the inequalities \eqref{eq_delta_tau_Theta_phin1}, \eqref{eq_psi_Rhpsi_Hneg1_norm_inequation} and the fact that 
$\|\cdot\|_{H^{-2}}\leq \|\cdot\|_{H^{-1}}$, we obtain
\begin{equation}
\begin{aligned}
	\|\delta_{\tau}\Theta_{\phi h}^{n+1}\|_{H^{-2}}\leq&~\|\Delta_h e_{\mu}^{n+1}\|_{H^{-2}}+\|\delta_{\tau}\Phi_{\phi h}^{n+1}\|_{H^{-1}}
	+\|\mathbf{u}^{n+1}\cdot\nabla\phi^{n+1}-\mathbf{u}_h^n\cdot\nabla\phi_h^{n+1}\|_{H^{-1}}\\
	\leq&~\tau\|e_{\mu}^{n+1}\|+\frac{C\tau}{2}\|R_1^{n+1}\|+C\tau\left(\|\mathbf{u}_t\|_{L^{\infty}(0,T;\mathbf{H}^1(\Omega))}\|\nabla\phi^{n+1}\|
	\right)\\
	&+C\tau\|\nabla\phi^{n+1}\|\left(\|\Phi_{\mathbf{u} h}^{n+1}\|+\|\Theta_{\mathbf{u} h}^{n+1}\|\right)
	+C\tau\|\nabla\mathbf{u}_h^{n+1}\|\left(\|\Phi_{\phi h}^{n+1}\|+\|\Theta_{\phi h}^{n+1}\|\right).
\end{aligned}
\end{equation}
Using the Lemma \ref{Lemma_0403}, the inequalities \eqref{eq_phi3_error} and
\eqref{eq_estimate_error_T3}, the term $Tb_9$ is estimated by
\begin{equation}
	\label{eq_boundness_Tb9}
	\begin{aligned}
		\bigg|Tb_3\bigg|=&~\bigg|2\tau\left((\phi_h^{n+1})^3-(\phi^{n+1})^3-e_{\phi}^n,\delta_{\tau}\Delta_h^{-1}\Theta_{\phi h}^{n+1}\right)\bigg|\\
		\leq&~ 2\tau\bigg|\left(\phi_h^{n+1})^3-(\phi^{n+1})^3,\delta_{\tau}\Delta_h^{-1}\Theta_{\phi h}^{n+1}\right)\bigg|
		+2\tau\bigg|\left(e_{\phi}^n,\delta_{\tau}\Delta_h^{-1}\Theta_{\phi h}^{n+1}\right)\bigg|\\
		\leq&~C\left(\|\phi_h^{n+1})^3-(\phi^{n+1})^3\|+\|e_{\phi}^n\|\right)\|\delta_{\tau}\Theta_{\phi h}^{n+1}\|_{H^{-2}}\\
		\leq&~\frac{\tau}{2}\|e_{\mu}^{n+1}\|^2+\frac{C\tau}{2}\|R_1^{n+1}\|^2+C\tau^3\|\mathbf{u}_t\|_{L^\infty(0,T;\mathbf{H}^1(\Omega))}^2\|\nabla\phi^{n+1}\|^2\\
		&+C\tau\left(\|\nabla\phi^{n+1}\|^2+\|\nabla\mathbf{u}_h^n\|^2
		+\|\Theta_{\phi h}^{n}\|^2+\|\Phi_{\phi h}^{n}\|^2\right)
		\left(\|\Theta_{\mathbf{u} h}^{n+1}\|^2+\|\Phi_{\mathbf{u} h}^{n+1}\|^2\right.\\
		&\left.
		+\|\Theta_{\phi h}^{n+1}\|^2+\|\Phi_{\phi h}^{n+1}\|^2\right)
		+C h^{2(r+1)}.
	\end{aligned}
\end{equation}
Using a similar proof to that of the Lemma \ref{lemma_u_p_L2_norm_boundness}, it can easily be proved that
\begin{equation}
	\label{eq_boundness_sum_Tb3_8}
	\begin{aligned}
		\bigg|\sum_{i=4}^{9}Tb_i\bigg|\leq&~ C\tau\left(\frac{1}{3}\|\nabla\Theta_{\tilde{\mathbf{u}} h}^{n+1}\|^2+\|\delta_{\tau}\Phi_{\mathbf{u} h}^{n+1}\|^2
		+\|R_2^{n+1}\|^2\right)\\
		&+C\tau\left(\tau^2\|\nabla\Theta_{p h}^{n+1}\|^2+\|R_3^{n+1}\|^2+\tau^3\|\delta_{\tau}\nabla\Phi_{p h}^{n+1}\|^2\right)\\
		&+C\tau^3\left(\|\delta_{\tau}\Phi_{p h}^{n+1}\|_{H^1}^2
			+\|\delta_{\tau}\nabla\Phi_{p h}^{n+1}\|^2\right).
	\end{aligned}
\end{equation}
We next estimate the three terms $Tb_{10}$, $Tb_{11}$ and $Tb_{12}$ as follows
\begin{equation}
	\begin{aligned}
		\bigg|Tb_{10}\bigg|=&~\bigg|2\tau\left(\mathbf{u}^{n+1}\cdot\nabla\phi^{n+1}-\mathbf{u}_h^n\cdot\nabla\phi_h^{n+1},\Delta_h^{-1}\Theta_{\mu h}^{n+1}\right)\bigg|\\
		\leq &~2\tau\left(e_{\mathbf{u}}^{n+1}\cdot\nabla\phi^{n+1}+\mathbf{u}_h^n\cdot\nabla e_{\phi}^{n+1}+\left(\mathbf{u}_h^{n+1}-\mathbf{u}_h^n\right)\cdot\nabla\phi^{n+1},\Delta_h^{-1}\Theta_{\mu h}^{n+1}\right)\\
		\leq &~\frac{C\tau}{4}\|\Theta_{\mu h}^{n+1}\|^2+C\tau\left(\|\phi^{n+1}\|_{H^1}^2+\|\nabla\mathbf{u}_h^n\|^2+\|\mathbf{u}_t\|_{L^{\infty}(0,T;\mathbf{H}^1(\Omega))}^2\|\phi^{n+1}\|_{H^1}^2\right),
	\end{aligned}
\end{equation}
and
\begin{equation}
	\begin{aligned}
		\bigg|Tb_{11}\bigg|=&~\bigg|2\tau\left(\mu^{n+1}\nabla\phi^{n+1}-\mu_h^{n+1}\nabla\phi_h^{n+1},\Theta_{\tilde{\mathbf{u}} h}^{n+1}\right)\bigg|\\
		=&~2\tau\bigg|\left(e_{\mu}^{n+1}\nabla\phi^{n+1}+\mu_h^{n+1}\nabla e_{\phi}^{n+1},\Theta_{\tilde{\mathbf{u}} h}^{n+1}\right)\bigg|\\
		\leq&~2\tau\|e_{\mu}^{n+1}\|_{L^4}\|\phi^{n+1}\|_{L^4}\|\nabla\Theta_{\tilde{\mathbf{u}} h}^{n+1}\|
			+2\tau\|\mu_h^{n+1}\|_{L^4}\|e_{\phi}^{n+1}\|_{L^4}\|\nabla\Theta_{\tilde{\mathbf{u}} h}^{n+1}\|\\
		\leq&~\frac{\tau}{3}\|\nabla\Theta_{\tilde{\mathbf{u}} h}^{n+1}\|^2+\frac{\tau}{4}\|e_{\mu}^{n+1}\|^2+C\tau^3\left(\|\mu_h^{n+1}\|_{H^1}^2+\|\phi^{n+1}\|_{H^1}^2\right)\left(\|\Phi_{\phi h}^{n+1}\|^2+\|\Theta_{\phi h}^{n+1}\|^2\right),
	\end{aligned}
\end{equation}
and
\begin{equation}
	\begin{aligned}
		\bigg|Tb_{12}\bigg|=&~\bigg|-2\tau\left(\mathbf{u}^{n+1}\cdot\nabla\mathbf{u}^{n+1}-\mathbf{u}_h^n\cdot\nabla\tilde{\mathbf{u}}_h^{n+1},\Theta_{\tilde{\mathbf{u}} h}^{n+1}\right)\bigg|\\
		=&~2\tau\left(e_{\mathbf{u}}^{n+1}\cdot\nabla \mathbf{u}^{n+1}+\mathbf{u}_h^n\cdot\nabla e_{\mathbf{u}}^{n+1}+\mathbf{u}_h^n\cdot\nabla e_{\tilde{\mathbf{u}}}^{n+1}+\left(\mathbf{u}_h^{n+1}-\mathbf{u}_h^n\right)\cdot\nabla\mathbf{u}^{n+1},\Theta_{\tilde{\mathbf{u}} h}^{n+1}\right)\\
		\leq&~\frac{\tau}{3}\|\nabla\Theta_{\tilde{\mathbf{u}} h}^{n+1}\|^2+C\tau\left(\|\nabla\mathbf{u}^{n+1}\|^2+\|\nabla\mathbf{u}_h^n\|^2\right)\left(\|\Theta_{{\mathbf{u}} h}^{n+1}\|^2+\|\Phi_{\mathbf{u} h}^{n+1}\|^2\right)\\
				&+C\tau\|\nabla\mathbf{u}_h^n\|^2\left(\|\Theta_{\tilde{\mathbf{u}} h}^{n+1}\|^2+\|\Phi_{\mathbf{u} h}^{n+1}\|^2\right).
	\end{aligned}
\end{equation}
Combining the above inequalities $Tb_1,\cdots,Tb_{12}$ and summing up from $n=0$ to $m$, we have
\begin{equation}
	\begin{aligned}
		&\|\Theta_{\phi h}^{m+1}\|^2+\tau\sum_{n=0}^{m}\left(\frac{1}{2}\|\Theta_{\mu h}^{n+1}\|^2+\frac{1}{2}\|e_{\mu}^{n+1}\|^2\right)\\
		\leq&~\|\Theta_{\phi h}^0\|^2+\tau\sum_{n=0}^{m}\|\Phi_{\mu h}^{n+1}\|^2
				+C\tau\sum_{n=0}^{m}\left(\|\Phi_{\phi h}^{n+1}\|^2+\|R_1^{n+1}\|^2
				+\tau^2\|\mathbf{u}_t\|_{L^\infty(0,T;L^2(\Omega))}^2\|\nabla\phi^{n+1}\|^2\right)\\
			&+C\tau\sum_{n=0}^{m}\left(\|\nabla\phi^{n+1}\|^2+\|\nabla\mathbf{u}_h^n\|^2
			+\|\Theta_{\phi h}^{n}\|^2+\|\Phi_{\phi h}^{n}\|^2\right)
			\left(\|\Theta_{\mathbf{u} h}^{n+1}\|^2+\|\Phi_{\mathbf{u} h}^{n+1}\|^2\right.\\
			&\left.
			+\|\Theta_{\phi h}^{n+1}\|^2+\|\Phi_{\phi h}^{n+1}\|^2\right)+C h^{2(r+1)},\\
	\end{aligned}
\end{equation}
and 
\begin{equation}
	\begin{aligned}
		&\|\Theta_{{\mathbf{u}} h}^{m+1}\|^2+\tau^2\|\Theta_{p h}^{m+1}\|
		+\tau\sum_{n=0}^{m}\|\nabla\Theta_{\tilde{\mathbf{u}} h}^{n+1}\|^2\\
		\leq&~\|\Theta_{{\mathbf{u}} h}^0\|^2+\tau^2\|\nabla\Theta_{p h}^0\|^2
		+C\tau\sum_{n=0}^{m}\left(\|\delta_{\tau}\Phi_{\mathbf{u} h}^{n+1}\|^2+\|\Theta_{\mu h}^{n+1}\|^2+\|R_2^{n+1}\|^2+\|R_3^{n+1}\|^2\right)\\
		&+C\tau^3\sum_{n=0}^{m}\left(\|\mu_h^{n+1}\|_{H^1}^2+\|\phi^{n+1}\|_{H^1}^2\right)\left(\|\Phi_{\phi h}^{n+1}\|^2+\|\Theta_{\phi h}^{n+1}\|^2\right)\\
		&+C\tau^3\sum_{n=0}^{m}\left(\|\delta_{\tau}\Phi_{p h}^{n+1}\|^2
		+\|\mathbf{u}_t\|_{L^\infty(0,T;\mathbf{H}^1(\Omega))}^2\|\nabla\phi^{n+1}\|^2\right)\\
		&+C\tau\sum_{n=0}^{m}\left(\left(\|\nabla\mathbf{u}^{n+1}\|^2+\|\nabla\mathbf{u}_h^n\|^2\right)\left(\|\Theta_{{\mathbf{u}} h}^{n+1}\|^2+\|\Phi_{\mathbf{u} h}^{n+1}\|^2\right)+\|\nabla\mathbf{u}_h^n\|^2\left(\|\Theta_{\tilde{\mathbf{u}} h}^{n+1}\|^2+\|\Phi_{\mathbf{u} h}^{n+1}\|^2\right)\right).
	\end{aligned}
\end{equation}
According to the  triangle inequality, the inequality \eqref{eq_psi_Rhpsi_Hneg1_norm_inequation}, the Lemma \ref{lemma_stokes_projection001}, \ref{Lemma_0301} and \ref{lemma_boundness_phi_mu},
this completes the proof of Lemma \ref{lemma_L2_e_phi_u_mu}.
\end{proof}
\begin{Remark}\label{remark_uinfty_boundedness_explanation}
	Finally, we aim to confirm that, as a consequence of $\tau\leq \frac{\alpha}{C\|\mathbf{u}_h^{n}\|_{L^{\infty}}^2}$, the $a~ priori$ assumption \eqref{eq_Linfty_u_boundedness} remains valid for all sufficiently small values of $h$. To achieve this, we will show that, for sufficiently small $h$, the following requirement is satisfied. We proceed to prove \eqref{eq_Linfty_u_boundedness} using mathematical induction. Specifically, for $h \leq h_2$, \eqref{eq_Linfty_u_boundedness} holds when $n=0$, given that $\mathbf{u}_h^0=\boldsymbol{P}_h\mathbf{u}^0$. Assuming that \eqref{eq_Linfty_u_boundedness} holds for $n \leq m$, we will demonstrate that this assumption also holds for $n = m + 1$. Indeed, if \eqref{eq_Linfty_u_boundedness} holds for $n \leq m$, then \eqref{eq_assumption_conditions} holds for $n = m$. By using the inverse inequality, for $h\leq h_2$, we have
	\begin{equation}
		\begin{aligned}	
			\left\|\mathbf{u}_h^{m+1}\right\|_{L^{\infty}} \leq &~ \left\|\mathbf{u}_h^{m+1}-\boldsymbol{P}_h \mathbf{u}^{m+1}\right\|_{L^{\infty}}+\left\|\boldsymbol{P}_h \mathbf{u}^{m+1}\right\|_{L^{\infty}} \\
			\leq &~ C_1 h^{-d / 2}\left\|\mathbf{u}_h^{m+1}-\boldsymbol{P}_h \mathbf{u}^{m+1}\right\|+\left\|\boldsymbol{P}_h \mathbf{u}^{m+1}\right\|_{L^{\infty}} \\
			\leq & ~C_1 h^{-d / 2}\left\|\mathbf{u}^{m+1}-\mathbf{u}_h^{m+1}\right\|+C h^{-d / 2}\left\|\mathbf{u}^{m+1}-\boldsymbol{P}_h\mathbf{u}^{m+1}\right\| \\
			& \quad+C\|\mathbf{u}\|_{L^{\infty}\left(0, T; \mathbf{L}^{\infty}(\Omega)\right)} \\
			\leq & ~C\|\mathbf{u}\|_{L^{\infty}\left(0, T; \mathbf{L}^{\infty}(\Omega)\right)} .
		\end{aligned}
	\end{equation}
	Therefore, \eqref{eq_Linfty_u_boundedness} also holds at $n=m+1$. The boundedness explanation for $\|\mathbf{u}_h^n\|_{L^\infty}$ is complete.
	
\end{Remark}
\begin{Lemma}\label{lemma_eu_Hne1_norm}
	Under the assumption of regularity \eqref{eq_varibles_satisfied_regularities}, for all $m\geq 0$, there exists a positive constant that does not depend on $\tau$ and $h$ such that
	\begin{equation}
		\sum_{n=0}^{m}\|e_{\mathbf{u}}^{n+1}-e_{\mathbf{u}}^n\|_{H^{-1}}^2\leq C\left(\tau^2+h^{2r}\right).
	\end{equation}
\end{Lemma}
\begin{proof}
	Adding the \eqref{eq_error_equations_ns}-\eqref{eq_error_equations_u_p_cor}, we obtain
	\begin{equation}
		\label{eq_adding_ns_incompressible}
		\begin{aligned}
			\left(\delta_{\tau}e_{\mathbf{u}}^{n+1},\mathbf{v}_h\right)+\left(\nabla e_{\tilde{\mathbf{u}}}^{n+1},\nabla\mathbf{v}_h\right)
			-\left(e_p^{n+1},\nabla\cdot\mathbf{v}_h\right)-\left(R_2^{n+1},\mathbf{v}_h\right)-\left(R_3^{n+1},\mathbf{v}_h\right)&\\
			+\left(\mathbf{u}^{n+1}\cdot\nabla\mathbf{u}^{n+1}-\mathbf{u}_h^{n}\cdot\nabla\tilde{\mathbf{u}}_h^{n+1},\mathbf{v}_h\right)
			-\left(\mu^{n+1}\nabla\phi^{n+1}-\mu_h^{n+1}\phi_h^{n+1},\mathbf{v}_h\right)&=0.
		\end{aligned}
	\end{equation}
Testing $\mathbf{v}_h=\tau A_h^{-1}(e_{\mathbf{u}}^{n+1}-e_{\mathbf{u}}^n)$ and using the 
\eqref{eq_discrete_stokes_operators001}-\eqref{eq_discrete_stokes_operators002}, \eqref{eq_error_equations_incompressible}, \eqref{eq_recast_error_equation_projection_step},
we have
\begin{equation}
	\label{eq_eu_H_ne1_norms}
	\begin{aligned}
		\|e_{\mathbf{u}}^{n+1}-e_{\mathbf{u}}^n\|_{H^{-1}}^2=&~\tau\left(\Delta_h e_{\tilde{\mathbf{u}}}^{n+1},A_h^{-1}\left(e_{\mathbf{u}}^{n+1}-e_{\mathbf{u}}^n\right)\right)
		+\tau\left(\mu^{n+1}\nabla\phi^{n+1}-\mu_h^{n+1}\phi_h^{n+1},A_h^{-1}(e_{\mathbf{u}}^{n+1}-e_{\mathbf{u}}^n)\right)\\
		&-	\tau\left(\mathbf{u}^{n+1}\cdot\nabla\mathbf{u}^{n+1}-\mathbf{u}_h^{n}\cdot\nabla\tilde{\mathbf{u}}_h^{n+1},A_h^{-1}(e_{\mathbf{u}}^{n+1}-e_{\mathbf{u}}^n)\right)\\
		&+\tau\left(R_2^{n+1},A_h^{-1}(e_{\mathbf{u}}^{n+1}-e_{\mathbf{u}}^n)\right)+\tau\left(R_3^{n+1},A_h^{-1}(e_{\mathbf{u}}^{n+1}-e_{\mathbf{u}}^n)\right).
	\end{aligned}
\end{equation}
According to \eqref{eq_L2_projection_operators} and \eqref{eq_fully_discrete_CHNS_scheme_tilde_u_u1_p_p1}, 
we can obtain $\mathbf{u}_h^{n+1}=\boldsymbol{I}_h\tilde{\mathbf{u}}_h^{n+1}$. 
Using the \eqref{eq_discrete_Laplacian_operator} and the above results, we have 
\begin{equation}
	\label{eq_projections_reslut}
	\begin{aligned}
		\tau\left(\Delta_h e_{\tilde{\mathbf{u}}}^{n+1},A_h^{-1}\left(e_{\mathbf{u}}^{n+1}-e_{\mathbf{u}}^n\right)\right)
		=&~-\left(e_{\tilde{\mathbf{u}}}^{n+1},-\tau\Delta_hA_h^{-1}\left(e_{\mathbf{u}}^{n+1}-e_{\mathbf{u}}^n\right)\right)\\
		=&~-\tau\left(e_{{\mathbf{u}}}^{n+1},e_{\mathbf{u}}^{n+1}-e_{\mathbf{u}}^n\right)\\
		=&~-\frac{\tau}{2}\left(\|e_{\mathbf{u}}^{n+1}\|^2-\|e_{\mathbf{u}}^{n}\|^2+\|e_{\mathbf{u}}^{n+1}-e_{\mathbf{u}}^{n}\|^2\right).
	\end{aligned}
\end{equation}
By substituting \eqref{eq_projections_reslut} into \eqref{eq_eu_H_ne1_norms}, we obtain
\begin{equation}
	\label{eq_substituting}
	\begin{aligned}
		&\|e_{\mathbf{u}}^{n+1}-e_{\mathbf{u}}^n\|_{H^{-1}}^2+\frac{\tau}{2}\left(\|e_{\mathbf{u}}^{n+1}\|^2-\|e_{\mathbf{u}}^{n}\|^2+\|e_{\mathbf{u}}^{n+1}-e_{\mathbf{u}}^{n}\|^2\right)\\
		=&~\tau\left(\mu^{n+1}\nabla\phi^{n+1}-\mu_h^{n+1}\phi_h^{n+1},A_h^{-1}(e_{\mathbf{u}}^{n+1}-e_{\mathbf{u}}^n)\right)\\
		&-	\tau\left(\mathbf{u}^{n+1}\cdot\nabla\mathbf{u}^{n+1}-\mathbf{u}_h^{n}\cdot\nabla\tilde{\mathbf{u}}_h^{n+1},A_h^{-1}(e_{\mathbf{u}}^{n+1}-e_{\mathbf{u}}^n)\right)\\
		&+\tau\left(R_2^{n+1},A_h^{-1}(e_{\mathbf{u}}^{n+1}-e_{\mathbf{u}}^n)\right)+\tau\left(R_3^{n+1},A_h^{-1}(e_{\mathbf{u}}^{n+1}-e_{\mathbf{u}}^n)\right).
	\end{aligned}
\end{equation}
We then estimate the right-hand side terms of \eqref{eq_substituting} as follows
\begin{equation}
	\label{eq_substituting_error_estimate001}
	\begin{aligned}
		&~\bigg|\tau\left(\mu^{n+1}\nabla\phi^{n+1}-\mu_h^{n+1}\nabla\phi_h^{n+1},A_h^{-1}(e_{\mathbf{u}}^{n+1}-e_{\mathbf{u}}^n)\right)\bigg|\\
		=&~\bigg|\tau\left(e_{\mu}^{n+1}\nabla\phi^{n+1}+\mu_h^{n+1}\nabla e_{\phi}^{n+1},A_h^{-1}(e_{\mathbf{u}}^{n+1}-e_{\mathbf{u}}^n)\right)\bigg|\\
		\leq&~\frac{1}{6}\|e_{\mathbf{u}}^{n+1}-e_{\mathbf{u}}^n\|_{H^{-1}}^2
			+C\tau^2\left(\|\phi^{n+1}\|_{H^1}^2\|e_{\mu}^{n+1}\|_{H^1}^2+\|\mu_h^{n+1}\|_{H^1}^2\|e_{\phi}^{n+1}\|_{H^1}^2\right),
	\end{aligned}
\end{equation}
and 
\begin{equation}
	\label{eq_substituting_error_estimate002}
	\begin{aligned}
		&\bigg|-	\tau\left(\mathbf{u}^{n+1}\cdot\nabla\mathbf{u}^{n+1}-\mathbf{u}_h^{n}\cdot\nabla\tilde{\mathbf{u}}_h^{n+1},A_h^{-1}(e_{\mathbf{u}}^{n+1}-e_{\mathbf{u}}^n)\right)\bigg|\\
		=&~\bigg| \tau\left(e_{\mathbf{u}}^{n}\cdot\nabla \mathbf{u}^{n+1}+\mathbf{u}_h^n\cdot\nabla e_{\mathbf{u}}^{n+1}+\mathbf{u}_h^n\cdot\nabla e_{\tilde{\mathbf{u}}}^{n+1}+\left(\mathbf{u}^{n+1}-\mathbf{u}^n\right)\cdot\nabla\mathbf{u}^{n+1},A_h^{-1}(e_{\mathbf{u}}^{n+1}-e_{\mathbf{u}}^n)\right)\bigg|\\
		\leq&~\frac{1}{6}\|e_{\mathbf{u}}^{n+1}-e_{\mathbf{u}}^n\|_{H^{-1}}^2+C\tau\left(
		\|e_{\mathbf{u}}^n\|^2\|\nabla\mathbf{u}^{n+1}\|^2+\|\nabla\mathbf{u}_h^{n}\|^2\|e_{\mathbf{u}}^{n+1}\|^2
		\right)+C\tau^4\|\mathbf{u}_t\|_{L^\infty(0,T;H^1(\Omega))}^2\|\nabla\mathbf{u}^{n+1}\|^2,
	\end{aligned}
\end{equation}
and
\begin{equation}
	\label{eq_substituting_error_estimate003}
	\begin{aligned}
		&\bigg|\tau\left(R_2^{n+1},A_h^{-1}(e_{\mathbf{u}}^{n+1}-e_{\mathbf{u}}^n)\right)+\tau\left(R_3^{n+1},A_h^{-1}(e_{\mathbf{u}}^{n+1}-e_{\mathbf{u}}^n)\right)\bigg|\\
		\leq&~\frac{1}{6}\|e_{\mathbf{u}}^{n+1}-e_{\mathbf{u}}^n\|_{H^{-1}}^2+C\tau^2\left(\|R_2^{n+1}\|+\|R_3^{n+1}\|^2\right).
	\end{aligned}
\end{equation}
Combining \eqref{eq_substituting_error_estimate001}-\eqref{eq_substituting_error_estimate003} with \eqref{eq_substituting}
and summing over $n,~(0\leq n\leq m)$, we have
\begin{equation}
	\begin{aligned}
		&~\frac{1}{2}\sum_{n=0}^{m}\|e_{\mathbf{u}}^{n+1}-e_{\mathbf{u}}^n\|_{H^{-1}}^2+\frac{\tau}{2}\|e_{\mathbf{u}}^{m+1}\|^2\\
		\leq&~\frac{\tau}{2}\|e_{\mathbf{u}}^0\|^2+C\tau^2\sum_{n=0}^{m}\left( \|\phi^{n+1}\|_{H^1}^2\|e_{\mu}^{n+1}\|_{H^1}^2+\|\mu_h^{n+1}\|_{H^1}^2\|e_{\phi}^{n+1}\|_{H^1}^2+\|R_2^{n+1}\|+\|R_3^{n+1}\|^2\right)\\
		&+C\tau\sum_{n=0}^{m}\left(
		\|e_{\mathbf{u}}^n\|^2\|\nabla\mathbf{u}^{n+1}\|^2+\|\nabla\mathbf{u}_h^{n}\|^2\|e_{\mathbf{u}}^{n+1}\|^2
		\right)+C\tau^4\sum_{n=0}^{m}\|\mathbf{u}_t\|_{L^\infty(0,T;\mathbf{H}^1(\Omega))}^2\|\nabla\mathbf{u}^{n+1}\|^2.
	\end{aligned}
\end{equation} 
The proof of this lemma is completed. 
\end{proof}
Next, we prove the optimal $L^2$ error estimate for the pressure. 
Using the inf-sup condition \eqref{eq_inf_sup_condition} and 
the equation \eqref{eq_adding_ns_incompressible}, we obtain
\begin{equation}
	\label{eq_error_ep}
	\begin{aligned}
		C\|e_p^{n+1}\|\leq&~ \frac{\left(e_p^{n+1},\nabla\cdot\mathbf{v}_h\right)}{\|\nabla\mathbf{v}_h\|}\\
		\leq&~[\left(\delta_{\tau}e_{\mathbf{u}}^{n+1},\mathbf{v}_h\right)+\left(\nabla e_{\tilde{\mathbf{u}}}^{n+1},\nabla\mathbf{v}_h\right)
		-\left(R_2^{n+1},\mathbf{v}_h\right)-\left(R_3^{n+1},\mathbf{v}_h\right)\\
		&+\left(\mathbf{u}^{n+1}\cdot\nabla\mathbf{u}^{n+1}-\mathbf{u}_h^{n}\cdot\nabla\tilde{\mathbf{u}}_h^{n+1},\mathbf{v}_h\right)
		-\left(\mu^{n+1}\nabla\phi^{n+1}-\mu_h^{n+1}\phi_h^{n+1},\mathbf{v}_h\right)]/\|\nabla\mathbf{v}_h\|\\
		\leq&~\frac{1}{\tau}\|e_{\mathbf{u}}^{n+1}-e_{\mathbf{u}}^{n}\|_{H^{-1}}+\|\nabla e_{\tilde{\mathbf{u}}}^{n+1}\|+
			\|R_2^{n+1}\|+\|R_3^{n+1}\|\\
			&+\|\nabla\mathbf{u}^{n+1}\|\left(\|e_{\mathbf{u}}^{n}\|
			+\|\mathbf{u}_t\|_{L^\infty(0,T;\mathbf{H}^1(\Omega))}\right)
			+\|e_{\mu}^{n+1}\|_{H^1}\|\nabla\phi^{n+1}\|+\|\mu_h^{n+1}\|_{H^1}\|\nabla e_{\phi}^{n+1}\|.
	\end{aligned}
\end{equation}
Squaring both ends of the \eqref{eq_error_ep} and multiplying the result by $\tau$, summing over $n~(0\leq n\leq m)$, and
using the Lemma \ref{Lemma_0301}, \ref{lemma_boundness_phi_mu}, \ref{lemma_truncation_errors}, \ref{lemma_eu_Hne1_norm}
and the Theorem \ref{theorem_e_phi_L2}, we can obtain the result of this lemma as follows
\begin{equation}
	\tau\sum_{n=0}^{m}\|e_p^{n+1}\|^2\leq C\left(\tau^2+h^{2r}\right).
\end{equation}
Up to this point, by utilizing the above lemmas and theorems, we have successfully arrived at the  result of Theorem \ref{theorem_error_estimates_u_phi_mu_r}.

\section{Numerical experiments}\label{section_numerical_experiments}
	In this section, we conduct three numerical experiments for the following purposes: 1) to validate the convergence rates as stated in Theorem \ref{theorem_error_estimates_u_phi_mu_r}; 2) to simulate the coarsening dynamics using a random initial phase function; and 3) to demonstrate numerically the unconditional stability of our scheme, which is only conditionally stable in theory. In all tests, we employ the finite element spaces $S^2_h \times S^2_h \times \mathbf{X}^2_h \times S^1_h$ for the phase field function, chemical potential, velocity, and pressure, respectively; in other words, we use the $P_2$ function space for both $\phi^{n+1}_h$ and $\mu^{n+1}_h$, and the $\mathbf{P}_2\times P_1$ space for $(\mathbf{u}^{n+1}_h, p^{n+1}_h)$.
	
	\subsection{Convergence tests}\label{subsection_convergence_tests}
	The first test is performed on $\Omega = [0,1]\times[0,1]$ and time $T = 0.01$, and the physical parameters are selected as $M = 0.1$, $\lambda = 0.04$, $\epsilon = 0.04$ and $\nu = 0.1$. The time-step $\tau$ is set up as $\tau = 0.1h^3$ to balance the convergence rates between time and space. The manufactured solutions are chosen from \cite{2022_ChenYaoyao_CHNS_2022_AMC} and takes the form
	\begin{equation}
		\left\{
		\begin{aligned}
			\phi(t,x,y) &= 2 + \sin(t)\cos(\pi x)\cos(\pi y),  \\
			\mathbf{u}(t,x,y) &= \left[\pi\sin(\pi x)^2\sin(2\pi y)\sin(t), -\pi\sin(\pi y)^2\sin(2\pi x)\sin(t)\right]^{\text{T}},  \\
			p(t,x,y) &= \cos(\pi x)\sin(\pi y)\sin(t),
		\end{aligned}\right.
	\end{equation}
	and the exact chemical potential $\mu(t,x,y)$ is obtained by its definition, i.e., 
	\begin{equation}
		\begin{aligned}
			\mu(t,x,y) :=& -\lambda \Delta \phi(t,x,y) + \lambda F^{\prime}\big(\phi(t,x,y)\big) \\
			=&  -\lambda \Delta \phi(t,x,y) + \frac{\lambda}{\epsilon^2} \left(\phi(t,x,y)^3 - \phi(t,x,y)\right) \\
			=&~ 2\lambda \pi^2 \cos(\pi x)\cos(\pi y) \sin(t)  \\
			&-\frac{\lambda}{\epsilon^2}\left[\cos(\pi x)\cos(\pi y) \sin(t) - \big(\cos(\pi x)\cos(\pi y) \sin(t) + 2\big)^3 + 2 \right].
		\end{aligned}
	\end{equation}
	
	For simplicity of notations, we adopt the following notations:
	\begin{equation}
		\begin{aligned}
			\|\cdot\|_{\ell^{\infty}(L^2)} := \max\limits_{0\leq n \leq N-1}\|\cdot\|_{L^2}, \quad 
			|\cdot|_{\ell^{\infty}} := \max\limits_{0\leq n \leq N-1}|\cdot|, \quad \|\cdot\|_{\ell^{2}(L^2)} := \sqrt{\tau \sum_{n=0}^{N-1}\|\cdot\|_{L^2}}.
		\end{aligned}
	\end{equation}
	Therefore, according to Theorem \ref{theorem_error_estimates_u_phi_mu_r}, the following theoretically optimal orders should be observed:
	\begin{equation}
		\begin{aligned}
			\left\|\phi - \phi_h\right\|_{\ell^{\infty}(L^2)} & \approx  \mathcal{O}(h^3), \qquad \left\|\mu - \mu_h\right\|_{\ell^2(L^2)}  \approx  \mathcal{O}(h^3), \\
			\left\|\boldsymbol{u} - \boldsymbol{u}_h\right\|_{\ell^{\infty}(L^2)} & \approx   \mathcal{O}(h^3), \qquad \left\|p - p_h\right\|_{\ell^2(L^2)}  \approx  \mathcal{O}(h^2), \\
			\left\|\nabla(\boldsymbol{u} - \boldsymbol{u}_h)\right\|_{\ell^{2}(L^2)} & \approx  \mathcal{O}(h^2).
		\end{aligned}
	\end{equation}
	\autoref{L2-error-convRates-FirstOrder} and \autoref{H1-error-convRates-FirstOrder} present the errors and convergence rates, where $L^2$ errors and convergence orders, and $H^1$ errors and convergence orders for velocity  of CHNS  are presented by spatial ones after setting $\tau = \mathcal{O}(h^3)$. 
	From those tables, we can conclude that the numerical results are consistent with the above \emph{a prior} rates, and hence the scheme proposed here can indeed obtain the optimal convergence rates.  
	\begin{table}[h]
		\centering
		\fontsize{10}{10}
		\begin{threeparttable}
			\caption{$L^2$ errors and convergence orders.}\label{L2-error-convRates-FirstOrder}
			\begin{tabular}{c|c|c|c|c|c|c}
				\toprule
				\multirow{2.5}{*}{$h$}  & \multicolumn{2}{c}{$\left\|\phi - \phi_h\right\|_{\ell^{\infty}(L^2)}$} & \multicolumn{2}{c}{$\left\|\mu - \mu_h\right\|_{\ell^2(L^2)}$}  & \multicolumn{2}{c}{$\left\|p - p_h\right\|_{\ell^2(L^2)}$} \cr
				\cmidrule(lr){2-3} \cmidrule(lr){4-5} \cmidrule(lr){6-7}  
				& error & rate & error & rate & error & rate \cr
				\midrule
				1/4   	& 8.5456e$-$05 &  -       		 & 6.5126e$-$06  & -       		     & 8.4150e$-$03 & -      \cr
				1/8   	& 1.0599e$-$05 &   3.0112  	 & 8.3148e$-$07  & 2.9695  & 1.7367e$-$03 & 3.2502 \cr
				1/16  	 & 1.2872e$-$06 &  3.0417     & 1.0217e$-$07  & 3.0247  & 1.7543e$-$04 & 3.1413 \cr
				1/32  	& 1.5815e$-$07 &  3.0249  & 1.2640e$-$08  & 3.0149  & 2.1352e$-$05 & 3.0409 \cr
				1/64  	& 1.9585e$-$08 &  3.0135  & 1.5707e$-$08  & 3.0085   & 2.6588e$-$06& 3.0074 \cr
				\bottomrule
			\end{tabular}
		\end{threeparttable}
	\end{table}
	\begin{table}[h]
		\centering
		\fontsize{10}{10}
		\begin{threeparttable}
			\caption{$L^2$- and $H^1$-norm errors and convergence orders for velocity and time convergence orders for CHNS (presented by spatial orders).}
			\label{H1-error-convRates-FirstOrder}
			\begin{tabular}{c|c|c|c|c}
				\toprule
				\multirow{2.5}{*}{$h$} & \multicolumn{2}{c}{$\left\|\mathbf{u} - \mathbf{u}_h\right\|_{\ell^{\infty}(L^2)}$} & \multicolumn{2}{c}{$\left\|\nabla(\mathbf{u} - \mathbf{u}_h)\right\|_{\ell^{2}(L^2)}$} \cr
				\cmidrule(lr){2-3}  \cmidrule(lr){4-5} 
				& error & rate  & error & rate \cr
				\midrule
				1/4   \qquad&\quad 3.4875e$-$04  & - &\quad 1.6371e$-$02 $\quad$&\quad  -       			   \cr
				1/8   \qquad& \quad7.0458e$-$05  & 2.3079 &\quad 5.1257e$-$03 $\quad$&\quad  1.6751     \cr
				1/16  \qquad& \quad8.6735e$-$06  & 3.0217 &\quad 1.2321e$-$03 $\quad$&\quad  2.0565    \cr
				1/32  \qquad& \quad1.0212e$-$06  & 3.0861 &\quad 2.9154e$-$04 $\quad$&\quad  2.0794     \cr
				1/64  \qquad& \quad1.2489e$-$07  & 3.0315&\quad 7.1441e$-$05 $\quad$&\quad  2.0289   \cr
				\bottomrule
			\end{tabular}
		\end{threeparttable}
	\end{table}
	
	\subsection{Coarsening dynamics}\label{subsection_coarsening_dynamics}
	In this example, we employ the Cahn–Hilliard–Navier–Stokes equations \eqref{eq_chns_equations} to simulate the coarsening dynamics, with a random phase function. We set the domain as $\Omega = [0, 1] \times [0, 1]$, and choose the parameters as $M = 0.01$, $\lambda = 0.02$, $\epsilon = 0.01$, and $\nu = 1$, along with a random initial condition for the phase function within the range $[-0.1, 0.1]$. The spatial discretization is set at $h = 1/64$, and the time step is $\tau = 0.001$. We conduct this test up to $T = 10$ and record the phase function snapshots at $t = 0,0.01,0.05,0.1,0.5,1,3,5,10s$, which are displayed in \autoref{Figure-CoarDyna}.
	\begin{figure}[h!]
		\centering
		\includegraphics[width=1\linewidth]{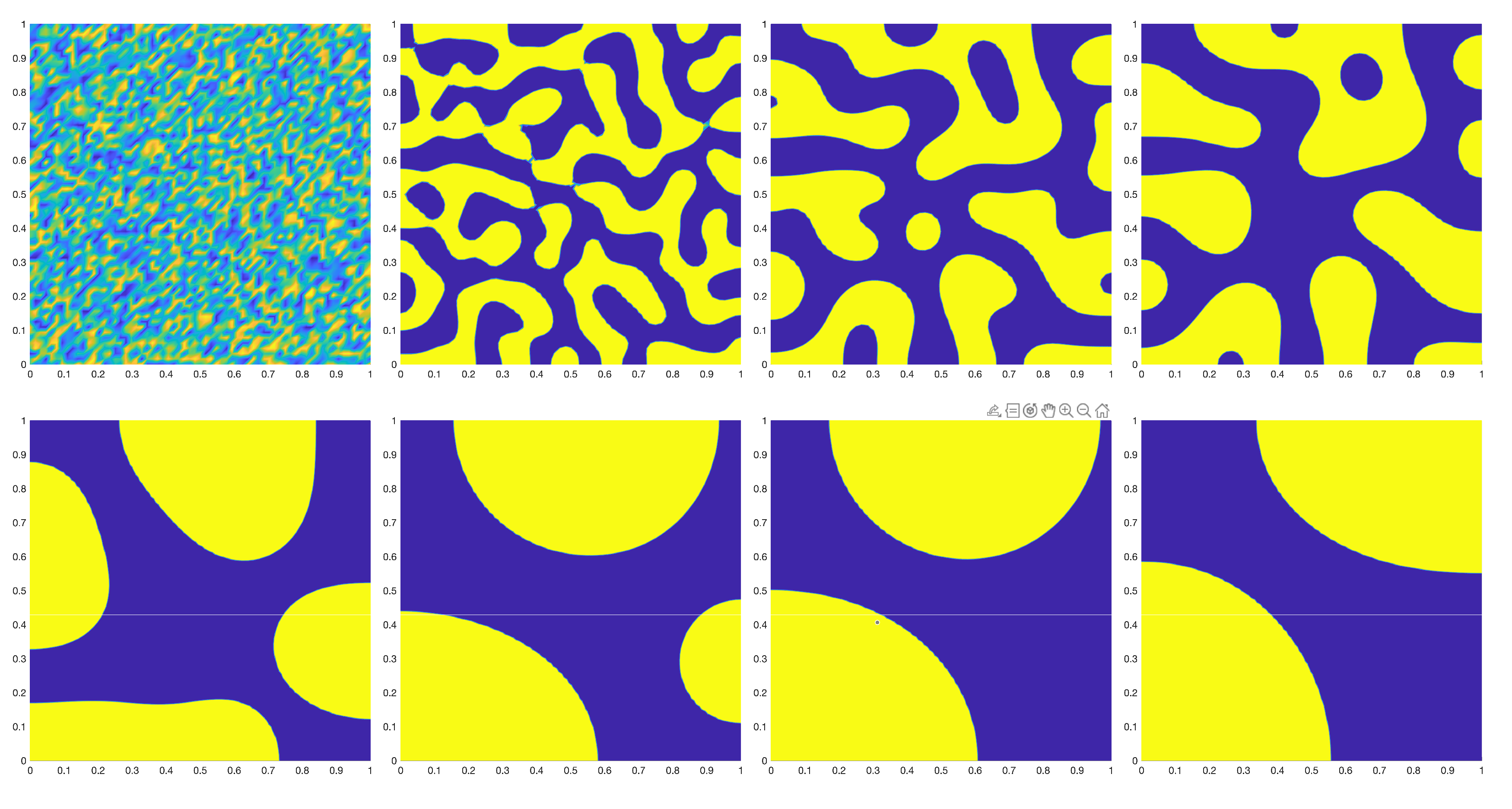}
		\caption{Snapshots of phase function at difference times from left to right row by row with $t=0,0.01,0.05,0.1,0.5,1,3,5,10s$, respectively.}\label{Figure-CoarDyna}
	\end{figure}
	\subsection{Shape Relaxation with Rotational Boundary Condition}
	In this section, we present an example of shape relaxation within the domain $\Omega = (0, 1) \times (0, 1)$, where the boundary condition is set to $\mathbf{u} = (y - 0.5, -x + 0.5)$ on $\partial\Omega$. We evaluate the performance of the scheme described by \eqref{eq_fully_discrete_CHNS_scheme_phi}-\eqref{eq_fully_discrete_CHNS_scheme_tilde_incompressible_condition} using critical phase field initial conditions. Specifically, we define $\phi^0 = 1$ within a polygonal subdomain featuring reentrant corners and $\phi^0 = -1$ in the rest of $\Omega$. The initial velocity is given by $\mathbf{u}^0 = (y - 0.5, -x + 0.5)$. This problem has been numerically studied in \cite{2008_Kay_David_and_Welford_Richard_Finite_element_approximation_of_a_Cahn_Hilliard_Navier_Stokes_system}. For our simulations, we adopt the following parameters: $M = 0.001$, $\lambda = 0.1$, $\epsilon = 0.01$, $\nu = 1$, and $T = 0.5$. We solve the problem using scheme \eqref{eq_fully_discrete_CHNS_scheme_phi}-\eqref{eq_fully_discrete_CHNS_scheme_tilde_incompressible_condition} with a $P_2$ element for $\phi_h^{n+1}$ and $\mu_h^{n+1}$, and a $P_2 \times P_1$ element for $\mathbf{u}_h^{n+1}$ and $p_h^{n+1}$. We run the simulation up to $T = 0.5$ and record the phase function snapshots at $t = 0, 0.01, 0.05, 0.08, 0.1, 0.2, 0.3, 0.5$, which are depicted in \autoref{Figure-RotationalBoundary}. Our scheme is unconditionally energy stable in this example, as shown in \autoref{Figure_CHNS_Decoupled_OptimalL2_Gao_3_RotationalBoundary_EnergyDesc}.
	\begin{figure}[h!]
		\centering
		\includegraphics[width=1\linewidth]{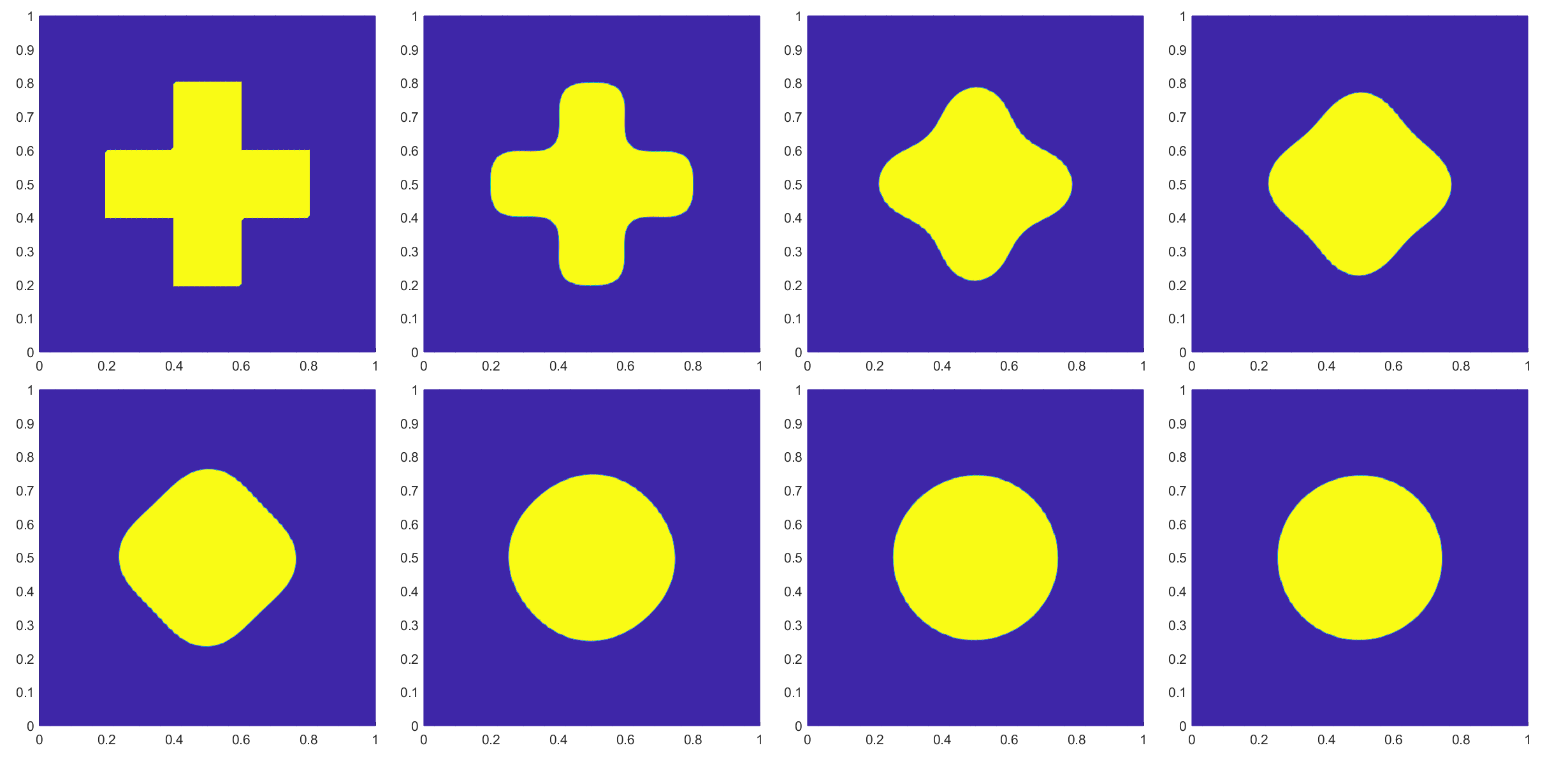}
		\caption{Snapshots of phase function at difference times from left to right row by row with $t=0, 0.01, 0.05, 0.08, 0.1, 0.2, 0.3, 0.5$, respectively.}
		\label{Figure-RotationalBoundary}
	\end{figure}
	\subsection{Comparison of time efficiency}
	In the subsection, we compare it with the following scheme  \cite{2023_CaiWentao_Optimal_L2_error_estimates_of_unconditionally_stable_finite_element_schemes_for_the_Cahn_Hilliard_Navier_Stokes_system} in the computation of the total time,
	\begin{equation}
		\begin{aligned}
			\left(\delta_{\tau}\phi_h^{n+1},w_h\right)+\left(\nabla \mu_h^{n+1},\nabla w_h\right)+b\left(\phi_h^n,\mathbf{u}_h^{n+1},w_h\right)=~&0,\\
			\left(\nabla\phi_h^{n+1},\nabla\varphi_h\right)-\left(\mu_h^{n+1},\varphi_h\right)+\left((\phi_h^{n+1})^3-\phi_h^n,\varphi_h\right)=~&0,\\
			\left(\delta_{\tau}\mathbf{u}_h^{n+1},\mathbf{v}_h\right)+\left(\nabla\mathbf{u}_h^{n+1},\nabla\mathbf{v}_h\right)+\boldsymbol{B}\left(\mathbf{u}_h^n,\mathbf{u}_h^{n+1},\mathbf{v}_h\right)-\left(p_h^{n+1},\nabla\cdot\mathbf{v}_h\right)-b\left(\phi_h^n,\mathbf{v}_h,\mu_h^{n+1}\right)=~&0,\\
			\left(\nabla\cdot\mathbf{u}_h^{n+1},q_h\right)=~&0.
		\end{aligned}
	\end{equation}
	By using the above scheme and our proposed scheme to the convergence tests of \autoref{subsection_convergence_tests}, we obtain the time results in \autoref{comparison_of_time}. The data in \autoref{comparison_of_time} shows that our algorithm takes 8.22 hour and the algorithm in \cite{2023_CaiWentao_Optimal_L2_error_estimates_of_unconditionally_stable_finite_element_schemes_for_the_Cahn_Hilliard_Navier_Stokes_system} takes 23.18 hour, which is about 3 times faster in time than the scheme of \cite{2023_CaiWentao_Optimal_L2_error_estimates_of_unconditionally_stable_finite_element_schemes_for_the_Cahn_Hilliard_Navier_Stokes_system}. Thus our scheme is more time efficient.
	\begin{table}
		\centering
		\caption{Comparison of time consumed in different schemes.}
		\label{comparison_of_time}
		\renewcommand\arraystretch{2}
		\begin{tabular}{c|c}
			\hline
			\qquad\qquad Scheme \qquad \qquad\qquad &   \qquad \qquad CPU Time (hour) \qquad\qquad\qquad \\
			\hline
			Cai et al. \cite{2023_CaiWentao_Optimal_L2_error_estimates_of_unconditionally_stable_finite_element_schemes_for_the_Cahn_Hilliard_Navier_Stokes_system}  &   23.18  \\
			\hline
			Our scheme  &   8.22  \\
			\hline
		\end{tabular}
	\end{table}
	\subsection{Lid driven cavity boundary condition}
	In the experiment, we take 
	$\Omega = [0,1 ]^2$, $\epsilon = 0.02$, $M = 0.025$, $\nu = 0.1$, $\lambda = 0.001$, boundary condition on the top is $g = [1,0]$, initial condition for velocity is $\mathbf{u}_0 = [1,0]$, $\tau = 0.001, h = 1/64$. We solve the problem using scheme \eqref{eq_fully_discrete_CHNS_scheme_phi}-\eqref{eq_fully_discrete_CHNS_scheme_tilde_incompressible_condition} with the $P_2\times P_2\times \mathbf{P}_2\times P_1$ for $\phi_h^{n+1}$, $\mu_h^{n+1}$, $\mathbf{u}_h^{n+1}$ and $p_h^{n+1}$, snapshots time $t = 0, 5, 10, 20, 30, 50$.
	The evolution of the concentration for this problem is shown through \autoref{Figure_CHNS_Decoupled_OptimalL2_CondStab_Gao3_LidDrivenBounCond} (see \cite{2002_Boyer_A_theoretical_and_numerical_model_for_the_study_of_incompressible_mixture_flows} for
	a similar problem).
	\begin{figure}[h!]
		\centering
		\includegraphics[width=0.7\linewidth]{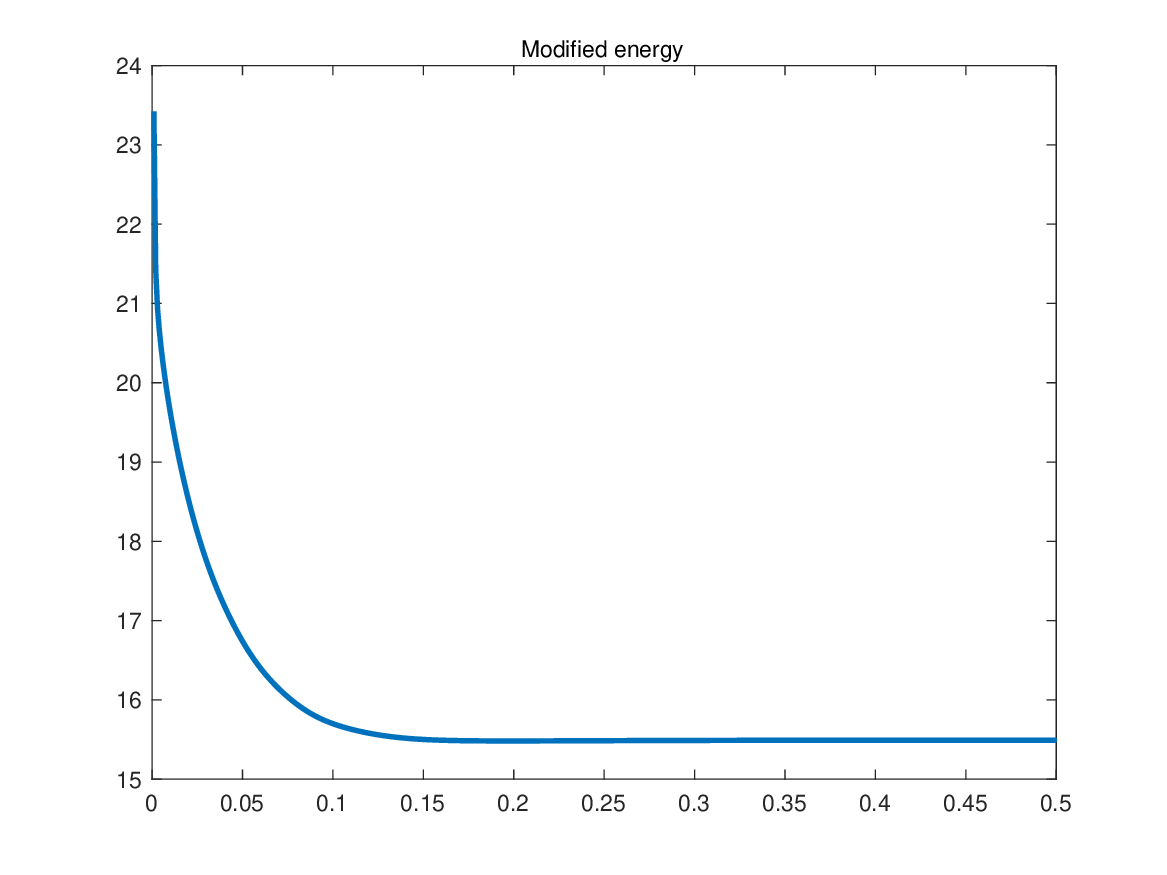}
		\caption{Energy decay of our scheme in this example.}
		\label{Figure_CHNS_Decoupled_OptimalL2_Gao_3_RotationalBoundary_EnergyDesc}
	\end{figure}
	\begin{figure}[h!]
		\centering
		\includegraphics[width=1\linewidth]{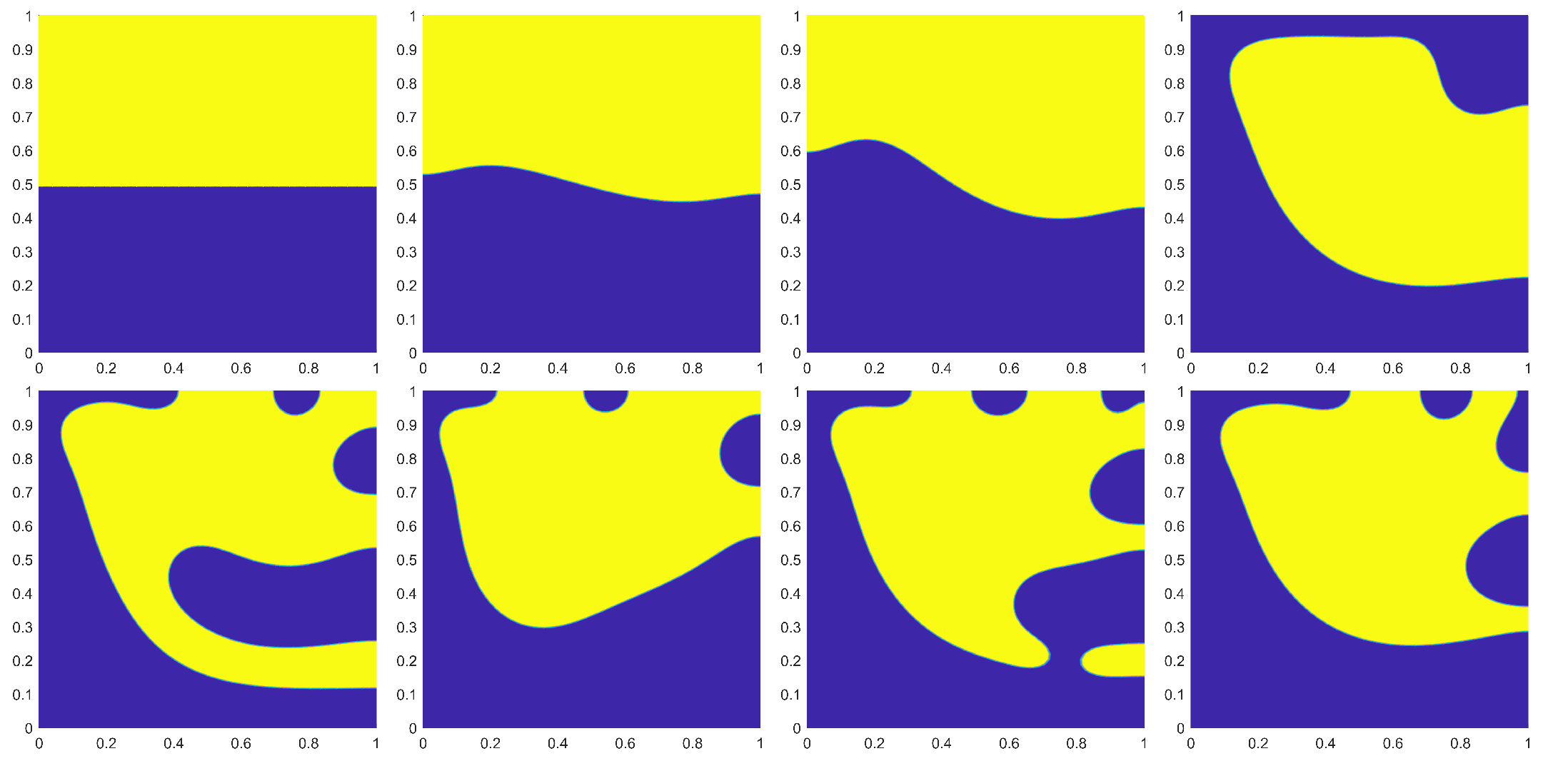}
		\caption{Evolution of the concentration with different time-step $\tau$.}\label{Figure_CHNS_Decoupled_OptimalL2_CondStab_Gao3_LidDrivenBounCond}
	\end{figure}
	
	\subsection{Verification of energy stability}
	In this final test, we aim to verify the energy stability of the new scheme. We use the same settings as those in the coarsening dynamics test (see subsection \ref{subsection_coarsening_dynamics}), but with varying time steps $\tau$. As shown in \autoref{Figure-CoarDyna-Energy-UncondStab}, our scheme demonstrates unconditional stability with respect to energy dissipation. 
	\begin{figure}[h!]
		\centering
		\includegraphics[width=0.7\linewidth]{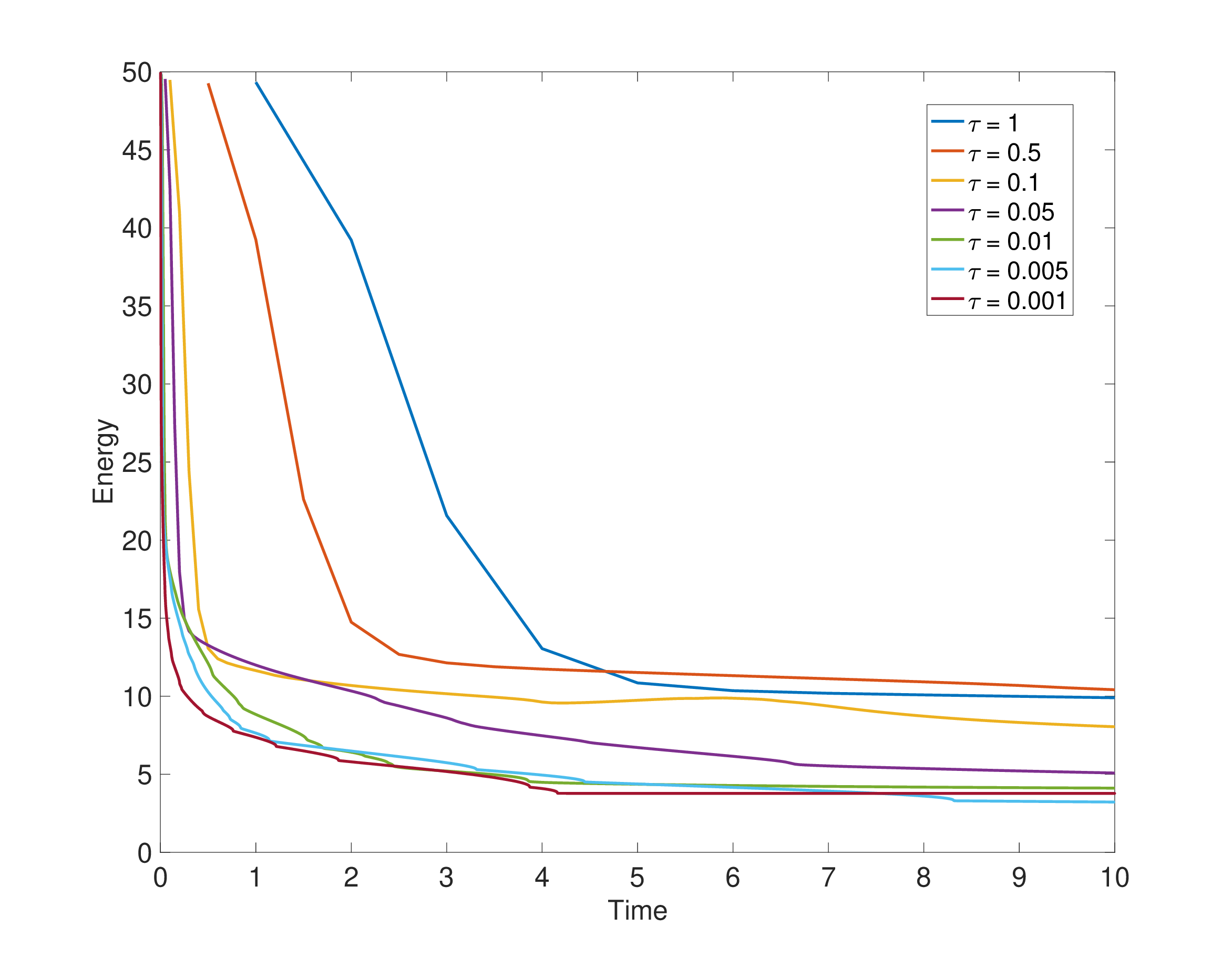}
		\caption{Evolution of the energy with different time-step $\tau$.}\label{Figure-CoarDyna-Energy-UncondStab}
	\end{figure}
	\section{Conclusions}\label{section_conclusion}
	In this work, we propose a decoupling method for the Cahn–Hilliard–Navier–Stokes (CHNS) model, a highly coupled nonlinear system. The method, based on a convex-splitting approach, is first-order, fully discrete, and energy stable, making it efficient and easy to implement. Moreover, we have conducted a rigorous error analysis for the fully discrete scheme and derived optimal error estimates for all relevant functions in the $L^2$ norm. The presented numerical experiments are designed to verify the theoretical results of our method. It has been experimentally proven that our scheme is unconditionally stable, a property that has not been theoretically proven and requires further research.
	\section*{Acknowledgments}
	This research is supported by the National Natural Science Foundation of China (No.11971337).
	\section*{Appendix}
	\subsection{Proof of Lemma \ref{lemma_solutions_boundedness}}
	\label{subsection_appendix_proof_lemma0302}
	\begin{proof}
		Firstly, we define a map $\mathcal{G}: S_h^r\times S_h^r\rightarrow S_h^r\times S_h^r$ by 
		$\mathcal{G}(\phi_h^{n+1},\mu_h^{n+1})=\left(\hat{\phi}_h^{n+1},\hat{\mu}_h^{n+1}\right)$,
		$\left(\hat{\phi}_h^{n+1},\hat{\mu}_h^{n+1}\right)\in S_h^r\times S_h^r$ satisfies
		\begin{subequations}
			\label{eq_CH_discrete_point}
			\begin{align}
				\left(\frac{\hat{\phi}_h^{n+1}-\phi_h^n}{\tau},w_h\right)-\left(\phi_h^{n+1}\mathbf{u}_h^{n},\nabla w_h\right)
				+\left(\nabla \hat{\mu}_h^{n+1},\nabla w_h\right)&=0,\\
				\left(\hat{\mu}_h^{n+1},\varphi_h\right)-\left(\nabla\hat{\phi}_h^{n+1},\nabla\varphi_h\right)-\left((\phi_h^{n+1})^3-\phi_h^n,\varphi_h\right)&=0,
			\end{align}
		\end{subequations}
		where $\left(\phi_h^{n+1},\mu_h^{n+1}\right)\in S_h^r\times S_h^r$ and $\left(w_h,\varphi_h\right)\in S_h^r\times S_h^r$. We will demonstrate that the mapping $\mathcal{G}$ fulfills the conditions of Lemma \ref{lemma_fixed_point}, and consequently, it possesses a fixed point that serves as a solution to the scheme \eqref{eq_fully_discrete_CHNS_scheme_phi}-\eqref{eq_fully_discrete_CHNS_scheme_mu}.
		
		Given $\mathbf{u}_h^n\in X_h^{r+1}$ and $\phi_h^n,\phi_h^{n+1}\in S_h^r$, the CH equation \eqref{eq_CH_discrete_point} can be rewritten as the following problem: find $\left(\hat{\phi}_h^{n+1},\hat{\mu}_h^{n+1}\right)\in S_h^r\times S_h^r$ satisfly
		\begin{equation}
			\begin{aligned}
				a(\hat{\mu}_h^{n+1},\varphi_h)+b(\hat{\phi}_h^{n+1},\varphi_h)=&~\langle f,\varphi_h\rangle,\\
				b(\hat{\mu}_h^{n+1},w_h)-c(\hat{\phi}_h^{n+1},w_h)=&~\langle g,w_h\rangle,
			\end{aligned}
		\end{equation}
		for any $\left(w_h,\varphi_h\right)\in S_h^r\times S_h^r$, where
		$$
		\begin{aligned}
			&a(\hat{\mu}_h^{n+1},\varphi_h)=\left(\hat{\mu}_h^{n+1},\varphi_h\right), ~
			b(\hat{\phi}_h^{n+1},\varphi_h)=-\left(\nabla\hat{\phi}_h^{n+1},\nabla \varphi_h\right),\\
			&c(\hat{\phi}_h^{n+1},w_h)=\frac{1}{\tau}\left(\hat{\phi}_h^{n+1},w_h\right),~
			\langle f,\varphi_h\rangle=\left((\phi_h^{n+1})^3-\phi_h^n,\varphi_h\right),\\
			&\langle g,w_h\rangle=-\frac{1}{\tau}\left(\phi_h^n,w_h\right)-\left(\phi_h^{n+1}\mathbf{u}_h^n,\nabla w_h\right).
		\end{aligned}
		$$
		We denote the $S_{0h}^r=\{\varphi_h\in S_h^r: b(\varphi_h,w_h)=0,\forall w_h\in S_h^r\}$. For any $\varphi_h\in S_{0h}^r$, it holds 
		$ \|\nabla \varphi_h\|=0$ and $a(\cdot,\cdot)$ is coercive on $S_{0h}^r$. Moreover, we can easily show that 
		$a(\cdot,\cdot)$ and $b(\cdot,\cdot)$ are continuous, and $c(\cdot,\cdot)$ is continuous, symmetric and positive semi-definite. Thus, according to the Section II.1.2 of \cite{1991_Brezzi_Fortin_Mixed_and_hybrid_finite_element_methods}, given  $\mathbf{u}_h^n\in \mathbf{X}_h^{r+1}$ and $\phi_h^n,\phi_h^{n+1}\in S_h^r$, the problem \eqref{eq_CH_discrete_point} is well-posed.
		Furthermore, because the spaces $S_h^r\times S_h^r$ are finite dimensional sapces, it holds that $\mathcal{G}$ is a compact map.
		
		Next, we give the boundedness of $\left(\hat{\phi}_h^{n+1},\hat{\mu}_h^{n+1}\right)$ in $S_h^r\times S_h^r$ as follows:
		\begin{equation}\label{eq_H1_phi_mu_boundedness}
			\|\hat{\phi}_h^{n+1}\|_{H^1}+\|\hat{\mu}_h^{n+1}\|_{H^1}\leq C,
		\end{equation}
		for some $C$ that is a positive constant independent of $\alpha$ and  $\left(\hat{\phi}_h^{n+1},\hat{\mu}_h^{n+1}\right)$.
		Letting $\mathcal{G}(\hat{\phi}_h^{n+1},\hat{\mu}_h^{n+1})=\frac{1}{\alpha}(\hat{\phi}_h^{n+1},\hat{\mu}_h^{n+1})$,
		we have
		\begin{subequations}
			\label{eq_ch_hat_alpha0011}
			\begin{flalign}
				\left(\frac{\hat{\phi}_h^{n+1}-\alpha\phi_h^n}{\tau},w_h\right)-\alpha\left(\hat{\phi}_h^{n+1}\mathbf{u}_h^n,\nabla w_h\right)+\left(\nabla \hat{\mu}_h^{n+1},\nabla w_h\right)&~=0,\\
				\left(\hat{\mu}_h^{n+1},\varphi_h\right)-\left(\nabla\hat{\phi}_h^{n+1},\nabla\varphi_h\right)-\alpha\left((\hat{\phi}_h^{n+1})^3-\phi_h^n,\varphi_h\right)&~=0,
			\end{flalign}
		\end{subequations}
		for any $\left(w_h,\varphi_h\right)\in S_h^r\times S_h^r$. Taking $\left(w_h,\varphi_h\right)=2\left(\tau \hat{\mu}_h^{n+1},-\hat{\phi}_h^{n+1}+\alpha\phi_h^n\right)$ into \eqref{eq_ch_hat_alpha0011}, and summing the obtained equalities, we get
		\begin{equation}
			\label{eq_hat_phi_nable_mu00001}
			\begin{aligned}
				&2\tau\|\nabla\hat{\mu}_h^{n+1}\|^2+\left(\|\nabla\hat{\phi}_h^{n+1}\|^2-\alpha^2\|\nabla\phi_h^n\|^2+\|\nabla(\hat{\phi}_h^{n+1}-\alpha\phi_h^n)\|^2\right)\\
				&\frac{\alpha}{2}\left(\|(\hat{\phi}_h^n)^2-1\|^2-\|(\alpha\phi_h^n)^2-1\|^2+\|(\hat{\phi}_h^{n+1})^2-(\alpha\phi_h^n)^2\|\right)\\
				&\alpha\left(\|\hat{\phi}_h^{n+1}(\hat{\phi}_h^{n+1}-\alpha\phi_h^n)\|^2+\|\hat{\phi}_h^{n+1}-\alpha\phi_h^n\|^2\right)\\
				\leq&~ 2\tau\left(\hat{\phi}_h^{n+1}\mathbf{u}_h^n,\nabla\hat{\mu}_h^{n+1}\right)+2\alpha(1-\alpha)\left(\phi_h^n,\hat{\phi}_h^{n+1}-\alpha\phi_h^n\right)\\
				\leq &~C\tau\|\hat{\phi}_h^{n+1}\|\|\mathbf{u}_h^n\|_{L^\infty}\|\nabla\hat{\mu}_h^{n+1}\|+\frac{\alpha}{4}\|\hat{\phi}_h^{n+1}-\alpha\phi_h^n\|^2+4\alpha(1-\alpha)^2\|\phi_h^n\|^2\\
				\leq &~C\tau\|\mathbf{u}_h^n\|_{L^\infty}^2\|\hat{\phi}_h^{n+1}-\alpha\phi_h^n+\alpha\phi_h^n\|^2+\tau\|\nabla\hat{\mu}_h^{n+1}\|^2\\
				&~+\frac{\alpha}{4}\|\hat{\phi}_h^{n+1}-\alpha\phi_h^n\|^2+4\alpha(1-\alpha)^2\|\phi_h^n\|^2\\
				\leq &~\frac{\alpha}{4}\|\hat{\phi}_h^{n+1}-\alpha\phi_h^n\|^2+\frac{\alpha^2}{2}\|\phi_h^n\|^2+\tau\|\nabla\hat{\mu}_h^{n+1}\|^2\\
				&~+\frac{\alpha}{4}\|\hat{\phi}_h^{n+1}-\alpha\phi_h^n\|^2+4\alpha(1-\alpha)^2\|\phi_h^n\|^2.
			\end{aligned}
		\end{equation}
		where we used the following assumption
		\begin{equation}
			\label{eq_assumption_conditions}
			\tau\leq \frac{\alpha}{C\|\mathbf{u}_h^{n}\|_{L^{\infty}}^2}.
		\end{equation}	
		\begin{Remark}
			According to the assumption $\tau\leq \frac{\alpha}{C\|\mathbf{u}_h^n\|_{L^{\infty}}^2}$, the above inequality \eqref{eq_hat_phi_nable_mu00001} holds. Therefore, we can derive the boundedness of the nth-time-layer velocity as follows:
			\begin{equation}
				\label{eq_Linfty_u_boundedness}
				\|\mathbf{u}_h^n\|_{L^{\infty}}\leq C\|\mathbf{u}\|_{L^{\infty}(0,T;\mathbf{L}^{\infty}(\Omega))}\leq C,
			\end{equation}
			where $C$ is a positive constant independent of $\alpha$, $\tau$, $h$, and $n$.
			 In this paper, we will analyze the reasonableness of this condition \eqref{eq_assumption_conditions}.
		\end{Remark}
		By from \eqref{eq_hat_phi_nable_mu00001} and $\alpha\in [0,1]$, we have
		\begin{equation}
			\begin{aligned}
				&~\tau\|\nabla\hat{\mu}_h^{n+1}\|^2+\|\nabla\hat{\phi}_h^{n+1}\|^2+\|\nabla(\hat{\phi}_h^{n+1}-\alpha\phi_h^n)\|^2+\frac{\alpha}{2}\|(\hat{\phi}_h^{n+1})^2-1\|^2\\
				+&~\alpha\left(\frac{1}{2}\|(\hat{\phi}_h^{n+1})^2-(\alpha\phi_h^n)^2\|^2+\|\hat{\phi}_h^{n+1}(\hat{\phi}_h^{n+1}-\alpha\phi_h^n)\|^2+\frac{1}{2}\|\hat{\phi}_h^{n+1}-\alpha\phi_h^n\|^2\right)\\
				\leq &~2\|\phi_h^n\|^2+\|\nabla\phi_h^n\|^2+\|(\phi_h^n)^2-1\|^2\leq C,
			\end{aligned}
		\end{equation}
		where $C$ is a positive contant independent of $\left(\hat{\phi}_h^{n+1},\hat{\mu}_h^{n+1}\right)$ and $\alpha$.
		
		Setting $\left(w_h,\varphi_h\right)=\left(2\tau\hat{\phi}_h^{n+1},2\tau\hat{\mu}_h^{n+1}\right)$ in \eqref{eq_ch_hat_alpha0011} and adding the obtained equalities, we get
		\begin{equation}
			\label{eq_ch_hat_phi_mu_0001}
			\begin{aligned}
				&\|\hat{\phi}_h^{n+1}\|^2-\alpha^2\|\phi_h^n\|^2+\|\hat{\phi}_h^{n+1}-\alpha\phi_h^n\|^2+2\tau\|\nabla\hat{\mu}_h^{n+1}\|^2\\
				=&~2\tau\alpha\left(\hat{\phi}_h^{n+1}\mathbf{u}_h^n,\nabla\hat{\phi}_h^{n+1}\right)+2\alpha\tau\left((\hat{\phi}_h^{n+1})^3-\phi_h^n,\hat{\mu}_h^{n+1}\right).
			\end{aligned}
		\end{equation}
		Thus, according to the following inequality 
		$$
		(\hat{\phi}_h^{n+1})^4=\left((\hat{\phi}_h^{n+1})^2-1\right)^2+2\left(\hat{\phi}_h^{n+1}-\alpha\phi_h^n\right)^2+4\alpha\left(\hat{\phi}_h^{n+1}-\alpha\phi_h^n\right)\phi_h^n-1+2\alpha^2(\phi_h^n)^2,
		$$
		we have
		\begin{equation}
			\label{eq_ch_hat_phi_mu_boundedness_0001right001}
			\begin{aligned}
				2\tau\alpha\left(\hat{\phi}_h^{n+1}\mathbf{u}_h^n,\nabla\hat{\phi}_h^{n+1}\right)
				\leq 2\tau\alpha\|\hat{\phi}_h^{n+1}\|_{L^4}\|\mathbf{u}_h^n\|_{L^4}\|\nabla\hat{\phi}_h^{n+1}\|,
			\end{aligned}
		\end{equation}
		and
		\begin{equation}
			\label{eq_ch_hat_phi_mu_boundedness_0001right002}
			\begin{aligned}
				&2\alpha\tau\left((\hat{\phi}_h^{n+1})^3-\phi_h^n,\hat{\mu}_h^{n+1}\right)
				\leq 2\tau\alpha\left(\|\hat{\phi}_h^{n+1}\|_{L^4}^3\|\hat{\mu}_h^{n+1}\|_{L^4}+\|\phi_h^n\|\|\hat{\mu}_h^{n+1}\|\right)\\
				\leq &~\tau\left(\|\hat{\mu}_h^{n+1}\|^2+\|\nabla\hat{\mu}_h^{n+1}\|^2\right)+C_0\alpha\tau\left(\|(\hat{\phi}_h^{n+1})^2-1\|^2+\|\hat{\phi}_h^{n+1}-\alpha\phi_h^n\|^2\right)^{\frac{3}{2}}\\
				&~+C_0\tau\alpha\left(\alpha^3\|\phi_h^n\|^3+\|\phi_h^n\|^2+C\right).
			\end{aligned}
		\end{equation}
		Combining the \eqref{eq_ch_hat_phi_mu_boundedness_0001right001},\eqref{eq_ch_hat_phi_mu_boundedness_0001right002} with \eqref{eq_ch_hat_phi_mu_0001}, we deduce the \eqref{eq_H1_phi_mu_boundedness}.
	\end{proof}	
	\subsection{Proof of Lemma \ref{lemma_unique_solvability}}
	\label{subsection_appendix_proof_lemma0303}
	\begin{proof}
		To prove the existence of a unique solution for the fully discrete scheme \eqref{eq_fully_discrete_CHNS_scheme_phi}-\eqref{eq_fully_discrete_CHNS_scheme_tilde_incompressible_condition}, we suppose the $\left(\phi_{h1}^{n+1},\mu_{h1}^{n+1}\right)$ and $\left(\phi_{h2}^{n+1},\mu_{h2}^{n+1}\right)$ are two solutions of the  fully discrete scheme \eqref{eq_fully_discrete_CHNS_scheme_phi}-\eqref{eq_fully_discrete_CHNS_scheme_mu},
		and denote
		$$
		\bar{\phi}_h^{n+1}=\phi_{h1}^{n+1}-\phi_{h2}^{n+1},\bar{\mu}_h^{n+1}=\mu_{h1}^{n+1}-\mu_{h2}^{n+1}.
		$$
		By according to \eqref{eq_fully_discrete_CHNS_scheme_phi}-\eqref{eq_fully_discrete_CHNS_scheme_mu}, we obtain
		\begin{subequations}
			\label{eq_bar_phi_mu00001}
			\begin{flalign}
				\label{eq_bar_phi_mu00001_phi}
				\left(\frac{\bar{\phi}_h^{n+1}}{\tau},w_h\right)-\left(\bar{\phi}_h^{n+1}\mathbf{u}_h^n,\nabla w_h\right)+\left(\nabla\bar{\mu}_h^{n+1},\nabla w_h\right)&~=0,\\
				\label{eq_bar_phi_mu00001_mu}
				\left(\bar{\mu}_h^{n+1},\varphi_h\right)-\left(\nabla\bar{\phi}_h^{n+1},\nabla\varphi_h\right)-\left(\bar{\phi}_h^{n+1}\left((\phi_{h1}^{n+1})^2+\phi_{h1}^{n+1}\phi_{h2}^{n+1}+(\phi_{h2}^{n+1})^2\right),\varphi_h\right)&~=0.
			\end{flalign}
		\end{subequations}
		Letting $\left(w_h,\varphi_h\right)=\left(\tau\bar{\mu}_h^{n+1},\bar{\phi}_h^{n+1}\right)$
		in \eqref{eq_bar_phi_mu00001} and combining the obtained results, we have
		\begin{equation}
			\label{eq_takeinnerproduct_muphibar_combine0001}
			\begin{aligned}
				\tau\|\nabla\bar{\mu}_h^{n+1}\|^2+&~\|\nabla\bar{\phi}_h^{n+1}\|^2+\left((\phi_{h1}^{n+1})^2+\phi_{h1}^{n+1}\phi_{h2}^{n+1}+(\phi_{h2}^{n+1})^2,(\bar{\phi}_h^{n+1})^2\right)\\
				=&~\tau\left(\bar{\phi}_h^{n+1}\mathbf{u}_h^n,\nabla \bar{\mu}_h^{n+1}\right).
			\end{aligned}
		\end{equation}
		It's easy to see that
		\begin{equation}
			\begin{aligned}
				&\left((\phi_{h1}^{n+1})^2+\phi_{h1}^{n+1}\phi_{h2}^{n+1}+(\phi_{h2}^{n+1})^2,(\bar{\phi}_h^{n+1})^2\right)\\
				&\qquad=\left(\frac{(\phi_{h1}^{n+1})^2+(\phi_{h2}^{n+1})^2}{2}+\frac{(\phi_{h1}^{n+1}+\phi_{h2}^{n+1})^2}{2},(\bar{\phi}_h^{n+1})^2\right)\geq 0.
			\end{aligned}
		\end{equation}
		According to $\tau\leq \frac{C}{\|\mathbf{u}_h^n\|_{L^{\infty}}^2}$, we obtain
		\begin{equation}
			\label{eq_boundedness_barPhi_u_mu0001}
			\begin{aligned}
				\tau\left(\bar{\phi}_h^{n+1}\mathbf{u}_h^n,\nabla \bar{\mu}_h^{n+1}\right)\leq&~ \|\bar{\phi}_h^{n+1}\|\|\mathbf{u}_h^n\|_{L^{\infty}}\|\nabla\bar{\mu}_h^{n+1}\|\\
				\leq&~C\tau\|\mathbf{u}_h^n\|_{L^{\infty}}^2\|\bar{\phi}_h^{n+1}\|^2+\frac{\tau}{2}\|\nabla\bar{\mu}_h^{n+1}\|^2\\
				\leq&~\frac{1}{2}\|\nabla\bar{\phi}_h^{n+1}\|^2+\frac{\tau}{2}\|\nabla\bar{\mu}_h^{n+1}\|^2.
			\end{aligned}
		\end{equation}
		It follows from \eqref{eq_takeinnerproduct_muphibar_combine0001} that
		\begin{equation}
			\|\nabla\bar{\mu}_h^{n+1}\|^2+\|\nabla\bar{\phi}_h^{n+1}\|^2=0.
		\end{equation}
		Therefore, the discrete scheme \eqref{eq_fully_discrete_CHNS_scheme_phi}-\eqref{eq_fully_discrete_CHNS_scheme_boundary_conditions} has a unique solution. 
		The uniqueness of the solutions $\left(\mathbf{u}_h^{n+1},p_h^{n+1}\right)$ for the discrete scheme \eqref{eq_fully_discrete_CHNS_scheme_tilde_u}-\eqref{eq_fully_discrete_CHNS_scheme_tilde_incompressible_condition} is a standard result. Hence, the existence and uniqueness of the solutions for the entire fully discrete scheme \eqref{eq_fully_discrete_CHNS_scheme_phi}-\eqref{eq_fully_discrete_CHNS_scheme_tilde_incompressible_condition} are proved.
	\end{proof}
	\bibliographystyle{unsrt}
	\bibliography{./../document/bibfile}

\begin{thebibliography}{10}

\bibitem{Temam_Roger_Navier_Stokes_equations_Theory_and_numerical_analysis}
Roger Temam.
\newblock {\em Navier-{S}tokes equations: {T}heory and numerical analysis},
  volume Vol. 2 of {\em Studies in {M}athematics and its {A}pplications}.
\newblock North-Holland Publishing Co., Amsterdam-New York-Oxford, 1977.

\bibitem{1958_Cahn_Hilliard_Free_Energy_of_a_Nonuniform_System_I_Interfacial_Free_Energy}
John Cahn and John Hilliard.
\newblock Free energy of a nonuniform system. {I}. {I}nterfacial {F}ree
  {E}nergy.
\newblock {\em J. Chem. Phys.}, 28:258–267, 1958.

\bibitem{2007_Kay_David_and_Welford_Richard_Efficient_numerical_solution_of_Cahn_Hilliard_Navier_Stokes_fluids_in_2D}
David Kay and Richard Welford.
\newblock Efficient numerical solution of {C}ahn-{H}illiard-{N}avier-{S}tokes
  fluids in 2{D}.
\newblock {\em SIAM J. Sci. Comput.}, 29(6):2241--2257, 2007.

\bibitem{2008_Kay_David_and_Welford_Richard_Finite_element_approximation_of_a_Cahn_Hilliard_Navier_Stokes_system}
David Kay, Vanessa Styles, and Richard Welford.
\newblock Finite element approximation of a {C}ahn-{H}illiard-{N}avier-{S}tokes
  system.
\newblock {\em Interfaces Free Bound.}, 10(1):15--43, 2008.

\bibitem{1998_Eyre_David_J_Unconditionally_gradient_stable_time_marching_the_Cahn_Hilliard_equation}
David Eyre.
\newblock Unconditionally gradient stable time marching the {C}ahn-{H}illiard
  equation.
\newblock In {\em Computational and mathematical models of microstructural
  evolution ({S}an {F}rancisco, {CA}, 1998)}, volume 529 of {\em Mater. Res.
  Soc. Sympos. Proc.}, pages 39--46. MRS, Warrendale, PA, 1998.

\bibitem{2011_ShenJie_Energy_stable_schemes_for_Cahn_Hilliard_phase_field_model_of_two_phase_incompressible_flows}
Jie Shen and Xiaofeng Yang.
\newblock Energy stable schemes for {C}ahn-{H}illiard phase-field model of
  two-phase incompressible flows.
\newblock {\em Chinese Ann. Math. Ser. B}, 31(5):743--758, 2010.

\bibitem{2017_YangXiaofeng_Numerical_approximations_for_the_molecular_beam_epitaxial_growth_model_based_on_the_invariant_energy_quadratization_method}
Xiaofeng Yang, Jia Zhao, and Qi~Wang.
\newblock Numerical approximations for the molecular beam epitaxial growth
  model based on the invariant energy quadratization method.
\newblock {\em J. Comput. Phys.}, 333:104--127, 2017.

\bibitem{2018_Shenjie_Xujie_SAV}
Jie Shen, Jie Xu, and Jiang Yang.
\newblock The scalar auxiliary variable ({SAV}) approach for gradient flows.
\newblock {\em J. Comput. Phys.}, 353:407--416, 2018.

\bibitem{2018_Shenjie_Xujie_CAEAFTSAVSYGF}
Jie Shen and Jie Xu.
\newblock Convergence and error analysis for the scalar auxiliary variable
  ({SAV}) schemes to gradient flows.
\newblock {\em SIAM J. Numer. Anal.}, 56(5):2895--2912, 2018.

\bibitem{2006_Guermond_ShenJie_An_overview_of_projection_methods_for_incompressible_flows}
Jean~Luc Guermond, Peter Minev, and Jie Shen.
\newblock An overview of projection methods for incompressible flows.
\newblock {\em Comput. Methods Appl. Mech. Engrg.}, 195(44-47):6011--6045,
  2006.

\bibitem{1968_Chorin_Alexandre_Numerical_solution_of_the_Navier_Stokes_equations}
Alexandre~Joel Chorin.
\newblock Numerical solution of the {N}avier-{S}tokes equations.
\newblock {\em Math. Comp.}, 22:745--762, 1968.

\bibitem{2017_Diegel_Convergence_analysis_and_error_estimates_for_a_second_order_accurate_finite_element_method_for_the_Cahn_Hilliard_Navier_Stokes_system}
Amanda Diegel, Cheng Wang, Xiaoming Wang, and Steven Wise.
\newblock Convergence analysis and error estimates for a second order accurate
  finite element method for the {C}ahn-{H}illiard-{N}avier-{S}tokes system.
\newblock {\em Numer. Math.}, 137(3):495--534, 2017.

\bibitem{2024_Wangcheng_Convergence_analysis_of_a_temporally_second_order_accurate_finite_element_scheme_for_the_Cahn_Hilliard_magnetohydrodynamics_system_of_equations}
Cheng Wang, Jilu Wang, Steven~M. Wise, Zeyu Xia, and Liwei Xu.
\newblock Convergence analysis of a temporally second-order accurate finite
  element scheme for the {C}ahn-{H}illiard-magnetohydrodynamics system of
  equations.
\newblock {\em J. Comput. Appl. Math.}, 436:Paper No. 115409, 17, 2024.

\bibitem{2015_HanDaozhi_A_second_order_in_time_uniquely_solvable_unconditionally_stable_numerical_scheme_for_Cahn_Hilliard_Navier_Stokes_equation}
Daozhi Han and Xiaoming Wang.
\newblock A second order in time, uniquely solvable, unconditionally stable
  numerical scheme for {C}ahn-{H}illiard-{N}avier-{S}tokes equation.
\newblock {\em J. Comput. Phys.}, 290:139--156, 2015.

\bibitem{2015_ShenJie_Decoupled_energy_stable_schemes_for_phase_field_models_of_two_phase_incompressible_flows}
Jie Shen and Xiaofeng Yang.
\newblock Decoupled, energy stable schemes for phase-field models of two-phase
  incompressible flows.
\newblock {\em SIAM J. Numer. Anal.}, 53(1):279--296, 2015.

\bibitem{2018_CaiYongyong_Error_estimates_for_a_fully_discretized_scheme_to_a_Cahn_Hilliard_phase_field_model_for_two_phase_incompressible_flows}
Yongyong Cai and Jie Shen.
\newblock Error estimates for a fully discretized scheme to a {C}ahn-{H}illiard
  phase-field model for two-phase incompressible flows.
\newblock {\em Math. Comp.}, 87(313):2057--2090, 2018.

\bibitem{2018_GaoYali_Decoupled_linear_and_energy_stable_finite_element_method_for_the_Cahn_Hilliard_Navier_Stokes_Darcy_phase_field_model}
Yali Gao, Xiaoming He, Liquan Mei, and Xiaofeng Yang.
\newblock Decoupled, linear, and energy stable finite element method for the
  {C}ahn-{H}illiard-{N}avier-{S}tokes-{D}arcy phase field model.
\newblock {\em SIAM J. Sci. Comput.}, 40(1):B110--B137, 2018.

\bibitem{2019_LinLianlei_Numerical_approximation_of_incompressible_Navier_Stokes_equations_based_on_an_auxiliary_energy_variable}
Lianlei Lin, Zhiguo Yang, and Suchuan Dong.
\newblock Numerical approximation of incompressible {N}avier-{S}tokes equations
  based on an auxiliary energy variable.
\newblock {\em J. Comput. Phys.}, 388:1--22, 2019.

\bibitem{2020_JiaHongen_A_novel_linear_unconditional_energy_stable_scheme_for_the_incompressible_Cahn_Hilliard_Navier_Stokes_phase_field_model}
Hongen Jia, Xue Wang, and Kaitai Li.
\newblock A novel linear, unconditional energy stable scheme for the
  incompressible {C}ahn-{H}illiard-{N}avier-{S}tokes phase-field model.
\newblock {\em Comput. Math. Appl.}, 80(12):2948--2971, 2020.

\bibitem{2020_YangXiaofeng_Error_Analysis_of_a_Decoupled_Linear_Stabilization_Scheme_for_the_Cahn_Hilliard_Model_of_Two_Phase_Incompressible_Flows}
Zhen Xu, Xiaofeng Yang, and Hui Zhang.
\newblock Error analysis of a decoupled, linear stabilization scheme for the
  {C}ahn-{H}illiard model of two-phase incompressible flows.
\newblock {\em J. Sci. Comput.}, 83(3):Paper No. 57, 27, 2020.

\bibitem{2021_Yang_Xiaofeng_Decoupled_linear_and_unconditionally_energy_stable_fully_discrete_finite_element_numerical_scheme_for_a_two_phase_ferrohydrodynamics_model}
Guodong Zhang, Xiaoming He, and Xiaofeng Yang.
\newblock Decoupled, linear, and unconditionally energy stable fully discrete
  finite element numerical scheme for a two-phase ferrohydrodynamics model.
\newblock {\em SIAM J. Sci. Comput.}, 43(1):B167--B193, 2021.

\bibitem{2022JieShen_LiXiaoliMSAVCHNStwo_phase_incompressible_flows}
Xiaoli Li and Jie Shen.
\newblock On fully decoupled {MSAV} schemes for the
  {C}ahn-{H}illiard-{N}avier-{S}tokes model of two-phase incompressible flows.
\newblock {\em Math. Models Methods Appl. Sci.}, 32(3):457--495, 2022.

\bibitem{2024_ChenWenbin_Convergence_analysis_of_a_second_order_numerical_scheme_for_the_Flory_Huggins_Cahn_Hilliard_Navier_Stokes_system}
Wenbin Chen, Jianyu Jing, Qianqian Liu, Cheng Wang, and Xiaoming Wang.
\newblock Convergence analysis of a second order numerical scheme for the
  {F}lory-{H}uggins-{C}ahn-{H}illiard-{N}avier-{S}tokes system.
\newblock {\em J. Comput. Appl. Math.}, 450:Paper No. 115981, 22, 2024.

\bibitem{2025_GaoHj_Stability_and_error_analysis_of_SAV_semi_discrete_scheme_for_Cahn_Hilliard_Navier_Stokes_model}
Haijun Gao, Xi~Li, and Minfu Feng.
\newblock Stability and error analysis of {SAV} semi-discrete scheme for
  {C}ahn-{H}illiard-{N}avier-{S}tokes model.
\newblock {\em Numer. Math. Theory Methods Appl.}, 18(1):66--102, 2025.

\bibitem{2022_Wenjuan_Semi_implicit_unconditionally_energy_stable_stabilized_finite_element_method_based_on_multiscale_enrichment_for_the_Cahn_Hilliard_Navier_Stokes_phase_field_model}
Juan Wen, Yinnian He, and Yaling He.
\newblock Semi-implicit, unconditionally energy stable, stabilized finite
  element method based on multiscale enrichment for the
  cahn-hilliard-navier-stokes phase-field model.
\newblock {\em Computers and Mathematics with Applications}, 126:172--181,
  2022.

\bibitem{2023_CaiWentao_Optimal_L2_error_estimates_of_unconditionally_stable_finite_element_schemes_for_the_Cahn_Hilliard_Navier_Stokes_system}
Wentao Cai, Weiwei Sun, Jilu Wang, and Zongze Yang.
\newblock Optimal {$L^2$} error estimates of unconditionally stable finite
  element schemes for the {C}ahn-{H}illiard-{N}avier-{S}tokes system.
\newblock {\em SIAM J. Numer. Anal.}, 61(3):1218--1245, 2023.

\bibitem{2024_YiNianyu_Convergence_analysis_of_a_decoupled_pressure_correction_SAV_FEM_for_the_Cahn_Hilliard_Navier_Stokes_model}
Jinting Yang and Nianyu Yi.
\newblock Convergence analysis of a decoupled pressure-correction {SAV}-{FEM}
  for the {C}ahn--{H}illiard--{N}avier--{S}tokes model.
\newblock {\em J. Comput. Appl. Math.}, 449:Paper No. 115985, 2024.

\bibitem{2025_GaoHaijun_A_new_decoupled_unconditionally_stable_scheme_and_its_optimal_error_analysis_for_the_Cahn_Hilliard_Navier_Stokes_equations}
Haijun Gao, Xi~Li, and Minfu Feng.
\newblock A new decoupled unconditionally stable scheme and its optimal error
  analysis for the {C}ahn-{H}illiard-{N}avier-{S}tokes equations.
\newblock {\em Comput. Math. Appl.}, 194:53--85, 2025.

\bibitem{1984_Brezzi_FortinA_stable_finite_element_for_the_Stokes_equations}
Douglas Arnold, Franco Brezzi, and Michel Fortin.
\newblock A stable finite element for the {S}tokes equations.
\newblock {\em Calcolo}, 21(4):337--344, 1984.

\bibitem{1973_Wheeler_Mary_Fanett_A_priori_L2_error_estimates_for_Galerkin_approximations_to_parabolic_partial_differential_equations}
Mary~Fanett Wheeler.
\newblock A priori {$L\sb{2}$} error estimates for {G}alerkin approximations to
  parabolic partial differential equations.
\newblock {\em SIAM J. Numer. Anal.}, 10:723--759, 1973.

\bibitem{2008_Brenner_Susanne_C_The_mathematical_theory_of_finite_element_methods}
Susanne Brenner and Ridgway Scott.
\newblock {\em The mathematical theory of finite element methods}, volume~15 of
  {\em Texts in Applied Mathematics}.
\newblock Springer, New York, third edition, 2008.

\bibitem{2008_Hou_An_efficient_semi_implicit_immersed_boundary_method_for_the_Navier_Stokes_equations}
Thomas~Yizhao Hou and Zuoqiang Shi.
\newblock An efficient semi-implicit immersed boundary method for the
  {N}avier-{S}tokes equations.
\newblock {\em J. Comput. Phys.}, 227(20):8968--8991, 2008.

\bibitem{2015_HeYinnian_Unconditional_convergence_of_the_Euler_semi_implicit_scheme_for_the_three_dimensional_incompressible_MHD_equations}
Yinnian He.
\newblock Unconditional convergence of the {E}uler semi-implicit scheme for the
  three-dimensional incompressible {MHD} equations.
\newblock {\em IMA J. Numer. Anal.}, 35(2):767--801, 2015.

\bibitem{1986_Girault_Vivette_Finite_element_methods_for_Navier_Stokes_equations}
Vivette Girault and Pierre~Arnaud Raviart.
\newblock {\em Finite element methods for {N}avier-{S}tokes equations},
  volume~5 of {\em Springer Series in Computational Mathematics}.
\newblock Springer-Verlag, Berlin, 1986.
\newblock Theory and algorithms.

\bibitem{2006_Thomee_Galerkin_finite_element_methods_for_parabolic_problems}
Vidar Thom\'ee.
\newblock {\em Galerkin finite element methods for parabolic problems},
  volume~25 of {\em Springer Series in Computational Mathematics}.
\newblock Springer-Verlag, Berlin, second edition, 2006.

\bibitem{2019_YangXiaofeng_Convergence_analysis_of_an_unconditionally_energy_stable_projection_scheme_for_magneto_hydrodynamic_equations}
Xiaofeng Yang, Guodong Zhang, and Xiaoming He.
\newblock Convergence analysis of an unconditionally energy stable projection
  scheme for magneto-hydrodynamic equations.
\newblock {\em Appl. Numer. Math.}, 136:235--256, 2019.

\bibitem{2020_Chen_Hongtao_Optimal_error_estimates_for_the_scalar_auxiliary_variable_finite_element_schemes_for_gradient_flows}
Hongtao Chen, Jingjing Mao, and Jie Shen.
\newblock Optimal error estimates for the scalar auxiliary variable
  finite-element schemes for gradient flows.
\newblock {\em Numer. Math.}, 145(1):167--196, 2020.

\bibitem{2006_FengXiaobing_FullydiscretefiniteelementapproximationsoftheNavierStokesCahnHilliarddiffuseinterfacemodelfortwophasefluidflows}
Xiaobing Feng.
\newblock Fully discrete finite element approximations of the
  {N}avier-{S}tokes-{C}ahn-{H}illiard diffuse interface model for two-phase
  fluid flows.
\newblock {\em SIAM J. Numer. Anal.}, 44(3):1049--1072, 2006.

\bibitem{2015_Diegel_Analysis_of_a_mixed_finite_element_method_for_a_Cahn_Hilliard_Darcy_Stokes_system}
Amanda Diegel, Xiaobing Feng, and Steven Wise.
\newblock Analysis of a mixed finite element method for a
  {C}ahn-{H}illiard-{D}arcy-{S}tokes system.
\newblock {\em SIAM J. Numer. Anal.}, 53(1):127--152, 2015.

\bibitem{2007_HeYinnian_SunWeiwei_Stability_and_convergence_of_the_Crank_Nicolson_Adams_Bashforth_scheme_for_the_time_dependent_Navier_Stokes_equations}
Yinnian He and Weiwei Sun.
\newblock Stability and convergence of the
  {C}rank-{N}icolson/{A}dams-{B}ashforth scheme for the time-dependent
  {N}avier-{S}tokes equations.
\newblock {\em SIAM J. Numer. Anal.}, 45(2):837--869, 2007.

\bibitem{1990_Heywood_Finite_element_approximation_of_the_nonstationary_Navier_Stokes_problem_IV_Error_analysis_for_second_order_time_discretization}
John~G Heywood and Rolf Rannacher.
\newblock Finite-element approximation of the nonstationary {N}avier-{S}tokes
  problem. {IV}. {E}rror analysis for second-order time discretization.
\newblock {\em SIAM J. Numer. Anal.}, 27(2):353--384, 1990.

\bibitem{2019_HeXiaoming_A_diffuse_interface_model_and_semi_implicit_energy_stable_finite_element_method_for_two_phase_magnetohydrodynamic_flows}
Jinjin Yang, Shipeng Mao, Xiaoming He, Xiaofeng Yang, and Yinnian He.
\newblock A diffuse interface model and semi-implicit energy stable finite
  element method for two-phase magnetohydrodynamic flows.
\newblock {\em Comput. Methods Appl. Mech. Engrg.}, 356:435--464, 2019.

\bibitem{2022_ChenYaoyao_CHNS_2022_AMC}
Yaoyao Chen, Yunqing Huang, and Nianyu Yi.
\newblock Error analysis of a decoupled, linear and stable finite element
  method for {C}ahn-{H}illiard-{N}avier-{S}tokes equations.
\newblock {\em Appl. Math. Comput.}, 421:Paper No. 126928, 17, 2022.

\bibitem{2002_Boyer_A_theoretical_and_numerical_model_for_the_study_of_incompressible_mixture_flows}
Franck Boyer.
\newblock A theoretical and numerical model for the study of incompressible
  mixture flows.
\newblock {\em Computers \& Fluids}, 31(1):41--68, 2002.

\bibitem{1991_Brezzi_Fortin_Mixed_and_hybrid_finite_element_methods}
Franco Brezzi and Michel Fortin.
\newblock {\em Mixed and hybrid finite element methods}, volume~15 of {\em
  Springer Series in Computational Mathematics}.
\newblock Springer-Verlag, New York, 1991.

\end{thebibliography}
\end{document}